\documentclass{article}

\usepackage[preprint, nonatbib]{neurips_2019}

\usepackage{amssymb, amsmath, amsthm, latexsym}
\usepackage[dvipsnames]{xcolor}
\usepackage{url}
\usepackage{algorithm}
\usepackage{algpseudocode}
\usepackage{tabularx}
\usepackage{paralist}
\usepackage{mathtools}
\usepackage{tcolorbox}
\usepackage{makecell}
\usepackage{enumitem}

\usepackage{pifont}
\newcommand{\cmark}{{\color{PineGreen}\ding{51}}}%
\newcommand{\xmark}{{\color{BrickRed}\ding{55}}}%

\usepackage{pifont}
\usepackage{subfigure} 
\newcommand{\Exp}{\mathbf{E}}
\newcommand{\Prob}{\mathbf{P}}
\newcommand{\R}{\mathbb{R}}

\newcommand{\eqdef}{\coloneqq}

\def\<#1,#2>{\left\langle #1,#2\right\rangle}

\newcommand{\compactify}{} 

\theoremstyle{definition}
\newtheorem{lemma}{Lemma}[section]
\newtheorem{theorem}{Theorem}[section]
\newtheorem{definition}{Definition}[section]

\newtheorem{assumption}{Assumption}[section]
\newtheorem{corollary}{Corollary}[section]
\newtheorem{example}{Example}[section]

\newtheorem{remark}[theorem]{Remark}

\newcommand{\argmin}{\mathop{\arg\!\min}}

\usepackage[colorinlistoftodos,bordercolor=orange,backgroundcolor=orange!20,linecolor=orange,textsize=scriptsize]{todonotes}

\newcommand{\cD}{{\cal D}}

\newcommand{\cL}{{\cal L}}

\newcommand{\cO}{{\cal O}}

\newcommand{\mA}{{\bf A}}

\newcommand{\mI}{{\bf I}}
\newcommand{\mJ}{{\bf J}}

\newcommand{\mS}{{\bf S}}

\newcommand{\mW}{{\bf W}}
\newcommand{\mX}{{\bf X}}

\newcommand{\EE}{\mathbf{E}}

\newcommand{\Jac}{{ \bf \nabla F}}
\newcommand{\Proj}{{\bf \Pi}}
\newcommand{\norm}[1]{\left \| #1 \right\|}
\newcommand{\ED}[1]{\mathbf{E}_{\cD}\left[#1\right] }
\newcommand{\Tr}[1]{\mbox{Tr}\left( #1\right)}
\newcommand{\ones}{e}

\newcommand{\prox}{\mathop{\mathrm{prox}}\nolimits}
\newcommand{\proxR}{\prox_{\gamma R}}

\usepackage{accents}
\newlength{\dhatheight}

    \usepackage{hyperref}

\usepackage[utf8]{inputenc} 
\usepackage[T1]{fontenc}    
\usepackage{hyperref}       
\usepackage{url}            
\usepackage{booktabs}       
\usepackage{amsfonts}       
\usepackage{nicefrac}       
\usepackage{microtype}      

\title{A Unified Theory of SGD: Variance Reduction, Sampling, Quantization  and Coordinate Descent}

\author{%
  Eduard Gorbunov \\
  MIPT, Russia \\
  \texttt{eduard.gorbunov@phystech.edu} \\
  \And
  Filip Hanzely \\
  KAUST, Saudi Arabia\\
  \texttt{filip.hanzely@kaust.edu.sa} \\  
  \And
  Peter Richt\'arik \\
  KAUST, Saudi Arabia and MIPT, Russia\\
  \texttt{peter.richtarik@kaust.edu.sa} \\
}

\begin{document}

\maketitle

\begin{abstract}
In this paper we introduce a unified analysis of a large family of variants of proximal stochastic gradient descent ({\tt SGD}) which so far have required different intuitions, convergence analyses, have different applications, and which have been developed separately in various communities. We show that our framework includes methods with and without the following tricks, and their combinations: variance reduction, importance sampling, mini-batch sampling, quantization, and coordinate sub-sampling.  As a by-product, we obtain the first unified theory of {\tt SGD} and randomized coordinate descent ({\tt RCD}) methods,  the first unified theory of variance reduced and non-variance-reduced {\tt SGD} methods, and the first unified theory of quantized and non-quantized methods. A key to our approach is a parametric assumption on the iterates and stochastic gradients. In a single theorem we establish a linear convergence result under this assumption and strong-quasi convexity of the loss function. Whenever we recover an existing method as a special case, our theorem gives the best known complexity result.
Our approach can be used to motivate the development of new useful methods, and offers pre-proved convergence guarantees. To illustrate the strength of our approach, we develop five new variants of {\tt SGD}, and through numerical experiments demonstrate some of their properties.  


\end{abstract}

\section{Introduction}

In this paper we are interested in  the optimization  problem
    \begin{equation}\label{eq:problem_gen}
        \min_{x\in\R^d} f(x) + R(x),
    \end{equation}
    where $f$ is convex, differentiable with Lipschitz gradient, and $R:\R^d\to \R\cup \{+\infty\}$ is a proximable (proper closed convex) regularizer. In particular, we  focus on situations when  it is prohibitively expensive to  compute the gradient of $f$, while an unbiased estimator of the gradient can be computed efficiently. This is typically the case for stochastic optimization problems, i.e., when \begin{equation} \label{eq:f_exp}
    f(x)=\EE_{\xi \sim \cD} \left[ f_\xi(x)\right] ,\end{equation} where $\xi$ is a random variable, and  $f_\xi:\R^d\to \R$  is smooth for all $\xi$. Stochastic optimization problems are of key importance in statistical supervised learning theory. In this setup, $x$ represents a machine learning model described by $d$ parameters (e.g., logistic regression or a deep neural network), $\cD$  is an unknown distribution of labelled examples,  $f_\xi(x)$ represents the loss of model $x$ on datapoint $\xi$, and $f$ is the generalization error. Problem \eqref{eq:problem_gen} seeks to find the model $x$ minimizing the generalization error. In statistical learning theory one assumes that while $\cD$ is not known, samples $\xi\sim \cD$ are available. In such a case, $\nabla f(x)$ is not computable, while $\nabla f_\xi(x)$, which is  an unbiased estimator of the gradient of $f$ at $x$, is easily computable. 

Another prominent example, one of special interest in this paper, are functions $f$ which arise as averages of a very large number of smooth functions: \begin{equation}\label{eq:f_sum} 
\compactify f(x)=\frac{1}{n}\sum \limits_{i=1}^n f_i(x).\end{equation} This problem often arises by approximation of the stochastic optimization loss function \eqref{eq:f_exp} via Monte Carlo integration, and is in this context known as the empirical risk minimization (ERM) problem.  ERM is currently the dominant paradigm for solving supervised learning problems \cite{shai_book}.  If index $i$ is chosen uniformly at random from $[n]\eqdef \{1,2,\dots,n\}$, $\nabla f_i(x)$ is an unbiased estimator of $\nabla f(x)$. Typically, $\nabla f(x)$ is about $n$ times more expensive to compute than $\nabla f_i(x)$. 

Lastly, in some applications, especially in distributed training of supervised models, one considers problem \eqref{eq:f_sum}, with $n$ being the number of machines, and each $f_i$ also having a finite sum structure, i.e.,
\begin{equation}\label{eq:f_i_sum}
\compactify    f_i(x) = \frac{1}{m}\sum\limits_{j=1}^m  f_{ij}(x),
\end{equation}
where $m$ corresponds to the number of training examples stored on machine $i$.

\section{The Many Faces of Stochastic Gradient Descent} \label{sec:many_faces_of_SGD}

 Stochastic gradient descent ({\tt SGD}) \cite{RobbinsMonro:1951, Nemirovski-Juditsky-Lan-Shapiro-2009, Vaswani2019-overparam} is a state-of-the-art algorithmic paradigm for solving optimization problems \eqref{eq:problem_gen} in situations when $f$ is either of structure 
 \eqref{eq:f_exp} or \eqref{eq:f_sum}. In its generic form, (proximal) {\tt SGD} defines the new iterate by subtracting a multiple of a stochastic gradient from the current iterate, and subsequently applying the proximal operator of $R$:
\begin{equation} \label{eq:SGD} x^{k+1} = \proxR(x^k - \gamma g^k).\end{equation} 
Here, $g^k$ is an unbiased estimator of the gradient (i.e., a stochastic gradient), \begin{equation}\label{eq:stoch_grad}\EE \left[g^k \;|\; x^k\right] = \nabla f(x^k),\end{equation}
and $\proxR(x)\eqdef \argmin_u \{\gamma R(x) + \frac{1}{2}\norm{u-x}^2\}$. However, and this is the starting point of our journey in this paper, there are {\em infinitely many} ways of obtaining a random vector $g^k$ satisfying \eqref{eq:stoch_grad}. On the one hand, this gives algorithm designers the flexibility to {\em construct}  stochastic gradients in various ways in order to target desirable properties such as convergence speed, iteration cost, parallelizability and generalization.  On the other hand, this poses considerable challenges in terms of convergence analysis. Indeed, if one aims to, as one should, obtain the sharpest bounds possible, dedicated analyses are needed to handle each of the particular variants of {\tt SGD}.

\textbf{Vanilla\footnote{In this paper, by {\em vanilla} {\tt SGD} we refer to {\tt SGD} variants with or without importance sampling and mini-batching, but {\em excluding} variance-reduced variants, such as {\tt SAGA} \cite{SAGA} and {\tt SVRG} \cite{SVRG}.} {\tt SGD}.} The flexibility in the design of efficient strategies for constructing $g^k$ has led to a creative renaissance in the optimization and machine learning communities, yielding a large number of immensely powerful new variants of {\tt SGD}, such as those employing  {\em importance sampling} \cite{IProx-SDCA, NeedellWard2015}, and {\em mini-batching} \cite{mS2GD}. These efforts are subsumed by the recently developed and remarkably sharp analysis of {\tt SGD} under  {\em arbitrary sampling} paradigm \cite{SGD_AS}, first introduced in the study of randomized coordinate descent methods by \cite{NSync}. The arbitrary sampling paradigm covers virtually all stationary mini-batch and importance sampling strategies in a unified way, thus making headway towards theoretical unification of two separate strategies for constructing stochastic gradients. For strongly convex $f$, the {\tt SGD} methods analyzed in \cite{SGD_AS} converge linearly to a neighbourhood of the solution $x^* = \arg \min_x f(x)$ for a fixed stepsize $\gamma^k=\gamma$. The size of the neighbourhood is proportional to the second moment of the stochastic gradient at the optimum ($\sigma^2 \eqdef \tfrac{1}{n}\sum_{i=1}^n \norm{\nabla f_i(x^*)}^2$), to the stepsize ($\gamma$), and inversely proportional to the modulus of strong convexity. The effect of various sampling strategies, such as importance sampling and mini-batching, is twofold: i) improvement of the linear convergence rate by enabling larger stepsizes, and ii) modification of $\sigma^2$. However, none of these strategies\footnote{Except for the full batch strategy, which is prohibitively expensive.} is able to completely eliminate the adverse effect of $\sigma^2$. That is,  {\tt SGD} with a fixed stepsize does not reach the optimum, unless one happens to be in the overparameterized case characterized by the identity $\sigma^2=0$.

\textbf{Variance reduced {\tt SGD}.} While sampling strategies such as importance sampling and mini-batching reduce the variance of the stochastic gradient, in the finite-sum case \eqref{eq:f_sum} a new type of {\em variance reduction} strategies has been developed over the last few years \cite{SAG, SAGA, SVRG, SDCA, QUARTZ,nguyen2017sarah, Loopless}. These variance-reduced {\tt SGD} methods differ from the sampling strategies discussed before in a significant way: they can iteratively {\em learn} the stochastic gradients at the optimum,  and in so doing are able to eliminate the adverse effect of the gradient noise $\sigma^2>0$ which, as mentioned above, prevents the iterates of vanilla {\tt SGD} from converging to the optimum. As a result, for strongly convex $f$, these new variance-reduced {\tt SGD} methods converge linearly to $x^*$, with a fixed stepsize. At the moment, these variance-reduced variants require a markedly different convergence theory from the vanilla variants of {\tt SGD}. An exception to this is the situation when $\sigma^2=0$ as then variance reduction is not needed; indeed, vanilla {\tt SGD} already converges to the optimum, and with a fixed stepsize. We end the discussion here by remarking that this {\em hints} at a possible existence of a more unified theory, one that would include both vanilla and variance-reduced {\tt SGD}.

\textbf{Distributed {\tt SGD}, quantization and variance reduction.} When {\tt SGD} is implemented in a distributed fashion,  the problem is often expressed in the form \eqref{eq:f_sum}, where $n$ is the number of workers/nodes, and  $f_i$ corresponds to the loss based on data stored on node $i$. Depending on the number of data points stored on each node, it may or may not be efficient to compute the gradient of $f_i$ in each iteration. In general, {\tt SGD} is implemented in this way: each  node $i$ first computes a stochastic gradient $g_i^k$ of $f_i$ at the current point $x^k$ (maintained individually by each node). These gradients are then aggregated by a master node \cite{DANE, RDME}, in-network by a switch~\cite{switchML}, or a different technique best suited to the architecture used. To alleviate the communication bottleneck, various lossy update compression strategies such as quantization~\cite{1bit, Gupta:2015limited, zipml}, sparsification~\cite{RDME, alistarh2018convergence, tonko} and dithering~\cite{alistarh2017qsgd} were proposed. The basic idea is for each worker to apply a randomized transformation $Q:\R^d\to \R^d$ to $g_i^k$, resulting in a vector which is still an unbiased estimator of the gradient, but one that can be communicated with fewer bits. Mathematically, this amounts to injecting additional noise into the already noisy stochastic gradient $g_i^k$. The field of quantized {\tt SGD} is still young, and even some basic questions remained open until recently. For instance, there was no distributed quantized {\tt SGD} capable of provably solving  \eqref{eq:problem_gen} until the {\tt DIANA} algorithm \cite{mishchenko2019distributed} was introduced. {\tt DIANA} applies quantization to {\em gradient differences}, and in so doing is able to learn the gradients at the optimum, which makes is able to work for any regularizer $R$. {\tt DIANA} has some structural similarities with {\tt SEGA}~\cite{hanzely2018sega}---the first coordinate descent type method which works for non-separable regularizers---but a more precise relationship remains elusive.  When the functions of $f_i$ are of a finite-sum structure as in \eqref{eq:f_i_sum}, one can apply variance reduction to reduce the variance of the stochastic gradients $g_i^k$ together with quantization, resulting in the {\tt VR-DIANA} method~\cite{horvath2019stochastic}. This is the first distributed quantized {\tt SGD} method which provably converges to the solution of \eqref{eq:problem_gen}+\eqref{eq:f_i_sum} with a fixed stepsize.


\textbf{Randomized coordinate descent ({\tt RCD}).} Lastly, in a distinctly separate strain, there are {\tt SGD} methods for the coordinate/subspace descent variety~\cite{RCDM}. While it is possible to see  {\em some} {\tt RCD} methods as special cases of \eqref{eq:SGD}+\eqref{eq:stoch_grad}, most of them do not follow this algorithmic template. First, standard {\tt RCD} methods use different stepsizes for updating different coordinates~\cite{ALPHA}, and this seems to be crucial to their success. Second,  until the recent discovery of the {\tt SEGA} method, {\tt RCD} methods were not able to converge with non-separable regularizers. Third, {\tt RCD} methods are naturally variance-reduced in the $R=0$ case as partial derivatives at the optimum are all zero. As a consequence, attempts at creating variance-reduced {\tt RCD} methods seem to be futile. Lastly, {\tt RCD} methods are typically analyzed using different techniques. While there are deep links between standard {\tt SGD} and {\tt RCD} methods, these are often indirect and rely on duality \cite{SDCA, FACE-OFF, SDA}. 


\section{Contributions}

As outlined in the previous section, the world of {\tt SGD} is vast and beautiful. It is formed by many largely disconnected islands populated by elegant and efficient methods, with their own applications, intuitions, and convergence analysis techniques. While some links already exist (e.g., the unification of importance sampling and mini-batching variants under the arbitrary sampling umbrella), there is no comprehensive general theory. 
It is becoming increasingly difficult for the community to understand the relationships between these variants, both in theory and practice. New variants are yet to be discovered, but it is not clear what tangible principles one should adopt beyond intuition to aid the discovery. This situation is exacerbated by the fact that a number of different assumptions on the stochastic gradient, of various levels of strength, is being used in the literature.


The main contributions of this work include: 

$\bullet$ {\bf  Unified analysis.} In this work we propose a {\em unifying theoretical framework}
 which covers all of the  variants of {\tt SGD} outlined in Section~\ref{sec:many_faces_of_SGD}. As a by-product, we obtain the {\em first unified analysis} of vanilla and variance-reduced {\tt SGD} methods.  For instance, our analysis covers as special cases vanilla {\tt SGD} methods from~\cite{nguyen2018sgd} and~\cite{SGD_AS}, variance-reduced {\tt SGD} methods such as {\tt SAGA}~\cite{SAGA}, {\tt L-SVRG}~\cite{hofmann2015variance, Loopless} and {\tt JacSketch}~\cite{gower2018stochastic}. Another by-product  is {\em the first unified analysis of {\tt SGD} methods which include {\tt RCD}.} For instance, our theory covers the subspace descent method {\tt SEGA}~\cite{hanzely2018sega} as a special case. Lastly, our framework is general enough to capture the phenomenon of {\em quantization}. For instance, we obtain the {\tt DIANA} and {\tt VR-DIANA} methods in special cases. 


$\bullet$ {\bf Generalization of existing methods.} An important yet {\em relatively} minor contribution of our work is that it enables  {\em generalization} of knowns methods.  For instance, some particular methods we consider, such as {\tt L-SVRG} (Alg~\ref{alg:L-SVRG})~\cite{Loopless}, were not analyzed in the proximal ($R\neq 0$) case before. To illustrate how this can be done within our framework, we do it here for {\tt L-SVRG}. Further, all methods we analyze can be extended to the {\em arbitrary sampling} paradigm.

$\bullet$ {\bf Sharp rates.} In all known special cases, the rates obtained from our general theorem (Theorem~\ref{thm:main_gsgm}) are the {\em best known rates} for these methods.

$\bullet$ {\bf New methods.}   Our general analysis provides estimates for a possibly infinite array of new and yet-to-be-developed variants of {\tt SGD}. One only needs to verify that Assumption~\ref{as:general_stoch_gradient} holds, and a complexity estimate is readily furnished by Theorem~\ref{thm:main_gsgm}. Selected existing and new methods that fit our framework are summarized in Table~\ref{tbl:special_cases2}. This list is for illustration only, we believe that future work by us and others will lead to its rapid expansion.

$\bullet$ {\bf Experiments.} We show through extensive experimentation that some of the {\em new} and {\em generalized} methods proposed here and analyzed via our framework have some intriguing practical properties when compared against appropriately selected existing methods.

\section{Main Result \label{sec:main_res}}

We first introduce the key assumption on the stochastic gradients $g^k$ enabling our general analysis (Assumption~\ref{as:general_stoch_gradient}), then state our assumptions on $f$ (Assumption~\ref{as:mu_strongly_quasi_convex}), and finally  state and comment on our  unified convergence result (Theorem~\ref{thm:main_gsgm}). 

{\bf Notation.} We use the following notation. $\langle x, y\rangle \eqdef \sum_i x_i y_i$ is the standard Euclidean inner product, and $\norm{x}\eqdef \langle x, x\rangle ^{1/2} $ is the induced $\ell_2$ norm. For simplicity we assume that \eqref{eq:problem_gen} has a unique minimizer, which we denote $x^*$. Let $D_f(x,y)$ denote the \textit{Bregman divergence} associated with $f$: $D_f(x,y) \eqdef f(x) - f(y) - \<\nabla f(y), x-y>$. We  often write $[n]\eqdef \{1,2,\dots,n\}$.

\subsection{Key assumption}

Our first assumption is of key importance. It is mainly an assumption on the sequence of stochastic gradients $\{g^k\}$ generated by an arbitrary randomized algorithm. Besides unbiasedness (see \eqref{eq:general_stoch_grad_unbias}), we require two recursions to hold for the iterates $x^k$ and the stochastic gradients $g^k$ of a randomized method. We allow for flexibility by casting these inequalities in a parametric manner.

\begin{assumption}\label{as:general_stoch_gradient} Let $\{x^k\}$ be the random iterates produced by proximal {\tt SGD} (Algorithm in Eq~\eqref{eq:SGD}).
We first assume that the stochastic gradients $g^k$ are unbiased
    \begin{equation}\label{eq:general_stoch_grad_unbias}
        \EE\left[ g^k\mid x^k\right] = \nabla f(x^k),
    \end{equation}
    for all $k\geq 0$. Further, we assume that there exist
non-negative constants $A, B, C, D_1, D_2, \rho$ and a (possibly) random sequence $\{\sigma_k^2\}_{k\ge 0}$ such that the following two relations hold\footnote{For convex and $L$-smooth $f$, one can show that $    \norm{\nabla f(x) - \nabla f(y)}^2 \le 2LD_{f}(x,y).$ Hence, $D_f$ can be used as a measure of proximity for the gradients.}
    \begin{equation}\label{eq:general_stoch_grad_second_moment}
        \EE\left[\norm{g^k -\nabla f(x^*)}^2\mid x^k\right] \le 2AD_f(x^k,x^*) + B\sigma_k^2 + D_1,
    \end{equation}
    \begin{equation}\label{eq:gsg_sigma}
        \EE\left[\sigma_{k+1}^2 \, \mid \, \sigma^2_k\right] \le (1-\rho) \sigma_k^2 + 2CD_f(x^k,x^*)  + D_2,
    \end{equation} 
The expectation above is with respect to the randomness of the algorithm. 

\end{assumption}

The unbiasedness assumption \eqref{eq:general_stoch_grad_unbias} is standard.  The key innovation we bring is inequality \eqref{eq:general_stoch_grad_second_moment} coupled with \eqref{eq:gsg_sigma}. We argue, and justify this statement by furnishing many examples in Section~\ref{sec:main-paper-special-cases}, that these inequalities capture the essence of a wide array of existing and some new {\tt SGD} methods, including vanilla, variance reduced, arbitrary sampling, quantized and coordinate descent variants. Note that in the case when $\nabla f(x^*) = 0$ (e.g., when $R=0$),  the inequalities in Assumption~\ref{as:general_stoch_gradient}  reduce to
\begin{equation}\label{eq:general_stoch_grad_second_moment_special}
        \EE\left[\norm{g^k}^2\mid x^k\right] \le 2A(f(x^k) - f(x^*)) + B\sigma_k^2 + D_1,
\end{equation}
\begin{equation}\label{eq:gsg_sigma_special}
    \EE\left[\sigma_{k+1}^2 \, \mid \, \sigma^2_k\right] \le (1-\rho) \sigma_k^2 + 2C(f(x^k) - f(x^*)) + D_2.
\end{equation}
Similar inequalities can be found in the analysis of stochastic first-order methods. However, this is the first time that such inequalities are generalized, equipped with parameters, and elevated to the status of an assumption that can be used on its own, independently from any other details defining the underlying method that generated them.

\subsection{Main theorem}

 For simplicity, we shall assume throughout that $f$ is $\mu$-strongly quasi-convex, which is a generalization of $\mu$-strong convexity. We leave an analysis under different assumptions on $f$ to future work.

\begin{assumption}[$\mu$-strong quasi-convexity]\label{as:mu_strongly_quasi_convex}
There exists $\mu>0$ such that $f:\R^d\to \R$ is \textit{$\mu$-strongly quasi-convex}.  That is, the following inequality holds:
    \begin{equation}\label{eq:mu_strongly_quasi_convex}
    \compactify
        f(x^*) \ge f(x) + \langle \nabla f(x), x^* - x\rangle + \frac{\mu}{2}\norm{x^* - x}^2, \qquad \forall x\in\R^d \,.    
    \end{equation}
\end{assumption}

We are now ready to present our main convergence result.

 \begin{theorem}\label{thm:main_gsgm}
Let Assumptions~\ref{as:general_stoch_gradient}~and~\ref{as:mu_strongly_quasi_convex} be satisfied. Choose constant $M$ such that $M > \frac{B}{\rho}$. Choose a stepsize  satisfying
    \begin{equation}\label{eq:gamma_condition_gsgm}
    \compactify
        0 < \gamma \le \min\left\{\frac{1}{\mu}, \frac{1}{A+CM}\right\}.
    \end{equation}
    Then the iterates $\{x^k\}_{k\geq 0}$ of proximal {\tt SGD} (Algorithm~\eqref{eq:SGD}) satisfy
    \begin{equation}\label{eq:main_result_gsgm}
 \compactify       \EE\left[V^k\right] \le \max\left\{(1-\gamma\mu)^k,\left(1+\frac{B}{M}-\rho\right)^k\right\} V^0 + \frac{(D_1+MD_2)\gamma^2 }{\min\left\{\gamma\mu, \rho - \frac{B}{M}\right\}},
    \end{equation}
    where the Lyapunov function $V^k$ is defined by $V^k \eqdef \norm{x^k - x^*}^2 + M\gamma^2\sigma_k^2$.
\end{theorem}

This theorem establishes a linear rate for a wide range of proximal {\tt SGD} methods up to a certain oscillation radius, controlled by the additive term in \eqref{eq:main_result_gsgm}, and namely, by parameters $D_1$ and $D_2$. As we shall see  in Section~\ref{sec:special_cases} (refer to Table~\ref{tbl:special_cases-parameters}), the main difference between the vanilla and variance-reduced {\tt SGD} methods is that while the former  satisfy inequality \eqref{eq:gsg_sigma} with $D_1>0$ or $D_2>0$, which in view of \eqref{eq:main_result_gsgm} prevents them from reaching the optimum $x^*$ (using a fixed stepsize), the latter methods satisfy inequality \eqref{eq:gsg_sigma} with $D_1=D_2=0$, which in view of \eqref{eq:main_result_gsgm} enables them to reach the optimum.

\section{The Classic, The Recent and The Brand New} \label{sec:main-paper-special-cases}

In this section we deliver on the promise from the introduction and show how many existing and some new variants of {\tt SGD} fit our general framework (see Table~\ref{tbl:special_cases2}). 

{\bf An overview.} As claimed, our framework is powerful enough to include vanilla methods (\xmark\; in the ``VR'' column) as well as variance-reduced methods (\cmark\; in the ``VR'' column), methods which generalize to arbitrary sampling (\cmark\; in the ``AS'' column), methods supporting gradient quantization (\cmark\; in the ``Quant'' column) and finally, also {\tt RCD} type methods (\cmark\; in the ``RCD'' column). 

\begin{table}[!t]
\begin{center}
\footnotesize
\begin{tabular}{|c|c|c|c|c|c|c|c|c|c|}
\hline
Problem & Method &  Alg \# & Citation &   VR?  & AS? & Quant? &  RCD? & Section  & Result \\
\hline
\eqref{eq:problem_gen}+\eqref{eq:f_exp} & {\tt SGD}  & Alg \ref{alg:sgd_prox} & \cite{nguyen2018sgd}  & \xmark &  \xmark & \xmark &  \xmark & \ref{sec:SGD} & 
Cor~\ref{cor:recover_sgd_rate} \\
\eqref{eq:problem_gen}+\eqref{eq:f_sum}  & {\tt SGD-SR} & Alg \ref{alg:sgdas} & \cite{SGD_AS} & \xmark &  \cmark & \xmark & \xmark & \ref{SGD-AS} &   Cor~\ref{cor:recover_sgd-as_rate} \\
\eqref{eq:problem_gen}+\eqref{eq:f_sum} &  {\tt SGD-MB} & Alg~\ref{alg:SGD-MB} & {\bf NEW} & \xmark & \xmark & \xmark & \xmark & \ref{sec:SGD-MB} & Cor~\ref{cor:mb}   \\
\eqref{eq:problem_gen}+\eqref{eq:f_sum} &  {\tt SGD-star} & Alg \ref{alg:SGD-star} & {\bf NEW} & \cmark & \cmark  & \xmark & \xmark & \ref{sec:SGD-star} & Cor~\ref{cor:SGD-star} \\
\eqref{eq:problem_gen}+\eqref{eq:f_sum}  & {\tt SAGA} & Alg~\ref{alg:SAGA} & \cite{SAGA} & \cmark & \xmark  & \xmark & \xmark & \ref{sec:saga} & Cor~\ref{thm:recover_saga_rate} \\
\eqref{eq:problem_gen}+\eqref{eq:f_sum}  & {\tt N-SAGA} & Alg~\ref{alg:N-SAGA} & {\bf NEW} & \xmark &  \xmark & \xmark & \xmark & \ref{N-SAGA} & Cor~\ref{cor:N-SAGA} \\
\eqref{eq:problem_gen} & {\tt SEGA}  & Alg~\ref{alg:SEGA} &  \cite{hanzely2018sega}  & \cmark &   \xmark & \xmark & \cmark & \ref{sec:sega} & Cor~\ref{cor:sega} \\
\eqref{eq:problem_gen} & {\tt N-SEGA}  & Alg~\ref{alg:N-SEGA} &  {\bf NEW}  & \xmark & \xmark  & \xmark  & \cmark & \ref{N-SEGA} &
Cor~\ref{cor:N-SEGA}\\
\eqref{eq:problem_gen}+\eqref{eq:f_sum}  & {\tt SVRG}${}^{a}$  & Alg~\ref{alg:SVRG} & \cite{SVRG} & \cmark & \xmark & \xmark & \xmark & \ref{sec:svrg} & Cor~\ref{cor:svrg} \\
\eqref{eq:problem_gen}+\eqref{eq:f_sum}  & {\tt L-SVRG} & Alg~\ref{alg:L-SVRG} & \cite{hofmann2015variance, Loopless} & \cmark &  \xmark & \xmark & \xmark & \ref{sec:L-SVRG} & Cor~\ref{cor:recover_l-svrg_rate}  \\
\eqref{eq:problem_gen}+\eqref{eq:f_sum}  & {\tt DIANA} & Alg~\ref{alg:diana} & \cite{mishchenko2019distributed,horvath2019stochastic} & \xmark & \xmark & \cmark &  \xmark & \ref{sec:diana} & Cor~\ref{cor:main_diana}  \\
\eqref{eq:problem_gen}+\eqref{eq:f_sum}  & {\tt DIANA}${}^b$  & Alg~\ref{alg:diana_case}  & \cite{mishchenko2019distributed, horvath2019stochastic} & \cmark &  \xmark & \cmark  &  \xmark &\ref{sec:diana}  & Cor~\ref{cor:main_diana_special_case}\\
\eqref{eq:problem_gen}+\eqref{eq:f_sum}  & {\tt Q-SGD-SR} & Alg \ref{alg:qsgdas} & {\bf NEW} & \xmark &  \cmark & \cmark & \xmark & \ref{Q-SGD-AS} &   Cor~\ref{cor:recover_q-sgd-as_rate} \\ 
\eqref{eq:problem_gen}+\eqref{eq:f_sum}+\eqref{eq:f_i_sum}  & {\tt VR-DIANA} & Alg~\ref{alg:vr-diana} & \cite{horvath2019stochastic}& \cmark & \xmark & \cmark & \xmark & \ref{sec:VR-DIANA} & Cor~\ref{cor:main_vr_diana} \\
\eqref{eq:problem_gen}+\eqref{eq:f_sum}& {\tt JacSketch} & Alg~\ref{alg:jacsketch} & \cite{gower2018stochastic} & \cmark &  \cmark \xmark & \xmark & \xmark & \ref{sec:JacSketch} & Cor~\ref{thm:main_jacsketch}\\
\hline
\end{tabular}
\end{center}
\caption{List of specific existing (in some cases generalized) and new methods which fit our general analysis framework. VR = variance reduced method, AS = arbitrary sampling, Quant = supports gradient quantization, RCD = randomized coordinate descent type method. ${}^{a}$ Special case of {\tt SVRG} with  1 outer loop only;  ${}^{b}$ Special case of {\tt DIANA} with $1$ node and quantization of exact gradient. }
\label{tbl:special_cases2}
\end{table}

For existing methods we provide a citation; new methods developed in this paper are marked accordingly.  Due to space restrictions, all algorithms are described (in detail) in the Appendix; we provide a link to the appropriate section for easy navigation. While these details are important, the main message of this paper, i.e., the generality of our approach, is captured by Table~\ref{tbl:special_cases2}. The ``Result'' column of Table~\ref{tbl:special_cases2} points to a corollary of Theorem~\ref{thm:main_gsgm}; these corollaries state in detail the convergence  statements for the various methods. In all cases where known methods are recovered, these corollaries of Theorem~\ref{thm:main_gsgm} recover the best known rates.

{\bf Parameters.} From the point of view of Assumption~\ref{as:general_stoch_gradient}, the methods listed in Table~\ref{tbl:special_cases2} exhibit certain patterns.  To shed some light on this, in Table~\ref{tbl:special_cases-parameters} we summarize the values of these parameters. 

\begin{table}[h]
\begin{center}
\footnotesize
\begin{tabular}{|c|c|c|c|c|c|c|}
\hline
 Method &   $A$ & $B$ & $\rho$ & $C$ & $D_1$ & $D_2$\\
\hline
 {\tt SGD}   &  $2L $ & $0$ & $1$ & $0$ & $2\sigma^2$ & $0$ \\
 {\tt SGD-SR} &  $2 \cL$ & $0$ & $1$ & $0$ & $2\sigma^2$ & $0$ \\
  {\tt SGD-MB}  &  $\frac{A' + L(\tau-1)}{\tau}$ & 0 & $1$ & $0$ & $ \frac{D'}{\tau}$ & $0$ \\
  {\tt SGD-star}  &  $2 \cL$ & $0$ & $1$ & $0$ & $0$ & $0$ \\
 {\tt SAGA}   &  $2L$ & $2$ & $1/n$ & $L/n$ & $0$ & $0$\\
{\tt N-SAGA}   &  $2L$ & $2$ & $1/n$ & $L/n$ & $2\sigma^2$ & $\frac{\sigma^2}{n}$\\
{\tt SEGA}   &   $2dL$ & $2d$ & $1/d$ & $L/d$ & $0$ & $0$\\
{\tt N-SEGA}  &   $2dL$ & $2d$ & $1/d$ & $L/d$ & $2d\sigma^2$ & $\frac{\sigma^2}{d}$\\
 {\tt SVRG}${}^{a}$   & $2L$  & $2$ & $0$ & $0$& $0$ & $0$ \\
 {\tt L-SVRG}   &  $2 L$ & $2$ & $p$ & $Lp$ & $0$ & $0$ \\
 {\tt DIANA}   & $\left(1+\frac{2\omega}{n}\right)L$ & $\frac{2\omega}{n}$ & $\alpha$ & $L\alpha$ & $\frac{(1+\omega)\sigma^2}{n}$ & $\alpha\sigma^2$  \\
 {\tt DIANA}${}^b$  &  $\left(1+2\omega\right)L$ & $2\omega$ & $\alpha$ & $L\alpha$ & $0$ & $0$ \\
{\tt Q-SGD-SR} &  $2(1+\omega)\cL$ & $0$ & $1$ & $0$ & $2(1+\omega)\sigma^2$ & $0$ \\
 {\tt VR-DIANA} &  $\left(1+\frac{4\omega + 2}{n}\right)L$ & $\frac{2(\omega+1)}{n}$ & $\alpha$ & $\left(\frac{1}{m}+4\alpha\right)L$ & $0$ & $0$ \\
 {\tt JacSketch}   &  $2\cL_1$ & $\frac{2\lambda_{\max}}{n}$ & $\lambda_{\min}$ & $\frac{\cL_2}{n}$ & $0$ & $0$ \\
\hline
\end{tabular}
\end{center}
\caption{The parameters for which the methods from Table~\ref{tbl:special_cases2} satisfy Assumption~\ref{as:general_stoch_gradient}. The meaning of the expressions appearing in the table is defined in detail in the Appendix.}
\label{tbl:special_cases-parameters}
\end{table}

Note, for example, that for all methods the parameter $A$ is non-zero. Typically, this a multiple of an appropriately defined smoothness parameter (e.g., $L$ is the Lipschitz constant of the gradient of $f$, $\cL$ and $\cL_1$ in {\tt SGD-SR}\footnote{{\tt SGD-SR} is first {\tt SGD} method analyzed in the {\em arbitrary sampling} paradigm. It was developed using the  {\em stochastic reformulation} approach (whence the ``SR'') pioneered in \cite{ASDA} in a numerical linear algebra setting, and later extended to develop the {\tt JacSketch} variance-reduction technique for finite-sum optimization  \cite{gower2018stochastic}.}, {\tt SGD-star} and {\tt JacSketch} are {\em expected smoothness} parameters). In the three variants of the {\tt DIANA} method, $\omega$ captures the variance of the  quantization operator $Q$. That is, one assumes that $\Exp{Q(x)}=x$ and $\Exp{\norm{Q(x)-x}^2} \leq \omega \norm{x}^2$ for all $x\in \R^d$. In view of \eqref{eq:gamma_condition_gsgm}, large $A$ means a smaller stepsize, which slows down the rate. Likewise, the variance $\omega$ also affects the parameter $B$, which in view of \eqref{eq:main_result_gsgm} also has an adverse effect on the rate.
Further, as predicted by Theorem~\ref{thm:main_gsgm}, whenever either $D_1>0$ or $D_2>0$, the corresponding method converges to an oscillation region only. These methods are not variance-reduced. All symbols used in Table~\ref{tbl:special_cases-parameters} are defined in the appendix, in the same place where the methods are described and analyzed.

{\bf Five new methods.} To illustrate the usefulness of our general framework, we develop  {\em 5 new variants} of {\tt SGD} never explicitly considered in the literature before (see Table~\ref{tbl:special_cases2}). Here we briefly motivate them; details can be found in the Appendix. 

$\bullet$ {\tt SGD-MB} (Algorithm~\ref{alg:SGD-MB}).
This method is specifically designed for functions of the finite-sum structure  \eqref{eq:f_i_sum}. As we show through experiments, this is a powerful mini-batch {\tt SGD} method, with mini-batches formed with replacement as follows: in each iteration, we repeatedly ($\tau$ times) and independently pick $i\in [n]$ with probability $p_i>0$. Stochastic gradient $g^k$ is then formed by averaging the stochastic gradients $\nabla f_i(x^k)$ for all selected indices $i$ (including each $i$ as many times as this index was selected). 

$\bullet$ {\tt SGD-star} (Algorithm~\ref{alg:SGD-star}). 
This new method forms a bridge between vanilla and variance-reduced {\tt SGD} methods. While  not practical, it sheds light on the role of variance reduction. Again, we consider functions of the finite-sum form~\eqref{eq:f_i_sum}. This methods answers the following question:  assuming that the gradients $\nabla f_i(x^*)$, $i\in [n]$ are {\em known}, can they be used to design a more powerful {\tt SGD} variant? The answer is {\em yes}, and {\tt SGD-star} is the method. In its most basic form, {\tt SGD-star}  constructs the stochastic gradient via $g^k=\nabla f_i(x^k) - \nabla f_i(x^*)$, where $i\in [n]$ is chosen uniformly at random. That is, the standard stochastic gradient $\nabla f_i(x^k)$ is perturbed by the stochastic gradient at the same index $i$ evaluated at the optimal point $x^*$. Inferring from Table~\ref{tbl:special_cases-parameters}, where $D_1=D_2=0$, this method converges to $x^*$, and not merely to some oscillation region.  Variance-reduced methods essentially work by iteratively constructing increasingly more accurate {\em estimates} of $\nabla f_i(x^*)$. Typically, the term $\sigma_k^2$ in the Lyapunov function of variance reduced methods will contain a term of the form $\sum_i \norm{h_i^k - \nabla f_i(x^k)}^2$, with $h_i^k$ being the estimators maintained by the method. Remarkably, {\tt SGD-star} was never explicitly considered in the literature before.

$\bullet$ {\tt N-SAGA} (Algorithm~\ref{alg:N-SAGA}). This is a novel variant of {\tt SAGA}~\cite{SAGA}, one in which one does not have access to the gradients of $f_i$, but instead only has access to {\em noisy} stochastic estimators thereof (with noise $\sigma^2$). Like {\tt SAGA}, {\tt N-SAGA} is able to reduce the variance inherent in the finite sum structure~\eqref{eq:f_i_sum} of the problem. However, it necessarily pays the price of noisy estimates of $\nabla f_i $, and hence, just like vanilla {\tt SGD} methods, is ultimately unable to converge to $x^*$. The oscillation region is governed by the noise level $\sigma^2$ (refer to $D_1$ and $D_2$ in Table~\ref{tbl:special_cases-parameters}). This method will be of practical importance for problems where each $f_i$ is of the form~\eqref{eq:f_exp}, i.e., for problems of the ``average of expectations'' structure. Batch versions of {\tt N-SAGA} would be well suited for distributed optimization, where each $f_i$ is owned by a different worker, as in such a case one wants the workers to work in parallel.

$\bullet$ {\tt N-SEGA} (Algorithm~\ref{alg:N-SEGA}). This is a {\em noisy} extension of the {\tt RCD}-type method {\tt SEGA}, in complete analogy with the relationship between {\tt SAGA} and {\tt N-SAGA}. Here we assume that we only have noisy estimates of partial derivatives (with noise $\sigma^2$). This situation is common in derivative-free optimization, where such a noisy estimate can be obtained by taking (a random) finite difference approximation~\cite{nesterov2017randomDFO}. Unlike {\tt SEGA}, {\tt N-SEGA} only converges to an oscillation region the size of which is governed by $\sigma^2$.

$\bullet$ {\tt Q-SGD-SR} (Algorithm~\ref{alg:qsgdas}). This is a quantized version of {\tt SGD-SR}, which is the first {\tt SGD} method analyzed in the  arbitrary sampling paradigm. As such, {\tt Q-SGD-SR}  is a vast generalization of the celebrated {\tt QSGD} method \cite{alistarh2017qsgd}.

\section{Experiments \label{sec:exp}}

In this section we numerically verify the claims from the paper. We present only a fraction of experiments here, the rest is contained in Appendix~\ref{sec:extra}.

In Section~\ref{sec:SGD-MB}, we describe in detail the {\tt SGD-MB} method already outlined before. The main advantage of {\tt SGD-MB} is that the sampling procedure it employs can be implemented in just $\cO(\tau \log n)$ time. In contrast, even the simplest without-replacement sampling which selects each function into the minibatch with a prescribed probability independently (we will refer to it as independent {\tt SGD}) requires $n$ calls of a uniform random generator.  We demonstrate numerically that {\tt SGD-MB} has essentially identical iteration complexity to  independent {\tt SGD} in practice. We consider logistic regression with Tikhonov regularization. For a fixed  expected sampling size $\tau$, consider two options for the probability of sampling the $i$-th function:
\begin{enumerate}[label=(\roman*)]
\item \label{item:unif} $\frac{\tau}{n}$, or
\item \label{item:imp} $\frac{\norm{a_i}^2+\lambda}{\delta+\norm{a_i}^2+\lambda}$, where $\delta$ is such that\footnote{An {\tt RCD} version of this sampling was proposed in~\cite{AccMbCd}; it was shown to be superior to uniform sampling both in theory and practice.} $\sum_{i=1}^n\frac{\norm{a_i}^2+\lambda}{\delta+\norm{a_i}^2+\lambda}=1$. 
\end{enumerate}
The results can be found in Figure~\ref{fig:SGDMB}, where we also report the choice of stepsize $\gamma$ and the choice of $\tau$ in the legend and title of the plot, respectively.

\begin{figure}[!h]
\centering
\begin{minipage}{0.24\textwidth}
  \centering
\includegraphics[width =  \textwidth ]{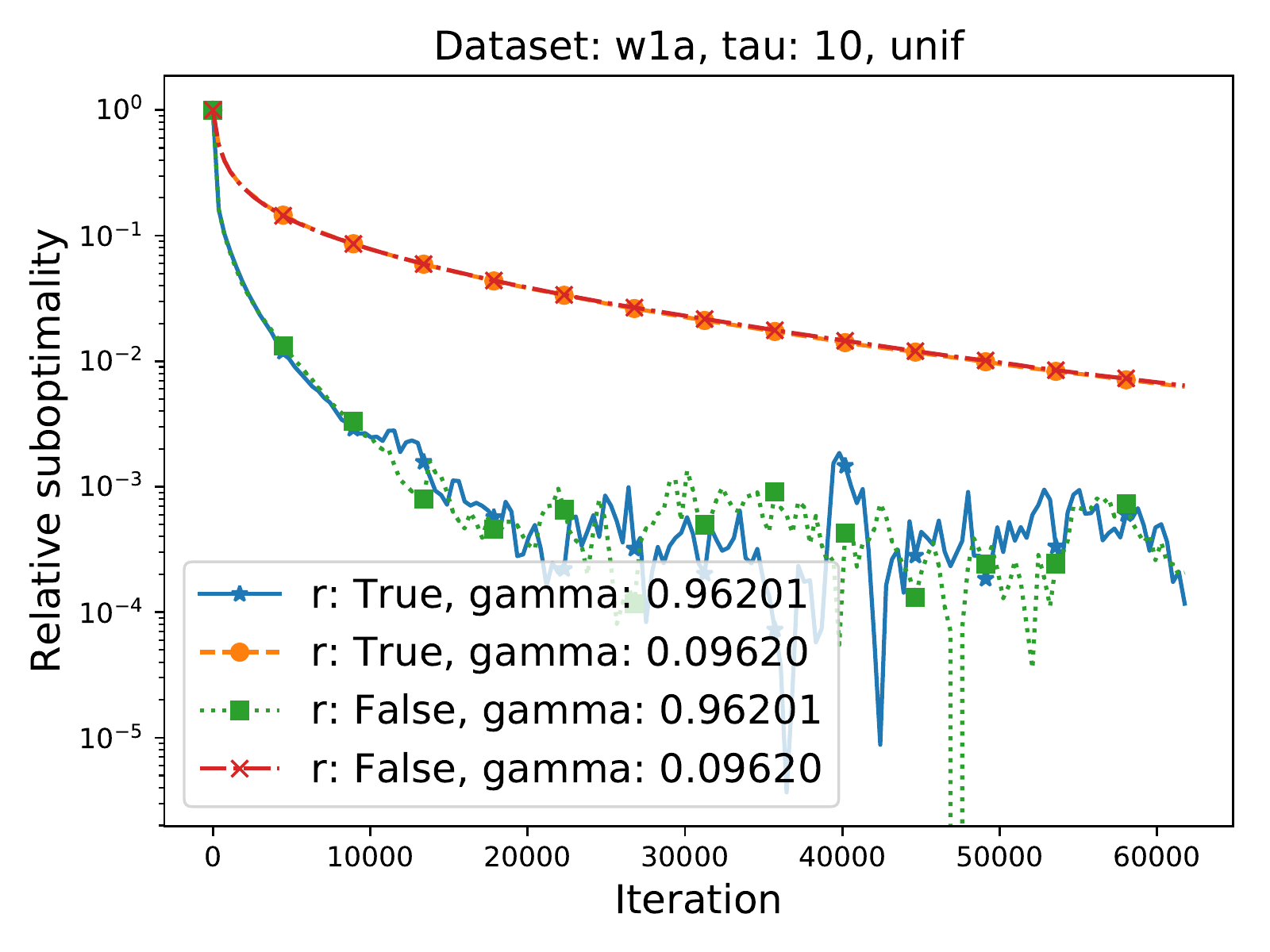}
\end{minipage}%
\begin{minipage}{0.24\textwidth}
  \centering
\includegraphics[width =  \textwidth ]{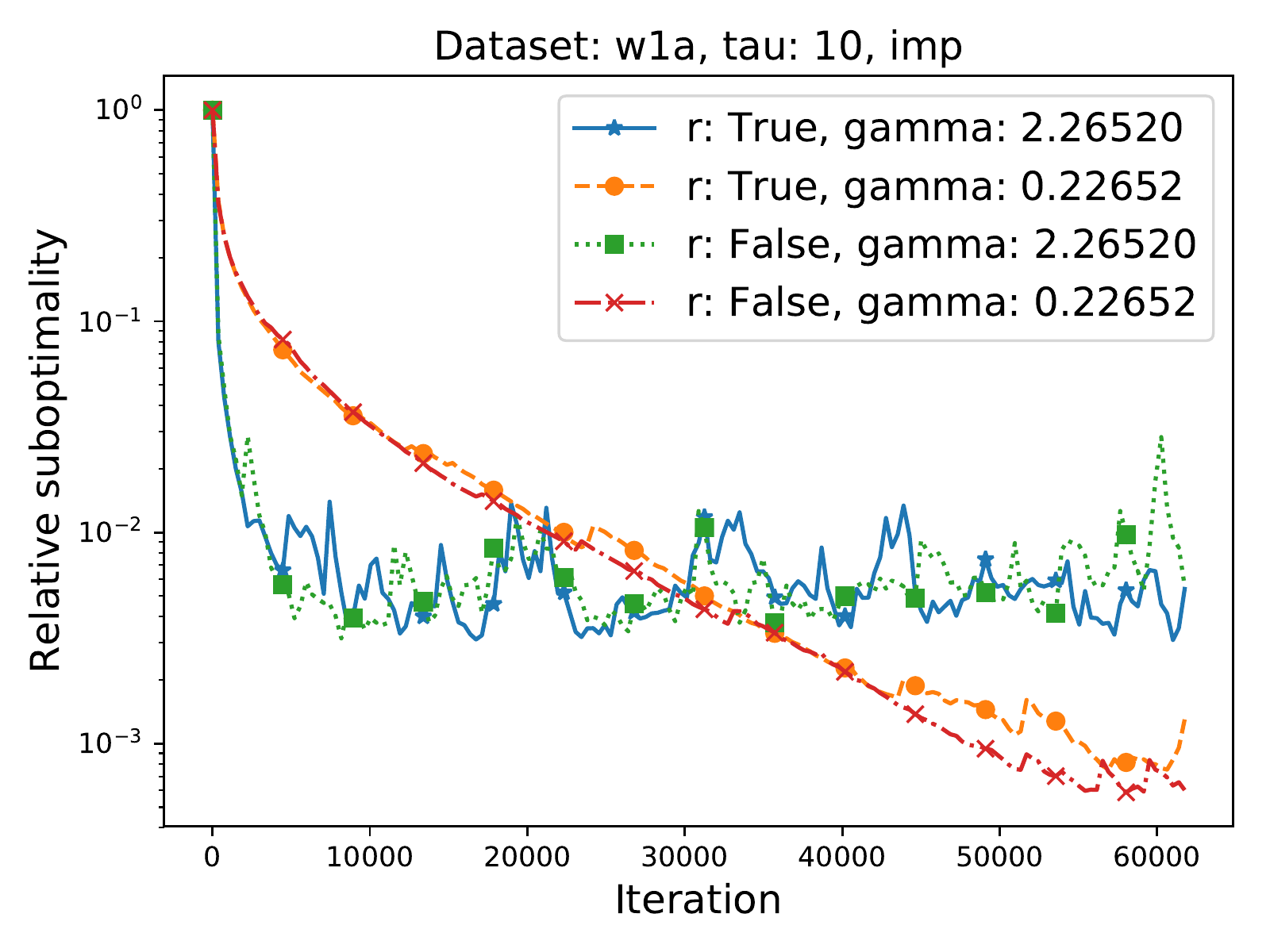}
\end{minipage}%
\begin{minipage}{0.24\textwidth}
  \centering
\includegraphics[width =  \textwidth ]{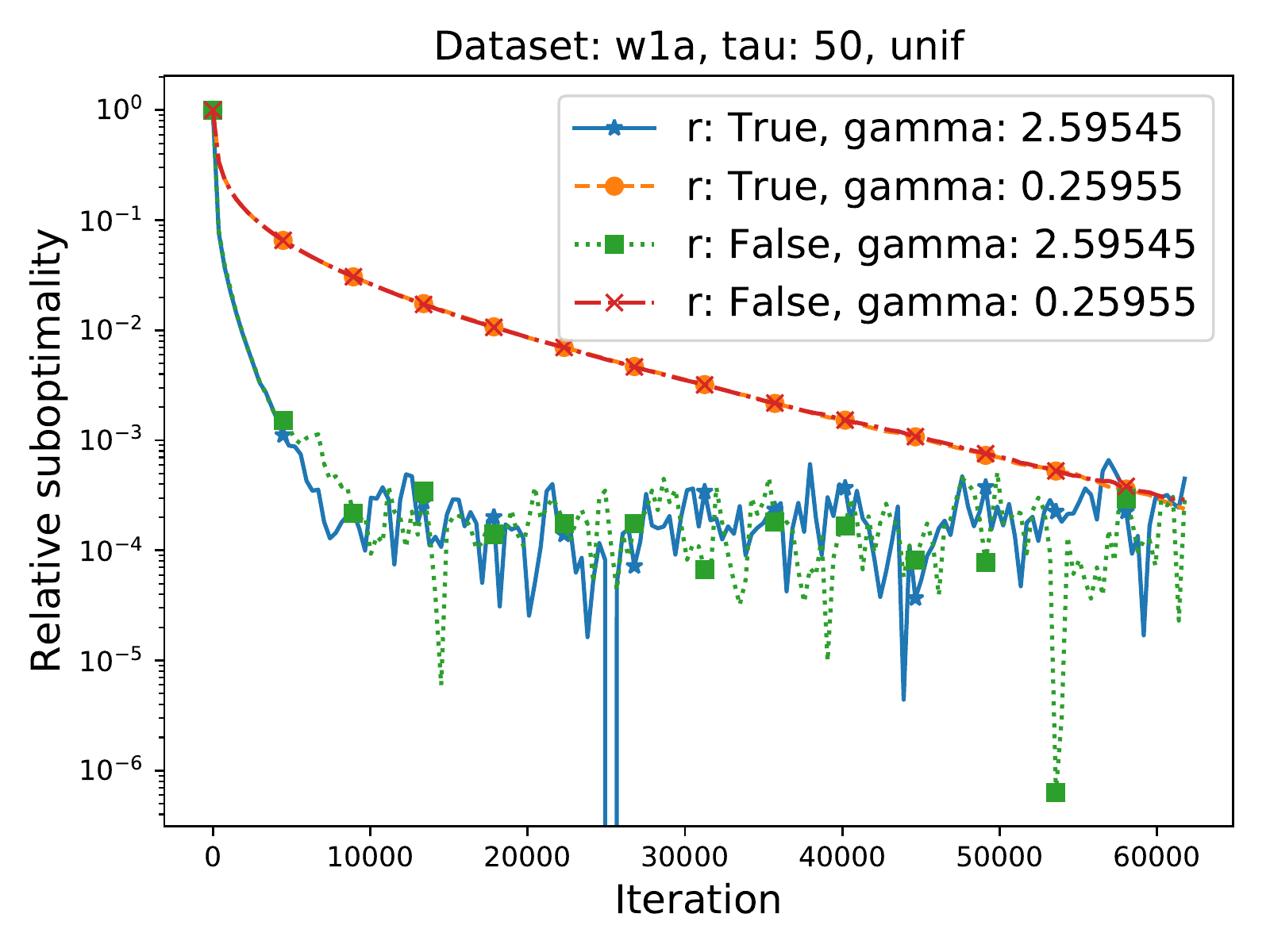}
\end{minipage}%
\begin{minipage}{0.24\textwidth}
  \centering
\includegraphics[width =  \textwidth ]{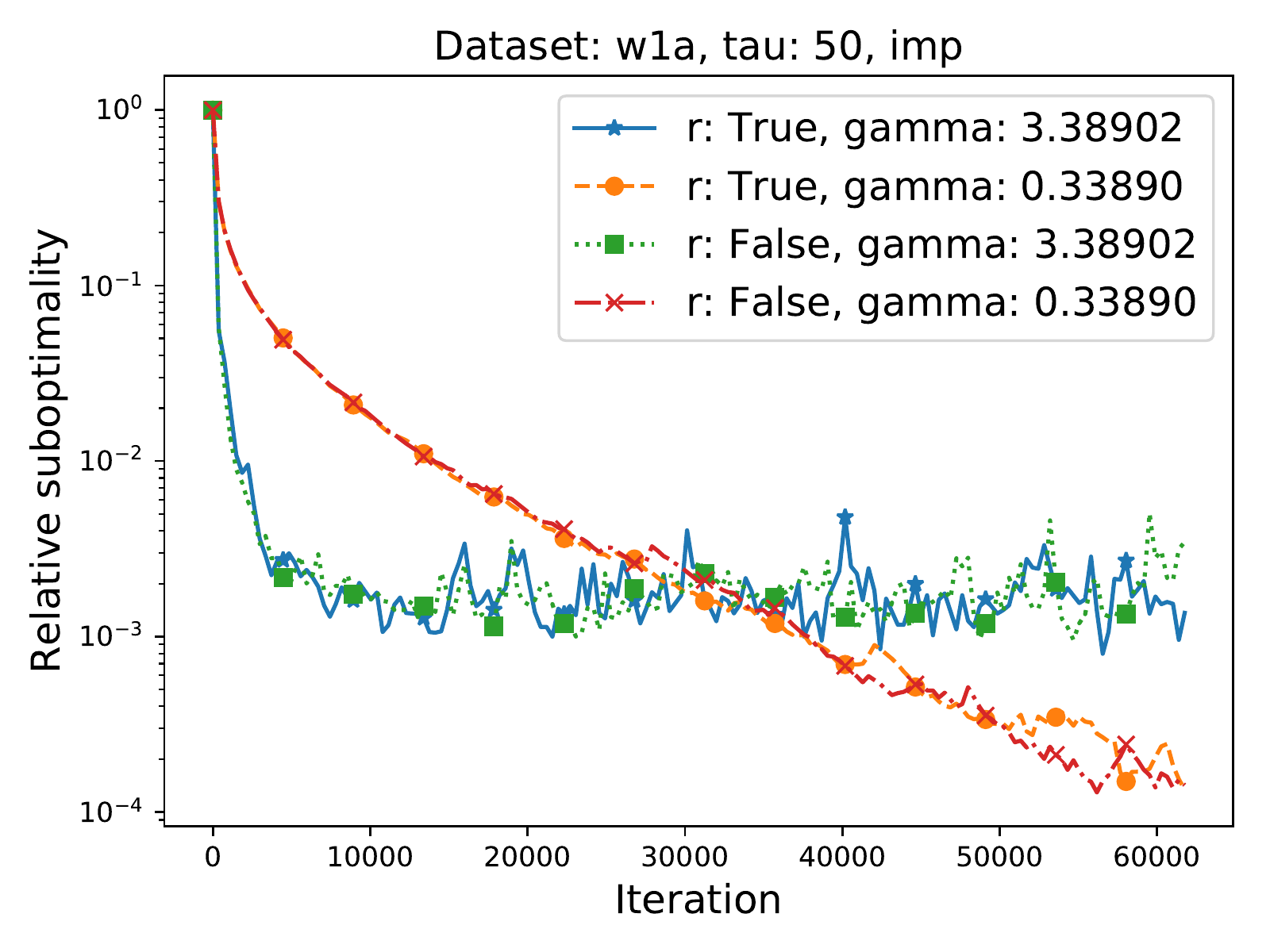}
\end{minipage}%
\\
\begin{minipage}{0.24\textwidth}
  \centering
\includegraphics[width =  \textwidth ]{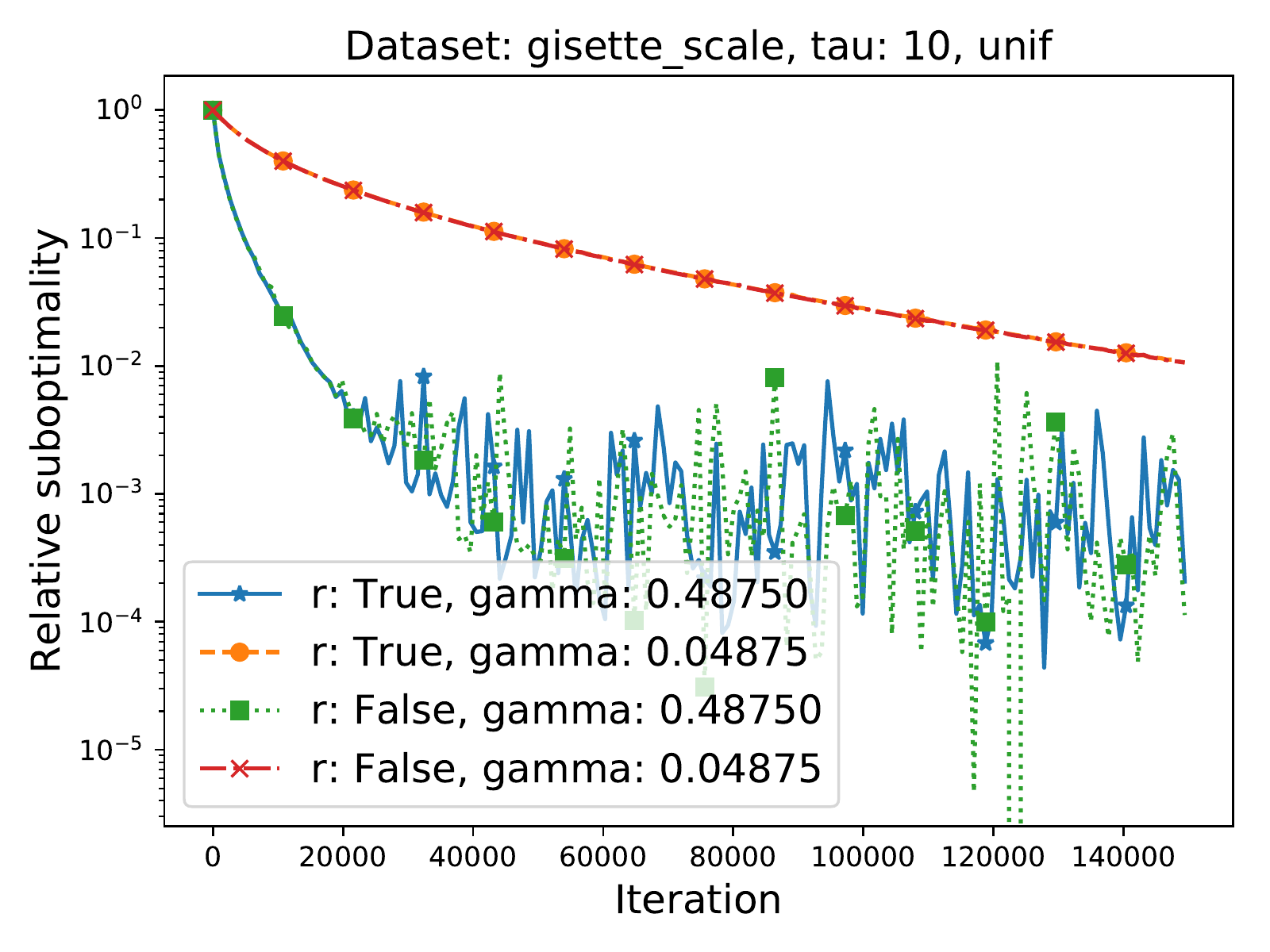}
\end{minipage}%
\begin{minipage}{0.24\textwidth}
  \centering
\includegraphics[width =  \textwidth ]{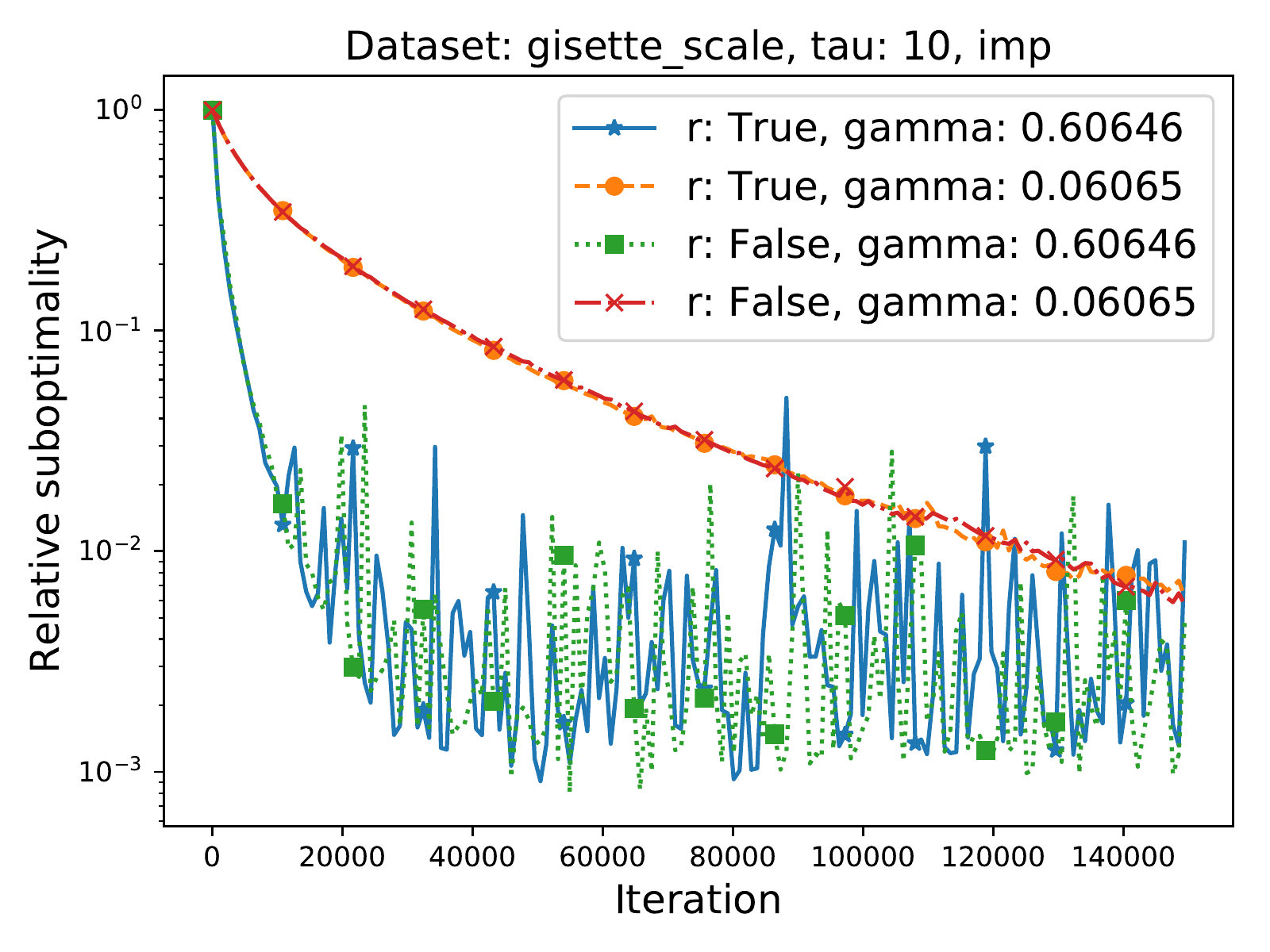}
\end{minipage}%
\begin{minipage}{0.24\textwidth}
  \centering
\includegraphics[width =  \textwidth ]{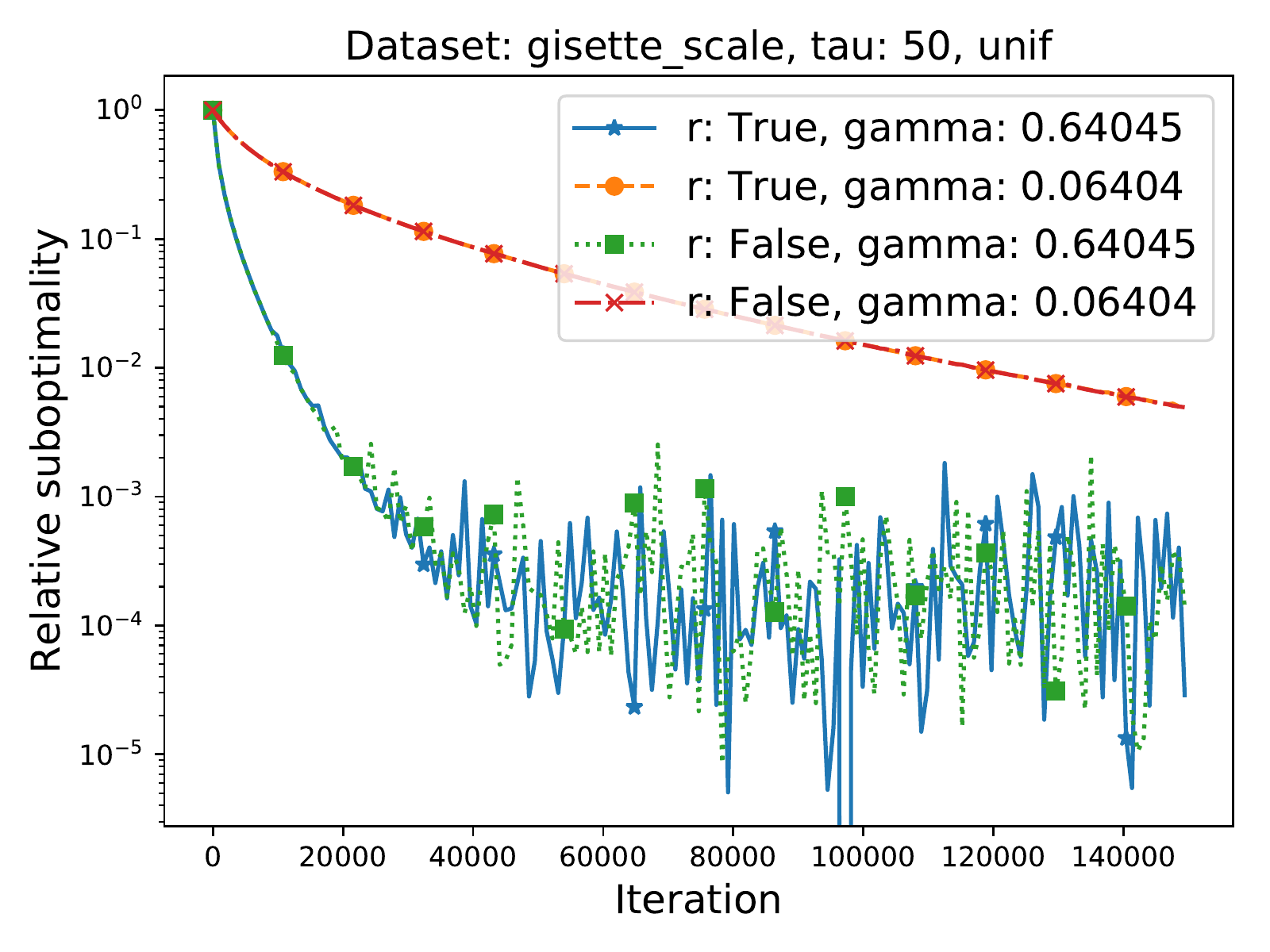}
\end{minipage}%
\begin{minipage}{0.24\textwidth}
  \centering
\includegraphics[width =  \textwidth ]{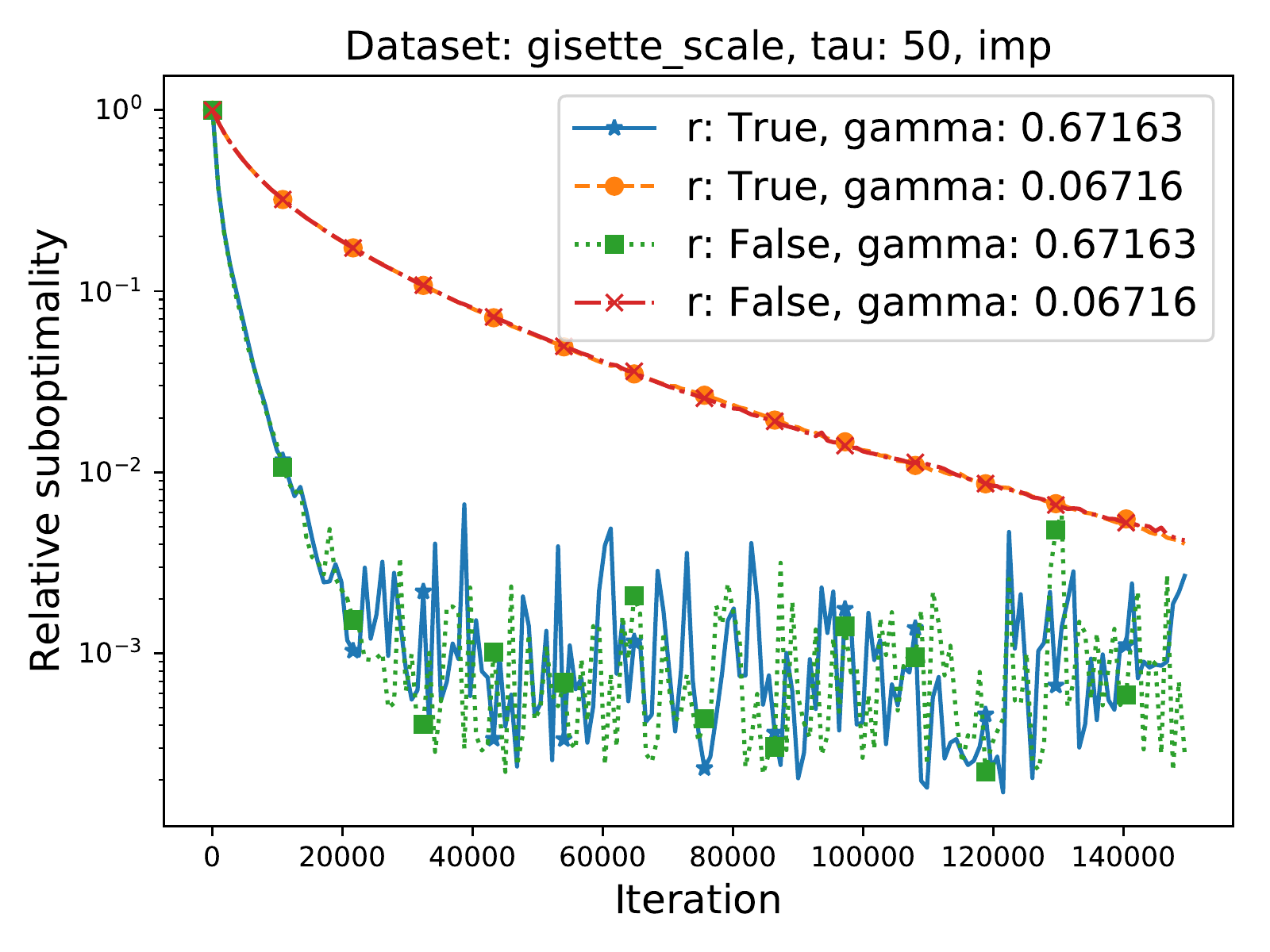}
\end{minipage}%
\caption{{\tt SGD-MB} and independent {\tt SGD} applied on LIBSVM~\cite{chang2011libsvm}. Title label ``unif'' corresponds to probabilities chosen by~\ref{item:unif} while label ``imp'' corresponds to probabilities chosen by~\ref{item:imp}. Lastly, legend label ``r'' corresponds to ``replacement'' with value ``True'' for {\tt SGD-MB} and value ``False'' for independent {\tt SGD}.}
\label{fig:SGDMB}
\end{figure}

Indeed, iteration complexity of {\tt SGD-MB} and independent {\tt SGD} is almost identical. Since the cost of each iteration of {\tt SGD-MB} is cheaper\footnote{The relative difference between iteration costs of {\tt SGD-MB} and independent {\tt SGD} can be arbitrary, especially for the case when cost of evaluating $\nabla f_i(x)$ is cheap, $n$ is huge and $n\gg \tau$. In such case, cost of one iteration of {\tt SGD-MB} is $\tau \text{Cost}(\nabla f_i) +\tau\log(n)$ while the cost of one iteration of independent {\tt SGD} is $\tau \text{Cost}(\nabla f_i) + n$.}, we conclude superiority of {\tt SGD-MB} to independent {\tt SGD}.
\section{Limitations and Extensions}

Although our approach is rather general, we still see several possible directions for  future extensions, including: 

$\bullet$ We believe our results can be extended to {\em weakly convex} functions. However, producing a comparable result in the {\em nonconvex} case remains a major open problem. 

$\bullet$ It would be further interesting to unify our theory with {\em biased} gradient estimators. If this was possible, one could recover methods as {\tt SAG}~\cite{SAG} in special cases, or obtain rates for the zero-order optimization. We have some preliminary results in this direction already.

$\bullet$ Although our theory allows for non-uniform stochasticity, it does not recover the best known rates for {\tt RCD} type methods with {\em importance sampling}. It would be thus interesting to provide a more refined  analysis capable of capturing importance sampling phenomena more accurately.

$\bullet$ An extension of Assumption~\ref{as:general_stoch_gradient} to {\em iteration dependent} parameters $A,B,C,D_1, D_2, \rho$ would enable an array of new methods, such as {\tt SGD} with decreasing stepsizes. 

$\bullet$ It would be interesting to provide a unified analysis of stochastic methods with {\em acceleration} and {\em momentum}. In fact,~\cite{kulunchakov2019estimate} provide (separately) a unification of some methods with and without variance reduction. Hence, an attempt to combine our insights with their approach seems to be a promising starting point in these efforts.

\bibliographystyle{plain} 
\bibliography{sigma_k}

\newpage
\appendix
\part*{Appendix  \\ \Large A Unified Theory of SGD: Variance Reduction, Sampling, Quantization  and Coordinate Descent}

\tableofcontents

\clearpage

\section{Special Cases}\label{sec:special_cases}

\subsection{Proximal {\tt SGD} for stochastic optimization}\label{sec:SGD}

\begin{algorithm}[h]
    \caption{{\tt SGD}}
    \label{alg:sgd_prox}
    \begin{algorithmic}
        \Require learning rate $\gamma>0$, starting point $x^0\in\R^d$, distribution $\cD$ over $\xi $
        \For{ $k=0,1,2,\ldots$ }
        \State{Sample $\xi \sim \cD$}
        \State{$g^k = \nabla f_\xi (x^k)$}
        \State{$x^{k+1} = \proxR(x^k - \gamma g^k)$}
        \EndFor
    \end{algorithmic}
\end{algorithm}

We start with stating the problem, the assumptions on the objective and on the stochastic gradients for {\tt SGD} \cite{nguyen2018sgd}. Consider the expectation minimization problem
\begin{equation}\label{eq:main_sgd}
    \min_{x\in\R^d} f(x) + R(x),\quad f(x) \eqdef \ED{f_\xi(x)}
\end{equation}
where $\xi \sim \cD$, $f_\xi(x)$ is differentiable and $L$-smooth almost surely in $\xi$.

 Lemma~\ref{lem:lemma1_sgd} shows that the stochastic gradient $g^k = \nabla f_\xi(x^k)$ satisfies Assumption~\ref{as:general_stoch_gradient}. The corresponding choice of parameters can be found in Table~\ref{tbl:special_cases-parameters}. 

\begin{lemma}[Generalization of Lemmas 1,2 from~\cite{nguyen2018sgd}]\label{lem:lemma1_sgd}
    Assume that $f_\xi(x)$ is convex in $x$ for every $\xi$. Then for every $x\in\R^d$
    \begin{equation}\label{eq:lemma1_sgd}
        \ED{\norm{\nabla f_\xi(x)- \nabla f(x^*)}^2} \le 4L(D_f(x,x^*)) + 2\sigma^2,
    \end{equation}
    where $\sigma^2\eqdef \EE_\xi\left[\norm{\nabla f_\xi(x^*)}^2\right]$. If further $f(x)$ is $\mu$-strongly convex with possibly non-convex $f_\xi$, then for every $x\in\R^d$
    \begin{equation}\label{eq:lemma2_sgd}
        \ED{\norm{\nabla f_\xi(x) - \nabla f(x^*)}^2} \le 4L\kappa(D_f(x,x^*)) + 2\sigma^2,
    \end{equation}
    where $\kappa = \frac{L}{\mu}$.
\end{lemma}

\begin{corollary}\label{cor:recover_sgd_rate}
    Assume that $f_\xi(x)$ is convex in $x$ for every $\xi$ and $f$ is $\mu$-strongly quasi-convex. Then {\tt SGD} with $\gamma \le \frac{1}{2L}$ satisfies
    \begin{equation}\label{eq:recover_sgd_rate}
        \EE\left[\norm{x^k - x^*}^2\right] \le (1-\gamma\mu)^k\norm{x^0-x^*}^2 + \frac{2\gamma \sigma^2}{\mu}.
    \end{equation}
    If we further assume that $f(x)$ is $\mu$-strongly convex with possibly non-convex $f_\xi(x)$,   {\tt SGD} with $\gamma \le \frac{1}{2L\kappa}$ satisfies~\eqref{eq:recover_sgd_rate} as well. 
\end{corollary}
\begin{proof}
    It suffices to plug parameters from Table~\ref{tbl:special_cases-parameters} into Theorem~\ref{thm:main_gsgm}. 
\end{proof}

\subsubsection*{Proof of Lemma~\ref{lem:lemma1_sgd}}
The proof is a direct generalization to the one from~\cite{nguyen2018sgd}.
Note that
\begin{eqnarray*}
&& \frac12 \ED{\norm{ \nabla f_\xi (x) - \nabla f(x^*)}^2} - \ED{\norm{\nabla f_\xi (x^*) - \nabla f(x^*)}^2} \\
&& \qquad \qquad=
\frac12 \ED{\norm{ \nabla f_\xi (x) - \nabla f(x^*)}^2 - \norm{\nabla f_\xi (x^*) - \nabla f(x^*)}^2} \\
&& \qquad \qquad\overset{\eqref{eq:1/2a_minus_b}}{\le}
 \ED{\norm{ \nabla f_\xi (x) - \nabla f_{\xi}(x^*)}^2}  \\
&& \qquad \qquad\leq
2LD_f(x,x^*).
\end{eqnarray*}
It remains to rearrange the above to get~\eqref{eq:lemma1_sgd}. To obtain~\eqref{eq:lemma2_sgd}, we shall proceed similarly:
\begin{eqnarray*}
&& \frac12 \ED{\norm{ \nabla f_\xi (x) - \nabla f(x^*)}^2} - \ED{\norm{ \nabla f_\xi (x^*) - \nabla f(x^*)}^2} \\
&& \qquad \qquad=
\frac12 \ED{\norm{ \nabla f_\xi (x) - \nabla f(x^*)}^2 - \norm{ \nabla f_\xi (x^*) - \nabla f(x^*)}^2} \\
&& \qquad \qquad\overset{\eqref{eq:1/2a_minus_b}}{\le}
 \ED{\norm{\nabla f_\xi (x) - \nabla f_{\xi}(x^*)}^2}  \\
&& \qquad \qquad\leq
L^2\norm{ x - x^*}^2 \\ 
&& \qquad \qquad\leq
2\frac{L^2}{\mu}D_f(x,x^*).
\end{eqnarray*}
Again, it remains to rearrange the terms.

\subsection{{\tt SGD-SR}}\label{SGD-AS}

In this section, we recover convergence result of {\tt SGD} under expected smoothness property from~\cite{SGD_AS}. This setup allows obtaining tight convergence rates of {\tt SGD} under arbitrary stochastic reformulation of finite sum minimization\footnote{For technical details on how to exploit expected smoothness for specific reformulations, see~\cite{SGD_AS}}. 

The stochastic reformulation is a special instance of~\eqref{eq:main_sgd}:
\begin{equation}\label{eq:problem_sgd-as}
    \min\limits_{x\in\R^d}f(x) + R(x),\quad f(x) =\ED{f_\xi(x)}, \quad  f_\xi(x) \eqdef \frac1n \sum_{i=1}^n \xi_i f_i(x) 
\end{equation}
where $\xi$ is a random vector from distribution $\cD$ such that for all $i$: $\ED{\xi_i}=1$ and $f_i$ (for all $i$) is smooth, possibly non-convex function. We next state the expextes smoothness assumption. A specific instances of this assumption allows to get tight convergence rates of SGD, which we recover in this section.

\begin{algorithm}[h]
    \caption{{\tt SGD-SR}}
    \label{alg:sgdas}
    \begin{algorithmic}
        \Require learning rate $\gamma>0$, starting point $x^0\in\R^d$, distribution $\cD$ over $\xi \in\R^n$ such that $\ED{\xi}$ is vector of ones
        \For{ $k=0,1,2,\ldots$ }
        \State{Sample $\xi \sim \cD$}
        \State{$g^k = \nabla f_\xi (x^k)$}
        \State{$x^{k+1} = \proxR(x^k - \gamma g^k)$}
        \EndFor
    \end{algorithmic}
\end{algorithm}

\begin{assumption}[Expected smoothness]\label{as:exp_smoothness_sgd-as}
    We say that $f$ is $\cL$-smooth in expectation with respect to distribution $\cD$ if there exists $\cL = \cL(f,\cD) > 0$ such that
    \begin{equation}\label{eq:exp_smoothness_sgd-as}
        \ED{\norm{\nabla f_\xi(x) - \nabla f_\xi(x^*)}^2} \le 2\cL D_f(x,x^*),
    \end{equation}
    for all $x\in\R^d$. For simplicity, we will write $(f,\cD) \sim ES(\cL)$ to say that \eqref{eq:exp_smoothness_sgd-as} holds.
\end{assumption}

Next, we present Lemma~\ref{lem:exp_smoothness_grad_up_bound_sgd-as} which shows that choice of constants for Assumption~\ref{as:general_stoch_gradient} from Table~\ref{tbl:special_cases-parameters} is valid. 

\begin{lemma}[Generalization of Lemma~2.4, \cite{SGD_AS}]\label{lem:exp_smoothness_grad_up_bound_sgd-as}
    If $(f,\cD)\sim ES(\cL)$, then
    \begin{equation}\label{eq:exp_smoothness_grad_up_bound_sgd-as}
        \ED{\norm{\nabla f_\xi(x) - \nabla f(x^*)}^2} \le 4\cL D_f(x,x^*) + 2\sigma^2.
    \end{equation}
    where $\sigma^2 \eqdef \ED{\norm{\nabla f_\xi(x^*)}^2}$.
\end{lemma}

A direct consequence of Theorem~\ref{thm:main_gsgm} in this setup is Corollary~\ref{cor:recover_sgd-as_rate}.

\begin{corollary}\label{cor:recover_sgd-as_rate}
    Assume that $f(x)$ is $\mu$-strongly quasi-convex and $(f,\cD)\sim ES(\cL)$. Then {\tt SGD-SR} with $\gamma^k \equiv\gamma \le \frac{1}{2\cL}$ satisfies
    \begin{equation}\label{eq:recover_sgd-as_rate}
        \EE\left[\norm{x^k - x^*}^2\right] \le (1-\gamma\mu)^k\norm{x^0-x^*}^2 + \frac{2\gamma\sigma^2}{\mu}.
    \end{equation}
\end{corollary}


\subsubsection*{Proof of Lemma~\ref{lem:exp_smoothness_grad_up_bound_sgd-as}}
Here we present the generalization of the proof of Lemma~2.4 from \cite{SGD_AS} for the case when $\nabla f(x^*) \neq 0$. In this proof all expectations are conditioned on $x^k$.
\begin{eqnarray*}
    \EE\left[\norm{\nabla f_{\xi} (x) - \nabla f(x^*)}^2\right] &=& \EE\left[ \norm{ \nabla f_\xi(x) - \nabla f_{\xi}(x^*) + \nabla f_{\xi}(x^*) - \nabla f(x^*) }^2\right] \\
    &\overset{\eqref{eq:a_b_norm_squared}}{\le}&  2 \EE\left[\norm{ \nabla f_{\xi}(x) - \nabla f_{\xi}(x^*)}^2\right] 
 + 2  \EE\left[\norm{\nabla f_{\xi}(x^*) - \nabla f(x^*)}^2\right] \notag\\ 
    &\overset{\eqref{eq:exp_smoothness_sgd-as}}{\le}& 4 \cL D_f(x,x^*) + 2 \sigma^2.
\end{eqnarray*}

\subsection{{\tt SGD-MB}}\label{sec:SGD-MB}
In this section, we present a specific practical formulation of~\eqref{eq:problem_sgd-as} which was not considered in~\cite{SGD_AS}. The resulting algorithm (Algorithm~\ref{alg:SGD-MB}) is novel; it was not considered in~\cite{SGD_AS} as a specific instance of {\tt SGD-SR}. The key idea behind {\tt SGD-MB} is constructing unbiased gradient estimate via with-replacement sampling.

Consider random variable $\nu \sim \cD$ such that 
\begin{equation}\label{eq:nb87fg87f}
 \Prob(\nu = i) =p_i;  \qquad \sum_{i=1}^np_i=1.
\end{equation} Notice that if we define \begin{equation}
\label{eq:reform_function}\psi_i(x)\eqdef \frac{1}{n p_i} f_i(x), \qquad i=1,2,\dots,n,\end{equation}
then 
\begin{equation} \label{eq:reform_problem}f(x) = \frac{1}{n} \sum_{i=1}^n f_i(x)  \overset{\eqref{eq:reform_function}}{=} \sum_{i=1}^n p_i \psi_i(x) \overset{\eqref{eq:nb87fg87f}}{=} \ED{\psi_\nu (x)}.\end{equation}
So, we have rewritten the finite sum problem \eqref{eq:f_sum} into the {\em equivalent stochastic optimization problem}  
\begin{equation}\label{eq:reform_opt} \min_{x\in \R^d} \ED{\psi_\nu (x)}.\end{equation}

We are now ready to describe our method. At each iteration $k$ we sample $\nu^k_i,\dots,\nu^k_{\tau}\sim \cD$ independently ($1\leq \tau \leq n$), and define $g^k\eqdef  \frac{1}{\tau}\sum_{i=1}^\tau \nabla \psi_{\nu^k_i}(x^k) $. Further, we use $g^k$ as a stochastic gradient, resulting in Algorithm~\ref{alg:SGD-MB}.

\begin{algorithm}[h]
    \caption{{\tt SGD-MB}}
    \label{alg:SGD-MB}
    \begin{algorithmic}
        \Require learning rate $\gamma>0$, starting point $x^0\in\R^d$, distribution $\cD$ over $\nu$ such that~\eqref{eq:nb87fg87f} holds.
        \For{ $k=0,1,2,\ldots$ }
        \State{Sample $\nu^k_i,\dots,\nu^k_{\tau}\sim \cD$ independently}
        \State{$g^k =  \frac{1}{\tau}\sum_{i=1}^\tau \nabla \psi_{\nu^k_i}(x^k) $}
        \State{$x^{k+1} = x^k - \gamma g^k$}
        \EndFor
    \end{algorithmic}
\end{algorithm}

To remain in full generality, consider the following Assumption.

\begin{assumption}\label{ass:main_mb} There exists constants $A'>0$ and $D'\geq 0$ such that \begin{equation}\label{eq:ABassumpt} \ED{ \norm{\nabla \psi_{\nu}(x)}^2 }\leq 2 A' (f(x) - f(x^*)) + D'\end{equation} for all $x\in \R^d$.  
\end{assumption}

Note that it is sufficient to have convex and smooth $f_i$ in order to satisfy Assumption~\ref{ass:main_mb}, as Lemma~\ref{lem:b98h0f} states. 

\begin{lemma}\label{lem:b98h0f}  Let $\sigma^2\eqdef \ED{\norm{\nabla \psi_\nu (x^*)}^2}$. If $f_i$ are convex and $L_i$-smooth, then Assumption~\ref{ass:main_mb} holds for $A'=2\cL$ and $D'=2\sigma^2$, where \begin{equation}\label{eq:ineq90f9f}\cL \leq \max_i \frac{L_i}{n p_i}.\end{equation} 
If moreover $\nabla f_i(x^*)=0$ for all $i$,  then Assumption~\ref{ass:main_mb} holds for $A'=\cL$ and $D'=0$. 
\end{lemma}

Next, Lemma~\ref{lem:mb} states that Algorithm~\ref{alg:SGD-MB} indeed satisfies Assumption~\ref{as:general_stoch_gradient}. 

\begin{lemma} \label{lem:mb}
Suppose that Assumption~\ref{ass:main_mb} holds. Then $g^k$ is unbiased; i.e. $\ED{g^k} = \nabla f(x^k)$. Further, 
\begin{eqnarray*}
\ED{\norm{g^k}^2} \leq  \frac{2A' + 2L(\tau-1)}{\tau}(f(x^k) - f(x^*)) +  \frac{  D'}{\tau}.
\end{eqnarray*}
\end{lemma}
Thus, parameters from Table~\ref{tbl:special_cases-parameters} are validated. As a direct consequence of Theorem~\ref{thm:main_gsgm} we get Corollary~\ref{cor:mb}.

\begin{corollary}\label{cor:mb}
    As long as $0< \gamma \leq \frac{\tau}{A' + L(\tau-1)}$, we have
    \begin{equation}\label{eq:mb_rate}
\EE \norm{x^k-x^*}^2 \leq (1-\gamma \mu)^k \norm{x^0-x^*}^2 + \frac{\gamma D'}{\mu \tau}.    \end{equation}
\end{corollary}

\begin{remark}
For $\tau=1$, {\tt SGD-MB} is a special of the method from~\cite{SGD_AS}, Section~3.2. However, for $\tau>1$, this is a different method; the difference lies in the with-replacement sampling. Note that with-replacement trick allows for efficient and implementation of independent importance sampling~\footnote{Distribution of random sets $S$ for which random variables $i\in S$ and $j\in S$ are independent for $j\neq i$.} with complexity $\cO(\tau \log(n))$. In contrast, implementation of without-replacement importance sampling has complexity $\cO(n)$, which can be significantly more expensive to the cost of evaluating $\sum_{i\in S} \nabla f_i(x)$.
 \end{remark}

\subsubsection*{Proof of Lemma~\ref{lem:mb}}
Notice first that
\begin{eqnarray*}
\ED{g^k}
&\overset{\eqref{eq:reform_function}}{=}&
  \frac{1}{\tau}\sum_{i=1}^\tau\ED{\frac{1}{n p_{\nu^k_{i}}} \nabla f_{\nu^k_i}(x^k)}
  \\
&=& 
\ED{\frac{1}{n p_{\nu}} \nabla f_{\nu}(x^k)}
  \\
&\overset{\eqref{eq:nb87fg87f}}{=}& 
\sum_{i=1}^n p_i \frac{1}{n p_i} \nabla f_i(x^k) 
\\
&=& \nabla f(x_k).
\end{eqnarray*}

So, $g^k$ is an unbiased estimator of the gradient $\nabla f(x^k)$. Next, 
\begin{eqnarray*}
\ED{\norm{g^k}^2}&=& \ED{\norm{\frac{1 }{\tau}\sum_{i=1}^\tau \nabla \psi_{\nu^k_i}(x^k)}^2 }\\ 
&=& \frac{1}{\tau^2}\ED{\sum_{i=1}^\tau \norm{\nabla \psi_{\nu^k_i}(x^k)}^2  +  2    \sum_{i< j} \< \nabla \psi_{\nu^k_i}(x^k),  \nabla \psi_{\nu^k_j}(x^k)>  } \\
&=& \frac{1}{\tau}\ED{ \norm{\nabla \psi_{\nu}(x^k)}^2}+  \frac{2}{\tau^2}    \sum_{i< j} \< \ED{\nabla \psi_{\nu^k_i}(x^k) },  \ED{\nabla \psi_{\nu^k_j}(x^k) }>     \\
&=& \frac{1}{\tau} \ED{\norm{\nabla \psi_{\nu}(x^k)}^2} + \frac{\tau - 1}{\tau} \norm{\nabla f(x^k)}^2 \\
&\overset{\eqref{eq:ABassumpt} }{\leq}& \frac{2A' (f(x^k) - f(x^*)) + D' + 2L(\tau-1)(f(x^k) - f(x^*))}{\tau}.
\end{eqnarray*}

\subsubsection*{Proof of Lemma~\ref{lem:b98h0f}}
Let $\cL = \cL(f,\cD)>0$ be any constant for which
\begin{equation}\label{eq:ES}\EE_{\xi\sim \cD} \norm{\nabla \phi_{\xi}(x) - \nabla \phi_{\xi}(x^*)}^2 \leq 2\cL (f(x) - f(x^*))\end{equation}
 holds for all $x\in \R^d$. This is the expected smoothness property (for a single item sampling) from \cite{SGD_AS}. It was shown in \cite[Proposition 3.7]{SGD_AS}  that \eqref{eq:ES} holds, and that $ \cL$ satisfies \eqref{eq:ineq90f9f}. The claim now follows by applying \cite[Lemma~2.4]{SGD_AS}. 

\subsection{{\tt SGD-star}}\label{sec:SGD-star}
Consider problem~\eqref{eq:problem_sgd-as}. Suppose that $\nabla f_i(x^*)$ is known for all $i$. In this section we present a novel algorithm~---~{\tt SGD-star}~---~which is {\tt SGD-SR} shifted by the stochastic gradient in the optimum. The method is presented under Expected Smoothness Assumption~\eqref{eq:exp_smoothness_sgd-as}, obtaining general rates under arbitrary sampling.  The algorithm is presented as Algorithm~\ref{alg:SGD-star}. 

\begin{algorithm}[h]
    \caption{{\tt SGD-star}}
    \label{alg:SGD-star}
    \begin{algorithmic}
        \Require learning rate $\gamma>0$, starting point $x^0\in\R^d$, distribution $\cD$ over $\xi \in \R^n$ such that $\ED{\xi}$ is vector of ones
        \For{ $k=0,1,2,\ldots$ }
        \State{Sample $\xi \sim \cD$}
        \State{$g^k = \nabla f_\xi(x^k) - \nabla f_\xi(x^*) + \nabla f(x^*)$}
        \State{$x^{k+1} = \proxR(x^k - \gamma g^k)$}
        \EndFor
    \end{algorithmic}
\end{algorithm}
Suppose that $(f,\cD) \sim ES(\cL)$. 
Note next that {\tt SGD-star} is just  {\tt SGD-SR} applied on objective $D_f(x,x^*)$ instead of $f(x)$ when $\nabla f(x^*) = 0$. This careful design of the objective yields $(D_f(\cdot, x^*),\cD) \sim ES(\cL)$ and $\ED{\norm{\nabla_x D_{f_\xi}(x,x^*)}^2 \, \mid x=x^*} = 0 $, and thus Lemma~\eqref{lem:exp_smoothness_grad_up_bound_sgd-as} becomes
\begin{lemma}[Lemma~2.4, \cite{SGD_AS}]\label{lem:SGD-star}
    If $(f,\cD)\sim ES(\cL)$, then
    \begin{equation}\label{eq:sgd-star_lemma}
        \ED{\norm{g^k -\nabla f(x^*)}^2} \le 4\cL D_f(x^k,x^*).
    \end{equation}
\end{lemma}

A direct consequence of Corollary (thus also a direct consequence of Theorem~\ref{thm:main_gsgm}) in this setup is Corollary~\ref{cor:SGD-star}.

\begin{corollary}\label{cor:SGD-star}
    Suppose that $(f,\cD)\sim ES(\cL)$. Then {\tt SGD-star} with $\gamma = \frac{1}{2\cL}$ satisfies
    \begin{equation}\label{eq:SGD-shift-xx}
        \EE\left[\norm{x^k - x^*}^2\right] \le \left(1-\frac{\mu}{2\cL} \right)^k\norm{x^0-x^*}^2.
    \end{equation}
\end{corollary}


\begin{remark}
Note that results from this section are obtained by applying results from~\ref{SGD-AS}. Since Section~\ref{sec:SGD-MB} presets a specific sampling algorithm for {\tt SGD-SR}, the results can be thus extended to {\tt SGD-star} as well.
\end{remark}

\subsubsection*{Proof of Lemma~\ref{lem:SGD-star}}
In this proof all expectations are conditioned on $x^k$.
\begin{eqnarray*}
    \ED{\norm{g^k -\nabla f(x^*)}^2}  &=& \ED{\norm{\nabla f_\xi(x^k) -\nabla f_\xi(x^*)}^2} \\ 
    &\overset{\eqref{eq:exp_smoothness_sgd-as}}{\le}& 4 \cL D_f(x^k,x^*).
\end{eqnarray*}

\subsection{{\tt SAGA}}\label{sec:saga}
In this section we show that our approach is suitable for {\tt SAGA} \cite{SAGA} (see Algorithm~\ref{alg:SAGA}). Consider the finite-sum minimization problem
\begin{equation}\label{eq:main_l-svrg}
    f(x) = \frac{1}{n}\sum\limits_{i=1}^n f_i(x) + R(x),
\end{equation}
where $f_i$ is convex, $L$-smooth for each $i$ and $f$ is $\mu$-strongly convex. 

\begin{algorithm}[h]
    \caption{{\tt SAGA} \cite{SAGA}}
    \label{alg:SAGA}
    \begin{algorithmic}
        \Require learning rate $\gamma>0$, starting point $x^0\in\R^d$
        \State Set $\psi_j^0 = x^0$ for each $j\in[n]$
        \For{ $k=0,1,2,\ldots$ }
        \State{Sample  $j \in [n]$ uniformly at random}
        \State{Set $\phi_j^{k+1} = x^k$ and $\phi_i^{k+1} = \phi_i^{k}$ for $i\neq j$}
        \State{$g^k = \nabla f_j(\phi_j^{k+1}) - \nabla f_j(\phi_j^k) + \frac{1}{n}\sum\limits_{i=1}^n\nabla f_i(\phi_i^k)$}
        \State{$x^{k+1} = \proxR\left(x^k - \gamma g^k\right)$}
        \EndFor
    \end{algorithmic}
\end{algorithm}

\begin{lemma}\label{lem:stoch_grad_second_moment_saga}
    We have
    \begin{equation}\label{eq:stoch_grad_second_moment_saga1}
        \EE\left[\norm{g^k- \nabla f(x^*)}^2\mid x^k\right] \le 4LD_f(x^k,x^*)+ 2\sigma_k^2
    \end{equation}
    and 
        \begin{equation}\label{eq:stoch_grad_second_moment_saga2}
        \EE\left[\sigma_{k+1}^2\mid x^k\right] \le \left(1 - \frac{1}{n}\right)\sigma_k^2 + \frac{2L}{n}D_f(x^k,x^*),
    \end{equation}
    where $\sigma_k^2 = \frac{1}{n}\sum\limits_{i=1}^n\norm{\nabla f_i(\phi_i^k) - \nabla f_i(x^*)}^2$.
\end{lemma}

Clearly, Lemma~\ref{lem:stoch_grad_second_moment_saga} shows that Algorithm~\ref{alg:SAGA} satisfies Assumption~\ref{as:general_stoch_gradient}; the corresponding parameter choice can be found in Table~\ref{tbl:special_cases-parameters}. Thus, as a direct consequence of Theorem~\ref{thm:main_gsgm} with $M=4n$ we obtain the next corollary.

\begin{corollary}\label{thm:recover_saga_rate}
    {\tt SAGA} with $\gamma = \frac{1}{6L}$ satisfies
    \begin{equation}\label{eq:recover_saga_rate_coro}
        \EE V^k \le \left(1-\min\left\{\frac{\mu}{6L},\frac{1}{2n}\right\}\right)^kV^0.
    \end{equation}
\end{corollary}

\subsubsection*{Proof of Lemma~\ref{lem:stoch_grad_second_moment_saga}}
Note that  Lemma~\ref{lem:stoch_grad_second_moment_saga} is a special case of Lemmas 3,4 from~\cite{mishchenko201999} without prox term. We reprove it with prox for completeness.

    Let all expectations be conditioned on $x^k$ in this proof. Note that $L$-smoothness and convexity of $f_i$ implies
\begin{equation}\label{eq:norm_diff_grads}
    \frac{1}{2L}\norm{\nabla f_i(x) - \nabla f_i(y)}^2 \le f_i(x) - f_i(y) - \<\nabla f_i(y),x-y>,\quad \forall x,y\in\R^d, i\in[n].
\end{equation}

 By definition of $g^k$ we have
    \begin{eqnarray*}
        \EE\left[\norm{g^k - \nabla f(x^*)}^2\right] 
        &=& 
        \EE\left[ \norm{\nabla f_j(\phi_j^{k+1}) - \nabla f_j(\phi_j^k) + \frac{1}{n}\sum\limits_{i=1}^n\nabla f_i(\phi_i^k) - \nabla f(x^*)}^2\right]
        \\
        &=&
         \EE\left[\norm{\nabla f_j(x^k) - \nabla f_j(x^*) + \nabla f_j(x^*) - \nabla f_j(\phi_j^k) + \frac{1}{n}\sum\limits_{i=1}^n\nabla f_i(\phi_i^k) - \nabla f(x^*) }^2 \right]
         \\
        &\overset{\eqref{eq:a_b_norm_squared}}{\le}& 
        2\EE\left[\norm{\nabla f_j(x^k) - \nabla f_j(x^*)}^2\mid x^k\right]
        \\
        &&\quad + 2\EE\left[\norm{\nabla f_j(x^*) - \nabla f_j(\phi_j^k) - \EE\left[\nabla f_j(x^*) - \nabla f_j(\phi_j^k)\right] }^2\right]
        \\
        &\overset{\eqref{eq:variance_decomposition}+\eqref{eq:norm_diff_grads}}{\le}&
         \frac{4L}{n}\sum\limits_{i=1}^nD_{f_i}(x^k,x^*)+ 2\EE\left[\norm{\nabla f_j(x^*) - \nabla f_j(\phi_j^k)}^2\mid x^k\right]\\
        &=& 4LD_f(x^k,x^*) + 2\underbrace{\frac{1}{n}\sum\limits_{i=1}^n\norm{\nabla f_i(\phi_i^k) - \nabla f_i(x^*)}^2}_{\sigma_k^2}.
    \end{eqnarray*}
    
To proceed with~\eqref{eq:stoch_grad_second_moment_saga2}, we have
    \begin{eqnarray*}
        \EE\left[\sigma_{k+1}^2\right] &=& \frac{1}{n}\sum\limits_{i=1}^n\EE\left[\norm{\nabla f_i(\phi_i^{k+1}) - \nabla f_i(x^*)}^2\right]\\
        &=& \frac{1}{n}\sum\limits_{i=1}^n\left(\frac{n-1}{n}\norm{\nabla f_i(\phi_i^k) - \nabla f_i(x^*)}^2 + \frac{1}{n}\norm{\nabla f_i(x^k) - \nabla f_i(x^*)}^2\right)\\
        &\overset{\eqref{eq:norm_diff_grads}}{\le}& \left(1 - \frac{1}{n}\right)\frac{1}{n}\sum\limits_{i=1}^n\norm{\nabla f_i(\phi_i^k) - \nabla f_i(x^*)}^2\\
        &&\quad + \frac{2L}{n^2}\sum\limits_{i=1}^n D_{f_i}(x^k,x^*)\\
        &=& \left(1 - \frac{1}{n}\right)\sigma_k^2 + \frac{2L}{n}D_f(x^k,x^*).
    \end{eqnarray*}

\subsection{{\tt N-SAGA}} \label{N-SAGA}

\begin{algorithm}[h]
    \caption{Noisy {\tt SAGA} ({\tt N-SAGA})}
    \label{alg:N-SAGA}
    \begin{algorithmic}
        \Require learning rate $\gamma>0$, starting point $x^0\in\R^d$
        \State Set $\psi_j^0 = x^0$ for each $j\in[0]$
        \For{ $k=0,1,2,\ldots$ }
        \State{Sample  $j \in [n]$ uniformly at random and $\zeta$}
        \State{Set $g_j^{k+1} = g_j(x^k,\xi)$ and $g_i^{k+1} = g_i^{k} $ for $i\neq j$}
        \State{$g^k =g_j(x^k,\xi) - g_j^k + \frac{1}{n}\sum\limits_{i=1}^ng_i^k$}
        \State{$x^{k+1} = \prox_{\gamma R}(x^k - \gamma g^k)$}
        \EndFor
    \end{algorithmic}
\end{algorithm}

Note that it can in practice happen that instead of $\nabla f_i(x)$ one can query $g_i(x,\zeta)$ such that $\EE_\xi g_i(\cdot,\xi)=  \nabla f_i(\cdot) $ and $\EE_\xi \norm{g_i(\cdot,\xi)}^2 \leq  \sigma^2$. This leads to a variant of {\tt SAGA} which only uses noisy  estimates of the stochastic gradients $\nabla_i(\cdot)$. We call this variant {\tt N-SAGA} (see Algorithm~\ref{alg:N-SAGA}).

\begin{lemma}\label{lem:saga_inexact1}
    We have
    \begin{equation}\label{eq:stoch_grad_second_moment_saga_noise1}
        \EE\left[\norm{g^k - \nabla f(x^*)}^2\mid x^k\right] \le 4LD_f(x^k,x^*) + 2\sigma_k^2 + 2\sigma^2,
    \end{equation}
    and
        \begin{equation}\label{eq:stoch_grad_second_moment_saga_noise1}
        \EE\left[\sigma_{k+1}^2\mid x^k\right] \le \left(1 - \frac{1}{n}\right)\sigma_k^2 + \frac{2L}{n}D_f(x^k,x^*) + \frac{\sigma^2}{n},
    \end{equation}
    where $\sigma_k^2 \eqdef \frac{1}{n}\sum\limits_{i=1}^n\norm{ g_i^k - \nabla f_i(x^*)}^2$.
\end{lemma}

\begin{corollary}\label{cor:N-SAGA}
Let $\gamma = \frac{1}{6L}$. Then, iterates of Algorithm~\ref{alg:N-SAGA} satisfy
\[
\EE{V^k}\leq \left( 1- \min\left( \frac{\mu}{6L}, \frac{1}{2n} \right)\right)^k V^0 + \frac{\sigma^2 }{ L \min(\mu, \frac{3 L }{n})}.
\]
\end{corollary}

Analogous results can be obtained for {\tt L-SVRG}.

\subsubsection*{Proof of Lemma~\ref{lem:saga_inexact1}}

    Let all expectations be conditioned on $x^k$. By definition of $g^k$ we have
    \begin{eqnarray*}
        \EE\left[\norm{g^k - \nabla f(x^*)}^2\right] 
        &\le& 
        \EE\left[\norm{g_j(x^k,\zeta)- g_j^k + \frac{1}{n}\sum\limits_{i=1}^n g_i^k - \nabla f(x^*)}^2\right]
        \\
        &=&
         \EE\left[\norm{g_j(x^k,\zeta) - \nabla f_j(x^*) + \nabla f_j(x^*) - g_j^k + \frac{1}{n}\sum\limits_{i=1}^n g_i^k  - \nabla f(x^*)}^2\right]
         \\
        &\overset{\eqref{eq:a_b_norm_squared}}{\le}& 
        2\EE\left[\norm{g_j(x^k,\zeta) - \nabla f_j(x^*)}^2\right]
        \\
        &&\quad 
        + 2\EE\left[\norm{\nabla f_j(x^*) -g_j^k - \EE\left[\nabla f_j(x^*) - g_j^k\right] }^2\right]
        \\
        &\overset{\eqref{eq:variance_decomposition}}{\le}&
         2\EE\left[\norm{g_j(x^k,\zeta) - \nabla f_j(x^*)}^2\right] + 2\EE\left[\norm{\nabla f_j(x^*) - g_j^k }^2\right]
         \\
        &=& 
        2\EE\left[\norm{g_j(x^k,\zeta) - \nabla f_j(x^*)}^2\right] + 2\underbrace{\frac{1}{n}\sum\limits_{i=1}^n\norm{g_i^k - \nabla f_i(x^*)}^2}_{\sigma_k^2} 
        \\
        &\overset{\eqref{eq:variance_decomposition}}{\le}& 
                 2\EE\left[ \norm{ \nabla f_j(x^k) - \nabla f_j(x^*)}^2\right] + 2\sigma^2 + 2\sigma_k^2
                \\
                &\overset{\eqref{eq:norm_diff_grads}}{\le}&
                4LD_f(x^k,x^*) + 2\sigma_k^2 + 2\sigma^2
    \end{eqnarray*}

    For the second inequality, we have
    \begin{eqnarray*}
        \EE\left[\sigma_{k+1}^2\right] &=& \frac{1}{n}\sum\limits_{i=1}^n\EE\left[\norm{g_i^{k+1} - \nabla f_i(x^*)}^2\right]
        \\
        &=& 
        \frac{1}{n}\sum\limits_{i=1}^n\left(\frac{n-1}{n}\norm{g_i^k  - \nabla f_i(x^*)}^2 + \frac{1}{n}\EE\left[ \norm{ g_i(x^k, \zeta) - \nabla f_i(x^*)}^2\right]\right)
        \\
        &\leq&
                \frac{1}{n}\sum\limits_{i=1}^n\left(\frac{n-1}{n}\norm{g_i^k  - \nabla f_i(x^*)}^2 + \frac{1}{n} \norm{ \nabla f_i(x^k) - \nabla f_i(x^*)}^2 + \frac{\sigma^2}{n}\right)
        \\
        &\overset{\eqref{eq:norm_diff_grads}}{\le}&
        \left(1 - \frac{1}{n}\right)\sigma_k^2 + \frac{2L}{n}D_f(x^k,x^*)+ \frac{\sigma^2}{n}.
    \end{eqnarray*}

\subsection{{\tt SEGA}}\label{sec:sega}

\begin{algorithm}[h]
    \caption{{\tt SEGA} \cite{hanzely2018sega}}
    \label{alg:SEGA}
    \begin{algorithmic}
        \Require learning rate $\gamma>0$, starting point $x^0\in\R^d$
        \State Set $h^0 = 0$
        \For{ $k=0,1,2,\ldots$ }
        \State{Sample  $j \in [d]$ uniformly at random}
        \State{Set $h^{k+1} = h^{k} +e_i( \nabla_i f(x^k) - h^{k}_i)$}
        \State{$g^k = de_i (\nabla_i f(x^k) - h_i^k) + h^k$}
        \State{$x^{k+1} = \prox_{\gamma R}(x^k - \gamma g^k)$}
        \EndFor
    \end{algorithmic}
\end{algorithm}

We show that the framework recovers the simplest version of {\tt SEGA} (i.e., setup from Theorem D1 from~\cite{hanzely2018sega}) in the proximal setting\footnote{General version for arbitrary gradient sketches instead of partial derivatives can be recovered as well, however, we omit it for simplicity}. 

\begin{lemma} (Consequence of Lemmas A.3., A.4. from~\cite{hanzely2018sega})
We have
\[
\EE\left[\norm{g^{k}-\nabla f(x^*) \mid x^k}^{2}\right] \leq 2d\norm{\nabla f\left(x^{k}\right)-\nabla f(x^*)}^{2}+2d\sigma_k^2
\]
and
\[
\EE
\left[\sigma_{k+1}^2 \mid x^k \right]= \left(1-\frac1d\right)\sigma_k^2+\frac1d \norm{\nabla f\left(x^{k}\right)-\nabla f(x^*)}^{2},
\]
where $\sigma_k^2 \eqdef \norm{h^{k}-\nabla f(x^*)}^2$.
\end{lemma}

Given that we have from convexity and smoothness $\norm{\nabla f(x^{k})-\nabla f(x^*)}^{2} \leq 2L D_f(x^k,x^*)$, Assumption~\ref{as:general_stoch_gradient} holds the parameter choice as per Table~\ref{tbl:special_cases-parameters}. Setting further $M = 4d^2$, we get the next corollary.
\begin{corollary}\label{cor:sega}
{\tt SEGA} with $\gamma =\frac{1}{6dL} $ satisfies
\[
\EE V^k \leq \left( 1- \frac{\mu}{6dL}\right)^kV^0.
\] 
\end{corollary}

\subsection{{\tt N-SEGA}} \label{N-SEGA}

\begin{algorithm}[h]
    \caption{Noisy {\tt SEGA} ({\tt N-SEGA})}
    \label{alg:N-SEGA}
    \begin{algorithmic}
        \Require learning rate $\gamma>0$, starting point $x^0\in\R^d$
        \State Set $h^0 = 0$
        \For{ $k=0,1,2,\ldots$ }
        \State{Sample  $i \in [d]$ uniformly at random and sample $\xi$}
        \State{Set $h^{k+1} = h^{k} +e_i( g_i(x,\xi)- h^{k}_i)$}
        \State{$g^k = de_i (g_i(x,\xi) - h_i^k) +  h^k$}
        \State{$x^{k+1} = x^k - \gamma g^k$}
        \EndFor
    \end{algorithmic}
\end{algorithm}

Here we assume that $g_i(x,\zeta)$ is a noisy estimate of the partial derivative $\nabla_i f(x)$ such that $\EE_\zeta g_i(x,\zeta) = \nabla_i f(x)$ and $\EE_\zeta | g_i(x,\zeta) - \nabla_i f(x)|^2 \leq \frac{\sigma^2}{d}$. 

\begin{lemma} \label{lem:sega_noise}
The following inequalities hold:
\[
\EE\left[\norm{g^{k}-\nabla f(x^*)}^{2}\right]
 \leq 
 4dLD_f(x^k,x^*)+2d\sigma_k^2 +2d\sigma^2,
\]
\[
\EE
\left[\sigma_{k+1}^2\right] \leq \left(1-\frac1d\right)\sigma_k^2+\frac{2L}{d}D_f(x^k,x^*) + \frac{\sigma^2}{d}, 
\]
where $\sigma_k^2 = \norm{h^{k}-\nabla f(x^*)}^2$.
\end{lemma}

\begin{corollary}\label{cor:N-SEGA}
Let $\gamma = \frac{1}{6Ld}$. Applying Theorem~\ref{thm:main_gsgm} with $M = 4d^2$, iterates of Algorithm~\ref{alg:N-SEGA} satisfy
\[
\EE{V^k}\leq \left( 1- \frac{\mu}{6dL}\right)^k V^0 + \frac{\sigma^2 }{ L \mu}.
\]
\end{corollary}

\subsubsection*{Proof of Lemma~\ref{lem:sega_noise}}
 
Let all expectations be conditioned on $x^k$. 
For the first bound, we write \[g^k  - \nabla f(x^*)= \underbrace{h^k - \nabla f(x^*)- d  h^k_i e_{i}+d\nabla_i f(x^*)e_i}_{a} +   \underbrace{ d g_i(x^k,\xi)e_i- d\nabla_i f(x^*)e_i}_{b}.\]
 Let us bound the expectation of each term individually. The first term can be bounded as
\begin{eqnarray*}
\EE{\norm{a}^2} &=& \EE{\norm{ \left(\mI - d e_{i} e_{i}^\top \right) (h^k - \nabla f(x^*))  }_{2}^2}\\
&=& (d-1)\norm{h^k- \nabla f(x^*)}^2\\
&\leq& d\norm{h^k- \nabla f(x^*)}^2.
\end{eqnarray*}
The second term can be bounded as
\begin{eqnarray*}
\EE{\norm{b}^2} &=&  \EE_i \EE_{\xi}{\norm{ d g_i(x,\xi)e_i- d\nabla f_i(x^*)e_i}^2}\\
&=& 
 \EE_{i} \EE_{\xi} \norm{ d g_i(x^k,\xi)e_i- d\nabla_i f(x^k)e_i }^2 + \EE_i\norm{ d\nabla_i f(x^k)e_i-d\nabla f_i(x^*)e_i}^2 
\\
&\leq & d\sigma^2 + d  \norm{ \nabla f(x^k)- \nabla f(x^*)}^2 \\
&\leq& d\sigma^2 + 2Ld D_f(x^k,x^*),
\end{eqnarray*}
where in the last step we used $L$--smoothness of $f$. It remains to combine the two bounds.

For the second bound, we have
\begin{eqnarray*}
\EE{\norm{ h^{k+1}  - \nabla f(x^*)}^2} &=& \EE{ \norm{h^k + g_i(x^k,\xi)e_i - h^k_i - \nabla f(x^*) }^2 }\\
 &=&  \EE{ \norm{\left(\mI - e_{i} e_{i}^\top \right)h^k + g_i(x^k,\xi)e_i -\nabla f(x^*)  }^2 }\\
 &=& \EE{\norm{ \left(\mI - e_{i} e_{i}^\top \right) ( h^k - \nabla f(x^*)) }^2} + \EE{\norm{g_i(x^k,\xi)e_i - \nabla_i f(x^*)e_i }^2 }\\
  &=& \left(1-\frac{1}{d}\right) \norm{h^k - \nabla f(x^*)}^2 + \EE{\norm{g_i(x^k,\xi)e_i - \nabla_i f(x^k)e_i }^2 }   \\
  && \qquad + \EE{\norm{\nabla_i f(x^k)e_i  - \nabla_i f(x^*)e_i}^2 } \\
 &=& \left(1-\frac{1}{d}\right) \norm{h^k - \nabla f(x^*)}^2 + \frac{\sigma^2}{d} + \frac{1}{d} \norm{\nabla f(x^k) - \nabla f(x^*)}^2 \\
  &\leq&
   \left(1-\frac{1}{d}\right) \norm{h^k - \nabla f(x^*)}^2 + \frac{\sigma^2}{d} + \frac{2L}{d}D_f(x^k,x^*).
\end{eqnarray*}

\subsection{{\tt SVRG}} \label{sec:svrg}

\begin{algorithm}[h]
    \caption{{\tt SVRG} \cite{SVRG}}
    \label{alg:SVRG}
    \begin{algorithmic}
        \Require learning rate $\gamma>0$, epoch length $m$, starting point $x^0\in\R^d$
        \State  $\phi = x^0$
        \For{ $s=0,1,2,\ldots$ }
        \For{ $k=0,1,2,\ldots, m-1$ }
        \State{Sample  $i \in \{1,\ldots, n\}$ uniformly at random}
        \State{$g^k = \nabla f_i(x^k) - \nabla f_i(\phi) + \nabla f(\phi)$}
        \State{$x^{k+1} = \proxR(x^k - \gamma g^k)$}
        \EndFor
        \State  $\phi = x^0 = \frac1m \sum_{k=1}^m x^k$
        \EndFor
    \end{algorithmic}
\end{algorithm}
Let $\sigma_k^2 \eqdef \frac{1}{n}\sum\limits_{i=1}^n\norm{\nabla f_i(\phi) - \nabla f_i(x^*)}^2$. We will show that Lemma~\ref{lem:iter_dec} recovers per-epoch analysis of {\tt SVRG} in a special case.

\begin{lemma}\label{lem:svrg_lemma_1}
For $k \mod m \neq 0$ we have    
        \begin{equation}\label{eq:svrg1}
        \EE\left[\norm{g^k -\nabla f(x^*)}^2\mid x^k\right] \le 4LD_f(x^k,x^*) + 2\sigma_k^2
    \end{equation}
    and
    \begin{equation}\label{eq:svrg3}
        \EE\left[\sigma_{k+1}^2\mid x^k\right] = \sigma_{k+1}^2 = \sigma_k^2.
    \end{equation}
\end{lemma}
\begin{proof}
The proof of~\eqref{eq:svrg1} is identical to the proof of~\eqref{eq:stoch_grad_second_moment_saga1}. Next,~\eqref{eq:svrg3} holds since $\sigma_k$ does not depend on $k$. 
\end{proof}

Thus, Assumption~\ref{as:general_stoch_gradient} holds with parameter choice as per Table~\ref{tbl:special_cases-parameters} and Lemma~\ref{lem:iter_dec} implies the next corollary.
\begin{corollary}\label{cor:svrg}
\begin{equation}\label{eq:svrg_iter_dec}
 \EE\norm{x^{k+1}-x^*}^2 + \gamma (1-2\gamma L)\EE D_{f}(x^k,x^*) \leq 
        (1-\gamma\mu)\EE\norm{x^k - x^*}^2 +2\gamma^2\EE\sigma_k^2.
\end{equation}
\end{corollary}

\subsubsection*{Recovering SVRG rate}
Summing~\eqref{eq:svrg_iter_dec} for $k=0, \dots, m-1$ using $\sigma_k = \sigma_0$ we arrive at
\begin{eqnarray*}
\EE\norm{x^{m}-x^*}^2+ \sum_{k=1}^m \gamma (1-2\gamma L)\EE D_{f}(x^k,x^*)
&\leq& (1-\gamma\mu)\EE\norm{x^0 - x^*}^2 + 2m\gamma^2\EE\sigma_0^2 \\
&\leq &2 \left( \mu^{-1} +  2m\gamma^2   L\right) D_f(x^0,x^*) \;.
\end{eqnarray*}

Since $D_f$ is convex in the first argument, we have 
\[
m \gamma (1-2\gamma L)D_{f}\left( \frac1m \sum_{k=1}^m x^k,x^*\right) \leq \norm{x^{m}-x^*}^2+ \sum_{k=1}^m \gamma (1-2\gamma L)D_{f}(x^k,x^*)  
\]
and thus
\[
D_{f}\left( \frac1m \sum_{k=1}^m x^k,x^*\right) \leq \frac{ 2 \left( \mu^{-1} +  2m\gamma^2   L\right)}
{m \gamma (1-2\gamma L)} D_f(x^0,x^*) ,
 \]
which recovers rate from Theorem 1 in~\cite{SVRG}.

\subsection{{\tt L-SVRG}}\label{sec:L-SVRG}

In this section we show that our approach also covers {\tt L-SVRG} analysis from \cite{hofmann2015variance, Loopless} (see Algorithm~\ref{alg:L-SVRG}). Consider the finite-sum minimization problem
\begin{equation}\label{eq:main_l-svrg}
    f(x) = \frac{1}{n}\sum\limits_{i=1}^n f_i(x) + R(x),
\end{equation}
where $f_i$ is $L$-smooth for each $i$ and $f$ is $\mu$-strongly convex.

\begin{algorithm}[h]
    \caption{{\tt L-SVRG} \cite{hofmann2015variance, Loopless}}
    \label{alg:L-SVRG}
    \begin{algorithmic}
        \Require learning rate $\gamma>0$, probability $p\in (0,1]$, starting point $x^0\in\R^d$
        \State $w^0 = x^0$
        \For{ $k=0,1,2,\ldots$ }
        \State{Sample  $i \in \{1,\ldots, n\}$ uniformly at random}
        \State{$g^k = \nabla f_i(x^k) - \nabla f_i(w^k) + \nabla f(w^k)$}
        \State{$x^{k+1} = x^k - \gamma g^k$}
        \State{$w^{k+1} = \begin{cases}
            x^{k}& \text{with probability } p\\
            w^k& \text{with probability } 1-p
            \end{cases}$
        }
        \EndFor
    \end{algorithmic}
\end{algorithm}
Note that the gradient estimator is again unbiased, i.e. $\EE\left[g^k\mid x^k\right] = \nabla f(x^k)$. Next, Lemma~\ref{lem:l-svrg} provides with the remaining constants for Assumption~\ref{as:general_stoch_gradient}. The corresponding choice is stated in Table~\ref{tbl:special_cases-parameters}.

\begin{lemma}[Lemma~4.2 and Lemma~4.3 from~\cite{Loopless} extended to prox setup]\label{lem:l-svrg}
    We have
    \begin{equation}\label{eq:lemma4.2_l-svrg}
        \EE\left[\norm{g^k - \nabla f(x^*)}^2\mid x^k\right] \le 4LD_f(x^k,x^*) + 2\sigma_k^2
    \end{equation}
    and 
        \begin{equation}\label{eq:lemma4.3_l-svrg}
        \EE\left[\sigma_{k+1}^2\mid x^k\right] \le (1-p)\sigma_k^2 + 2LpD_f(x^k, x^*),
    \end{equation}
    where $\sigma_k^2 \eqdef \frac{1}{n}\sum\limits_{i=1}^n\norm{\nabla f_i(w^k) - \nabla f_i(x^*)}^2$.
\end{lemma}

Next, applying Theorem~\ref{thm:main_gsgm} on Algorithm~\ref{alg:L-SVRG} with $M=\frac{4}{p}$ we get Corollary~\ref{cor:recover_l-svrg_rate}. 

\begin{corollary}\label{cor:recover_l-svrg_rate}
    {\tt L-SVRG} with $\gamma = \frac{1}{6L}$ satisfies
    \begin{equation}\label{eq:recover_l-svrg_rate}
        \EE V^k \le \left(1-\min\left\{\frac{\mu}{6L}, \frac{p}{2}\right\}\right)^kV^0.
    \end{equation}
\end{corollary}

\subsubsection*{Proof of Lemma~\ref{lem:l-svrg}}\label{sec:proofs_l-svrg}

    Let all expectations be conditioned on $x^k$. Using definition of $g^k$
    \begin{eqnarray*}
        \EE\left[\norm{g^k - \nabla f(x^*)}^2\right] &\overset{\text{Alg.}~\ref{alg:L-SVRG}}{=}&
        \EE\left[\norm{
            \nabla f_i(x^k) - \nabla f_i(x^*) + \nabla f_i(x^*) - \nabla f_i(w^k) + \nabla f(w^k)
        - \nabla f(x^*)}^2\right]
        \\
        &\overset{\eqref{eq:a_b_norm_squared}}{\leq}&
        2\EE\left[\norm{\nabla f_i(x^k) - \nabla f_i(x^*)}^2\right]\\
        &&\quad+ 
        2\EE\left[\norm{\nabla f_i(x^*) - \nabla f_i(w^k) - \EE\left[\nabla f_i(x^*) - \nabla f_i(w^k)\mid x^k\right]}^2\right]\\
        &\overset{\eqref{eq:norm_diff_grads},\eqref{eq:variance_decomposition}}{\leq}&
        4LD_f(x^k,x^*) + 2 \EE\left[\norm{\nabla f_i(w^k) - \nabla f_i(x^*)}^2\right] \\
        &=&
        4LD_f(x^k,x^*) + 2 \sigma_k^2.
    \end{eqnarray*}
    For the second bound, we shall have
        \begin{eqnarray*}
        \EE\left[\sigma_{k+1}^2\right] &\overset{\text{Alg.}~\ref{alg:L-SVRG}}{=}& (1-p) \sigma_k^2 +  \frac{p}{n} \sum\limits_{i=1}^{n} \norm{\nabla f(x^k) - \nabla f(x^*)}^2\\
        &\overset{\eqref{eq:norm_diff_grads}}{\leq}& (1-p) \sigma_k^2 + 2Lp D_f(x^k,x^*).
    \end{eqnarray*}

\subsection{{\tt DIANA}}\label{sec:diana}
In this section we consider a distributed setup where each function $f_i$ from~\eqref{eq:f_sum} is owned by $i$-th machine (thus, we have all together $n$ machines). 

We show that our approach covers the analysis of {\tt DIANA} from \cite{mishchenko2019distributed, horvath2019stochastic}. {\tt DIANA} is a specific algorithm for distributed optimization with { \em quantization} -- lossy compression of gradient updates, which reduces the communication between the server and workers\footnote{It is a well-known problem in distributed optimization that the communication between machines often takes more time than actual computation.}. 

In particular, {\tt DIANA} quantizes gradient differences instead of the actual gradients. This trick allows for the linear convergence to the optimum once the full gradients are evaluated on each machine, unlike other popular quantization methods such as {\tt QSGD} \cite{alistarh2017qsgd} or {\tt TernGrad} \cite{wen2017terngrad}. In this case, {\tt DIANA} behaves as variance reduced method -- it reduces a variance that was injected due to the quantization. However, {\tt DIANA} also allows for evaluation of stochastic gradients on each machine, as we shall further see.

First of all, we introduce the notion of quantization operator.

\begin{definition}[Quantization]\label{def:quantization}
    We say that $\hat \Delta$ is a \textit{quantization} of vector $\Delta\in\R^d$ and write $\hat \Delta \sim {\rm Q}(\Delta)$ if
    \begin{equation}\label{eq:quantization}
        \EE\hat\Delta = \Delta, \qquad \EE\norm{\hat \Delta - \Delta}^2 \le \omega \norm{\Delta}^2
    \end{equation}
    for some $\omega > 0$.
\end{definition}

\begin{algorithm}[t]
   \caption{{\tt DIANA} \cite{mishchenko2019distributed, horvath2019stochastic}}
   \label{alg:diana}
\begin{algorithmic}[1]
   \Require learning rates $\alpha>0$ and $\gamma>0$, initial vectors $x^0, h_1^0,\dotsc, h_n^0 \in \R^d$ and $h^0 = \frac{1}{n}\sum_{i=1}^n h_i^0$
   \For{$k=0,1,\dotsc$}
       \State Broadcast $x^{k}$ to all workers
        \For{$i=1,\dotsc,n$ in parallel}
            \State Sample $g^{k}_i$ such that $\EE [g^k_i \;|\; x^k]  =\nabla f_i(x^k)$ 
            \State $\Delta^k_i = g^k_i - h^k_i$
            \State Sample $\hat \Delta^k_i \sim {\rm Q}(\Delta^k_i)$
            \State $h_i^{k+1} = h_i^k + \alpha \hat \Delta_i^k$
            \State $\hat g_i^k = h_i^k + \hat \Delta_i^k$
        \EndFor
       \State $\hat \Delta^k = \frac{1}{n}\sum_{i=1}^n \hat \Delta_i^k$
       \State $ g^k = \frac{1}{n}\sum_{i=1}^n \hat g_i^k = h^k + \hat \Delta^k $
        \State $x^{k+1} = \proxR\left(x^k - \gamma  g^k \right)$
        \State $h^{k+1}  = \tfrac{1}{n}\sum_{i=1}^n h_i^{k+1} = h^k + \alpha \hat \Delta^k$
   \EndFor
\end{algorithmic}
\end{algorithm}

The aforementioned method is applied to solve problem \eqref{eq:problem_gen}+\eqref{eq:f_sum} where each $f_i$ is convex and $L$-smooth and $f$ is $\mu$-strongly convex. 

\begin{lemma}[Lemma 1 and consequence of Lemma 2 from \cite{horvath2019stochastic}]\label{lem:lemma1_diana}
    Suppose that $\alpha \le \frac{1}{1+\omega}$. For all iterations $k\ge 0$ of Algorithm~\ref{alg:diana} it holds
    \begin{eqnarray}
        \EE\left[ g^k\mid x^k\right] &=& \nabla f(x^k), \label{eq:unbiased_diana}\\
        \EE\left[\norm{g^k - \nabla f(x^*)}^2\mid x^k\right] &\le& \left(1+\frac{2\omega}{n}\right)\frac{1}{n}\sum\limits_{i=1}^n\norm{\nabla f_i(x^k) - \nabla f_i(x^*)}^2\notag\\
        &&\quad + \frac{2\omega\sigma_k^2}{n} + \frac{(1+\omega)\sigma^2}{n},\label{eq:second_moment_diana}
        \\
    \EE\left[\sigma_{k+1}^2 \mid x^k\right] &\le& (1-\alpha)\sigma_k^2 + \frac{\alpha}{n}\sum\limits_{i=1}^n\norm{\nabla f_i(x^k) - \nabla f_i(x^*)}^2 + \alpha \sigma^2.\label{eq:h_i_k_sec_moment_diana}
    \end{eqnarray}
    where $\sigma_k^2  = \frac{1}{n}\sum\limits_{i=1}^n\norm{h_i^k -  \nabla f_i(x^*)}^2$ and $\sigma^2$ is such that $\frac{1}{n}\sum\limits_{i=1}^n\EE\left[\norm{g_i^k - \nabla f_i(x^k)}^2\mid x^k\right]\le \sigma^2$.
\end{lemma}
Bounding further $\frac1n \sum_{i=1}^n\norm{\nabla f_i(x^k) - \nabla f_i(x^*)}^2 \leq 2L D_{f}(x^k,x^*)$ in the above Lemma, we see that Assumption~\ref{as:general_stoch_gradient} as per Table~\ref{tbl:special_cases-parameters} is valid. Thus, as a special case of Theorem~\ref{thm:main_gsgm}, we obtain the following corollary.

\begin{corollary}\label{cor:main_diana}
    Assume that $f_i$ is convex and $L$-smooth for all $i\in[n]$ and $f$ is $\mu$ strongly convex, $\alpha \le \frac{1}{\omega+1}$, $\gamma \le \frac{1}{\left(1+\frac{2\omega}{n}\right)L + ML\alpha}$ where $M > \frac{2\omega}{n\alpha}$. Then the iterates of {\tt DIANA} satisfy
    \begin{equation}\label{eq:convergence_diana}
        \EE\left[V^k\right] \le \max\left\{(1-\gamma\mu)^k, \left(1 + \frac{2\omega}{nM} - \alpha\right)^k\right\}V^0 + \frac{\left(\frac{1+\omega}{n} + M\alpha\right)\sigma^2\gamma^2}{\min\left\{\gamma\mu, \alpha - \frac{2\omega}{nM}\right\}},
    \end{equation}
    where the Lyapunov function $V^k$ is defined by $V^k \eqdef \norm{x^k - x^*}^2 + M\gamma^2\sigma_k^2$. For the particular choice $\alpha = \frac{1}{\omega+1}$, $M = \frac{4\omega(\omega+1)}{n}$, $\gamma = \frac{1}{\left(1 + \frac{6\omega}{n}\right)L}$, then {\tt DIANA} converges to a solution neighborhood and the leading iteration complexity term is
    \begin{equation}\label{eq:diana_leading_term}
        \max\left\{\frac{1}{\gamma\mu}, \frac{1}{\alpha - \frac{2\omega}{nM}}\right\} = \max\left\{\kappa + \kappa \frac{6\omega}{n}, 2(\omega+1)\right\},
    \end{equation}
    where $\kappa = \frac{L}{\mu}$.
\end{corollary}

As mentioned, once the full (deterministic) gradients are evaluated on each machine, {\tt DIANA} converges linearly to the exact optimum. In particular, in such case we have $\sigma^2 = 0$. Corollary~\ref{cor:main_diana_special_case} states the result in the case when $n=1$, i.e. there is only a single node~\footnote{node = machine}. For completeness, we present the mentioned simple case of {\tt DIANA} as Algorithm~\ref{alg:diana_case}.

\begin{algorithm}[t]
   \caption{{\tt DIANA}: 1 node $\&$ exact gradients \cite{mishchenko2019distributed, horvath2019stochastic}}
   \label{alg:diana_case}
\begin{algorithmic}[1]
   \Require learning rates $\alpha>0$ and $\gamma>0$, initial vectors $x^0, h^0 \in \R^d$
   \For{$k=0,1,\dotsc$}
       \State $\Delta^k = \nabla f(x^k) - h^k$
       \State Sample $\hat \Delta^k \sim {\rm Q}(\Delta^k)$
       \State $h^{k+1} = h^k + \alpha \hat \Delta^k$
       \State $g^k = h^k + \hat \Delta^k$
       \State $x^{k+1} = \proxR\left(x^k - \gamma  g^k \right)$
   \EndFor
\end{algorithmic}
\end{algorithm}

\begin{corollary}\label{cor:main_diana_special_case}
    Assume that $f_i$ is $\mu$-strongly convex and $L$-smooth for all $i\in[n]$, $\alpha \le \frac{1}{\omega+1}$, $\gamma \le \frac{1}{\left(1+2\omega\right)L + ML\alpha}$ where $M > \frac{2\omega}{\alpha}$. Then the stochastic gradient $\hat g^k$ and the objective function $f$ satisfy Assumption~\ref{as:general_stoch_gradient} with $A = \left(1+2\omega\right)L, B = 2\omega, \sigma_k^2 = \norm{h^k - h^*}^2, \rho = \alpha, C = L\alpha, D_1 = 0, D_2 = 0$ and 
    \begin{equation}\label{eq:convergence_diana_special_case}
        \EE\left[V^k\right] \le \max\left\{(1-\gamma\mu)^k, \left(1 + \frac{2\omega}{M} - \alpha\right)^k\right\}V^0,
    \end{equation}
    where the Lyapunov function $V^k$ is defined by $V^k \eqdef \norm{x^k - x^*}^2 + M\gamma^2\sigma_k^2$. For the particular choice $\alpha = \frac{1}{\omega+1}$, $M = 4\omega(\omega+1)$, $\gamma = \frac{1}{\left(1 + 6\omega\right)L}$ the leading term in the iteration complexity bound is
    \begin{equation}\label{eq:diana_leading_term_special_case}
        \max\left\{\frac{1}{\gamma\mu}, \frac{1}{\alpha - \frac{2\omega}{M}}\right\} = \max\left\{\kappa + 6\kappa\omega, 2(\omega+1)\right\},
    \end{equation}
    where $\kappa = \frac{L}{\mu}$.
\end{corollary}

\subsection{{\tt Q-SGD-SR}}\label{Q-SGD-AS}

In this section, we consider a quantized version of {\tt SGD-SR}. 

\begin{algorithm}[h]
    \caption{{\tt Q-SGD-SR}}
    \label{alg:qsgdas}
    \begin{algorithmic}
        \Require learning rate $\gamma>0$, starting point $x^0\in\R^d$, distribution $\cD$ over $\xi \in\R^n$ such that $\ED{\xi}$ is vector of ones
        \For{ $k=0,1,2,\ldots$ }
        \State{Sample $\xi \sim \cD$}
        \State{$g^k \sim {\rm Q}(\nabla f_\xi (x^k))$}
        \State{$x^{k+1} = \proxR(x^k - \gamma g^k)$}
        \EndFor
    \end{algorithmic}
\end{algorithm}

\begin{lemma}[Generalization of Lemma~2.4, \cite{SGD_AS}]\label{lem:exp_smoothness_grad_up_bound_q-sgd-as}
    If $(f,\cD)\sim ES(\cL)$, then
    \begin{equation}\label{eq:exp_smoothness_grad_up_bound_sgd-as}
        \ED{\norm{g^k - \nabla f(x^*)}^2} \le 4\cL(1+\omega)D_f(x^k,x^*) + 2\sigma^2(1+\omega).
    \end{equation}
    where $\sigma^2 \eqdef \ED{\norm{\nabla f_\xi(x^*)}^2}$.
\end{lemma}

A direct consequence of Theorem~\ref{thm:main_gsgm} in this setup is Corollary~\ref{cor:recover_q-sgd-as_rate}. 

\begin{corollary}\label{cor:recover_q-sgd-as_rate}
    Assume that $f(x)$ is $\mu$-strongly quasi-convex and $(f,\cD)\sim ES(\cL)$. Then {\tt Q-SGD-SR} with $\gamma^k \equiv\gamma \le \frac{1}{2(1+\omega)\cL}$ satisfies
    \begin{equation}\label{eq:recover_q-sgd-as_rate}
        \EE\left[\norm{x^k - x^*}^2\right] \le (1-\gamma\mu)^k\norm{x^0-x^*}^2 + \frac{2\gamma(1+\omega)\sigma^2}{\mu}.
    \end{equation}
\end{corollary}

\subsubsection*{Proof of Lemma~\ref{lem:exp_smoothness_grad_up_bound_q-sgd-as}}
In this proof all expectations are conditioned on $x^k$. First of all, from Lemma~\ref{lem:exp_smoothness_grad_up_bound_sgd-as} we have
\begin{eqnarray*}
    \ED{\norm{\nabla f_\xi(x^k) - \nabla f(x^*)}^2} \le 4\cL D_f(x^k,x^*) + 2\sigma^2.
\end{eqnarray*}
The remaining step is to understand how quantization of $\nabla f_\xi(x^k)$ changes the above inequality if we put $g^k\sim {\rm Q}(\nabla f_\xi(x^k))$ instead of $\nabla f_\xi(x^k)$. Let us denote mathematical expectation with respect randomness coming from quantization by $\EE_Q\left[\cdot\right]$. Using tower property of mathematical expectation we get
\begin{eqnarray*}
    \EE\left[\|g^k - \nabla f(x^*)\|^2\right] &=& \EE_{\cD}\left[\EE_Q\|g^k - \nabla f(x^*)\|^2\right]\\
    &\overset{\eqref{eq:variance_decomposition}}{=}& \EE\left[\|g^k - \nabla f_\xi(x^k)\|^2\right] + \EE\left[\|\nabla f_\xi(x^k) - \nabla f(x^*)\|^2\right]\\
    &\overset{\eqref{eq:exp_smoothness_grad_up_bound_sgd-as}}{\le}&  \EE\left[\|g^k - \nabla f_\xi(x^k)\|^2\right] + 4\cL D_f(x^k,x^*) + 2\sigma^2.
\end{eqnarray*} 
Next, we estimate the first term in the last row of the previous inequality
\begin{eqnarray*}
    \EE\left[\|g^k - \nabla f_\xi(x^k)\|^2\right] &\overset{\eqref{eq:quantization}}{\le}& \omega\EE\left[\|\nabla f_\xi(x^k)\|^2\right]\\
    &\overset{\eqref{eq:a_b_norm_squared}}{\le}& 2\omega\EE\left[\|\nabla f_\xi(x^k) - \nabla f_\xi(x^*)\|^2\right] + 2\omega\EE\left[\|\nabla f_\xi(x^*)\|^2\right]\\
    &\le& 4\omega\cL D_f(x^k,x^*) + 2\omega\sigma^2.
\end{eqnarray*}
Putting all together we get the result.

\subsection{{\tt VR-DIANA}} \label{sec:VR-DIANA}
Corollary~\ref{cor:main_diana} shows that once each machine evaluates a stochastic gradient instead of the full gradient, {\tt DIANA} converges linearly only to a certain neighborhood. In contrast, {\tt VR-DIANA}~\cite{horvath2019stochastic} uses a variance reduction trick within each machine, which enables linear convergence to the exact solution. In this section, we show that our approach recovers {\tt VR-DIANA} as well. 

\begin{algorithm}[t]
   \caption{{\tt VR-DIANA} based on L-SVRG (Variant 1), SAGA (Variant 2), \cite{horvath2019stochastic}}
   \label{alg:vr-diana}
\begin{algorithmic}[1]
        \Require{learning rates $\alpha > 0$ and $\gamma > 0$, initial vectors $x^0, h_{1}^0, \dots, h_{n}^0$, $h^0 = \frac{1}{n}\sum_{i=1}^n h_i^0$}
        \For{$k = 0,1,\ldots$}
        \State Sample random 
            $
                u^k = \begin{cases}
                    1,& \text{with probability } \frac{1}{m}\\
                    0,& \text{with probability } 1 - \frac{1}{m}\\
                \end{cases}
            $ \Comment{only for Variant 1}
        \State Broadcast $x^k$, $u^k$ to all workers\;
            \For{$i = 1, \ldots, n$ in parallel} \Comment{Worker side}
            \State Pick random $j_i^k \sim_{\rm u.a.r.} [m]$\;
            \State $\mu_i^k = \frac{1}{m} \sum\limits_{j=1}^{m} \nabla f_{ij}(w_{ij}^k)$\label{ln:mu} \;
            \State $g_i^k = \nabla f_{ij_i^k}(x^k) - \nabla f_{ij_i^k}(w_{ij_i^k}^k) + \mu_i^k$\;
            \State $\hat{\Delta}_i^k = Q(g_i^k - h_i^k)$\;
            \State $h_i^{k+1} = h_i^k + \alpha \hat{\Delta}_i^k$\;
                \For{$j = 1, \ldots, m$}
                    \State
                    $
                    w_{ij}^{k+1} =
                    \begin{cases}
                        x^k, & \text{if } u^k = 1 \\
                        w_{ij}^k, &\text{if } u^k = 0\\
                    \end{cases}
                    $ \Comment{Variant 1 (L-SVRG): update epoch gradient if $u^k = 1$}
                    \State
                    $
                    w_{ij}^{k+1} =
                    \begin{cases}
                    x^k, & j = j_i^k\\
                    w_{ij}^k, & j \neq j_i^k\\
                    \end{cases}
                    $ \Comment{Variant 2 (SAGA): update gradient table}
                \EndFor
            \EndFor
            \State $h^{k+1} \! = \! h^k \!+\! \frac{\alpha}{n} \displaystyle \sum_{i=1}^n \hat{\Delta}_i^k$ \Comment{Gather quantized updates} 
            \State $g^k = \frac{1}{n}\sum\limits_{i=1}^{n} (\hat{\Delta}_i^k + h_i^k)$\;
            \State $x^{k+1} = x^k - \gamma g^k$\;
        \EndFor

\end{algorithmic}  
\end{algorithm}

The aforementioned method is applied to solve problem \eqref{eq:problem_gen}+\eqref{eq:f_sum} where each $f_i$ is also of a finite sum structure, as in \eqref{eq:f_i_sum}, with  each $f_{ij}(x)$ being convex and $L$-smooth, and $f_i(x)$ being $\mu$-strongly convex. Note that $\nabla f(x^*) = 0$ and, in particular, $D_f(x,x^*) = f(x) - f(x^*)$ since the problem is considered without regularization.

\begin{lemma}[Lemmas 3, 5, 6 and 7 from \cite{horvath2019stochastic}]\label{lemmas_vr_diana}
    Let $\alpha \le \frac{1}{\omega+1}$. Then for all iterates $k\ge 0$ of Algorithm~\ref{alg:vr-diana} the following inequalities hold:
    \begin{eqnarray}
        \EE\left[g^k\mid x^k\right] &=& \nabla f(x^k),\label{eq:unbiased_g_k_vr_diana}\\
        \EE\left[H^{k+1}\mid x^k\right] &\le& \left(1-\alpha\right)H^k + \frac{2\alpha}{m}D^k + 8\alpha Ln\left(f(x^k) - f(x^*)\right),\label{eq:H_k+1_bound_vr_diana}\\
        \EE\left[D^{k+1}\mid x^k\right] &\le& \left(1 - \frac{1}{m}\right)D^k + 2Ln\left(f(x^k) - f(x^*)\right),\label{eq:D_k+1_bound_vr_diana}\\
        \EE\left[\norm{g^k}^2\mid x^k\right] &\le& 2L\left(1+\frac{4\omega + 2}{n}\right)\left(f(x^k)-f(x^*)\right) + \frac{2\omega}{n^2}\frac{D^k}{m} + \frac{2(\omega+1)}{n^2}H^k,\label{eq:second_moment_g_k_vr_diana}
    \end{eqnarray}
    where $H^k = \sum\limits_{i=1}^n\norm{h_i^k - \nabla f_i(x^*)}^2$ and $D^k = \sum\limits_{i=1}^n\sum\limits_{j=1}^m\norm{\nabla f_{ij}(w_{ij}^k) - \nabla f_{ij}(x^*)}^2$.
\end{lemma}

\begin{corollary}\label{cor:vr_diana_meets_assumption}
    Let $\alpha \le \min\left\{\frac{1}{3m},\frac{1}{\omega+1}\right\}$. Then stochastic gradient $\hat g^k$ (Algorithm~\ref{alg:vr-diana}) and the objective function $f$ satisfy Assumption~\ref{as:general_stoch_gradient} with $A = \left(1+\frac{4\omega + 2}{n}\right)L, B = \frac{2(\omega+1)}{n}, \rho = \alpha, C = L\left(\frac{1}{m}+4\alpha\right), D_1 = 0, D_2 = 0$ and
    \[
        \sigma_k^2 = \frac{H^k}{n} + \frac{D^k}{nm} = \frac{1}{n}\sum\limits_{i=1}^n\norm{h_i^k - \nabla f_i(x^*)}^2 + \frac{1}{nm}\sum\limits_{i=1}^n\sum\limits_{j=1}^m\norm{\nabla f_{ij}(w_{ij}^k) - \nabla f_{ij}(x^*)}^2.
    \]
\end{corollary}
\begin{proof}
    Indeed, \eqref{eq:general_stoch_grad_unbias} holds due to \eqref{eq:unbiased_g_k_vr_diana}. Inequality \eqref{eq:general_stoch_grad_second_moment} follows from \eqref{eq:second_moment_g_k_vr_diana} with $A = \left(1+\frac{4\omega + 2}{n}\right)L, B = \frac{2(\omega+1)}{n}, D_1 = 0, \sigma_k^2 = \frac{H^k}{n} + \frac{D^k}{nm}$ if we take into account that $\frac{2\omega}{n^2}\frac{D^k}{m} + \frac{2(\omega+1)}{n^2}H^k \le \frac{2(\omega+1)}{n}\left(\frac{D^k}{nm} + \frac{H^k}{n}\right)$. Finally, summing inequalities \eqref{eq:H_k+1_bound_vr_diana} and \eqref{eq:D_k+1_bound_vr_diana} and using $\alpha\le\frac{1}{3m}$
    \begin{eqnarray*}
        \EE\left[\sigma_k^2\mid x^k\right] &=& \frac{1}{n}\EE\left[H^{k+1}\mid x^k\right] + \frac{1}{nm}\EE\left[D^{k+1}\mid x^k\right]\\
        &\overset{\eqref{eq:H_k+1_bound_vr_diana}+\eqref{eq:D_k+1_bound_vr_diana}}{\le}& \left(1-\alpha\right)\frac{H^k}{n} + \left(1+2\alpha-\frac{1}{m}\right)\frac{D^k}{nm} + 2L\left(\frac{1}{m}+4\alpha\right)\left(f(x^k)-f(x^*)\right)\\
        &\le& \left(1-\alpha\right)\sigma_k^2 + 2L\left(\frac{1}{m}+4\alpha\right)\left(f(x^k)-f(x^*)\right)
    \end{eqnarray*}
    we get \eqref{eq:gsg_sigma} with $\rho = \alpha, C = L\left(\frac{1}{m}+4\alpha\right), D_2 = 0$.
\end{proof}

\begin{corollary}\label{cor:main_vr_diana}
    Assume that $f_i$ is $\mu$-strongly convex and $f_{ij}$ is convex and $L$-smooth for all $i\in[n], j\in[m]$, $\alpha \le \min\left\{\frac{1}{3m},\frac{1}{\omega+1}\right\}$, $\gamma \le \frac{1}{\left(1+\frac{4\omega + 2}{n}\right)L + ML\left(\frac{1}{m}+4\alpha\right)}$ where $M > \frac{2(\omega+1)}{n\alpha}$. Then the iterates of {\tt VR-DIANA} satisfy
    \begin{equation}\label{eq:convergence_vr_diana}
        \EE\left[V^k\right] \le \max\left\{(1-\gamma\mu)^k, \left(1 + \frac{2(\omega+1)}{nM} - \alpha\right)^k\right\}V^0,
    \end{equation}
    where the Lyapunov function $V^k$ is defined by $V^k \eqdef \norm{x^k - x^*}^2 + M\gamma^2\sigma_k^2$. Further, if we set  $\alpha = \min\left\{\frac{1}{3m},\frac{1}{\omega+1}\right\}$, $M = \frac{4(\omega+1)}{n\alpha}$, $\gamma = \frac{1}{\left(1 + \frac{20\omega+18}{n} + \frac{4\omega+4}{n\alpha m}\right)L}$, then to achieve precision $\EE\left[\norm{x^k-x^*}^2\right] \le \varepsilon V^0$ {\tt VR-DIANA} needs $\cO\left(\max\left\{\kappa+\kappa\frac{\omega+1}{n}+\kappa\frac{(\omega+1)\max\left\{m,\omega+1\right\}}{nm},m,\omega+1\right\}\log\frac{1}{\varepsilon}\right)$ iterations, where $\kappa = \frac{L}{\mu}$.
\end{corollary}
\begin{proof}
    Using Corollary~\ref{cor:vr_diana_meets_assumption} we apply Theorem~\ref{thm:main_gsgm} and get the result.
\end{proof}

\begin{remark}
{\tt VR-DIANA} can be easily extended to the proximal setup in our framework.
\end{remark}

\subsection{{\tt JacSketch}} \label{sec:JacSketch}

In this section, we show that our approach covers the analysis of {\tt JacSketch} from \cite{gower2018stochastic}. {\tt JacSketch} is a generalization of {\tt SAGA} in the following manner. {\tt SAGA} observes every iteration $\nabla f_i(x)$ for random index $i$ and uses it to build both stochastic gradient as well as the control variates on the stochastic gradient in order to progressively decrease variance. In contrast, {\tt JacSketch}  observes every iteration the random sketch of the Jacobian, which is again used to build both stochastic gradient as well as the control variates on the stochastic gradient.

For simplicity, we do not consider proximal setup, since~\cite{gower2018stochastic} does not either.

We first introduce the necessary notation (same as in \cite{gower2018stochastic}). Denote first the Jacobian the objective \begin{equation}\label{eq:jac_def}\Jac(x) \eqdef [\nabla f_1(x), \ldots, \nabla f_n(x)] \in \R^{d\times n}.\end{equation} 
Every iteration of the method, a random sketch of Jacobian $\nabla F(x^k)\mS$ (where $\mS\sim \cD$) is observed. Then, the method builds a variable $\mJ^k$, which is the current Jacobian estimate, updated using so-called sketch and project iteration~\cite{gower2015randomized}:
\[
\mJ^{k+1}  = \mJ^k(\mI - \Proj_{\mS_k}) + \Jac(x^k)\Proj_{\mS_k},
\]

where $\Proj_\mS$ is a projection under $\mW$ norm\footnote{Weighted Frobenius norm of matrix $\mX\in\R^{n\times n}$ with a positive definite weight matrix $\mW\in \R^{n\times n}$ is defined as 
$\norm{\mX}_{\mW^{-1}} \eqdef \sqrt{\Tr{ \mX \mW^{-1} \mX^\top}}.$
} ($\mW\in \R^{n\times n}$ is some positive definite weight matrix) defined as
$\Proj_\mS \eqdef  \mS (\mS^\top \mW \mS)^{\dagger} \mS^\top \mW$\footnote{Symbol $\dagger$ stands for Moore-Penrose pseudoinverse.}.

Further, in order to construct unbiased stochastic gradient, an access to the random scalar $\theta_{\mS}$ such that
\begin{equation}\label{eq:unbiased}
\ED{\theta_{\mS} \Proj_\mS} \ones  =  \ones,
\end{equation}
where $e$ is the vector of all ones.

Next, the simplest option for the choice of the stochastic gradient is $\nabla f_{\mS}(x)$ -- an unbiased estimate of $\nabla f$ directly constructed using $\mS,\theta_{\mS} $:
\begin{equation}
\label{eq:stochgradplain}
\nabla f_{\mS}(x)   = \frac{\theta_{\mS}}{n}\Jac(x) \Proj_\mS \ones.
\end{equation}

However, one can build a smarter estimate $ \nabla f_{\mS,\mJ}(x) $ via control variates constructed from $\mJ$:
\begin{equation}\label{eq:controlgradJ} 
 \nabla f_{\mS,\mJ}(x) = \frac{\theta_{\mS}}{n} (\Jac(x)-\mJ)   \Proj_\mS\ones  + \frac{1}{n} \mJ \ones.
\end{equation}
The resulting algorithm is stated as Algorithm~\ref{alg:jacsketch}.

\begin{algorithm}
    \begin{algorithmic}[1]
        \Require $\left(\cD, \mW, \theta_{\mS} \right)$, $x^0\in \R^d$, Jacobian estimate $\mJ^0 \in \R^{d \times n}$, stepsize $\gamma>0$         
        
        \For {$k =  0, 1, 2, \dots$}
        \State Sample a fresh copy $\mS_k\sim \cD$

        \State $ \mJ^{k+1}  = \mJ^k(\mI - \Proj_{\mS_k}) + \Jac(x^k)\Proj_{\mS_k}$
         \label{ln:jacupdate}        
        \State $g^{k} =  \nabla f_{\mS_k, \mJ^k}(x^k)$       \label{ln:gradupdate}    
                    
        \State $x^{k+1} = x^k - \gamma g^{k}$ \label{ln:xupdate}    
         
        \EndFor
    \end{algorithmic}
    \caption{{\tt JacSketch} \cite{gower2018stochastic}}
    \label{alg:jacsketch}
\end{algorithm}

Next we present Lemma~\ref{lem:lemmas39_310_jacsketch} which directly justifies the parameter choice from Table~\ref{tbl:special_cases2}. 

\begin{lemma}[Lemmas 2.5, 3.9 and 3.10 from \cite{gower2018stochastic}]\label{lem:lemmas39_310_jacsketch}
    Suppose that there are constants $\cL_1, \cL_2>0$ such that 
    \begin{eqnarray*}
      \ED{ \norm{ \nabla f_{\mS}(x) - \nabla f_{\mS}(x^*)}_2^2 } &\leq& 2  \cL_1 (f(x)-f(x^*)), \qquad \forall x\in \R^d \\
      \ED{\norm{(\Jac(x)-\Jac(x^*)) \Proj_{\mS}  }_{\mW^{-1}}^2 }& \leq & 2\cL_2 (f(x) -f(x^*)), \qquad \forall x\in \R^d, \label{eq:ES2}
\end{eqnarray*}    
    
     Then
    \begin{equation}\label{eq:jacs_contraction}
     \ED{\norm{\mJ^{k+1} -\Jac(x^*)}_{\mW^{-1}}^2} \le (1-\lambda_{\min}) \norm{\mJ^{k}-\Jac(x^*)}_{\mW^{-1}}^2 +  2\cL_2(f(x^k) -f(x^*)),
     \end{equation}
     \begin{equation}\label{eq:gradbndsubdeltaXX}
    \ED{\norm{g^k}_2^2 } \le 4 \cL_1  (f(x^k) -f(x^*)) +  2 \frac{\lambda_{\max}}{n^2} \norm{\mJ^{k} -\Jac(x^*)}_{\mW^{-1}}^2,
    \end{equation}
    where $\lambda_{\min} = \lambda_{\min}\left(\ED{\Proj_{\mS}}\right)$ and $\lambda_{\max} = \lambda_{\max}\left( \mW^{1/2}\left( \ED{\theta_{\mS}^2 \Proj_{\mS} \ones \ones^\top \Proj_{\mS}} -\ones \ones^\top\right) \mW^{1/2}\right)$. Further,     $
    \ED{ \nabla f_{\mS,\mJ}(x)} = \nabla f(x)$.
\end{lemma}

Thus, as a direct consequence of Theorem~\ref{thm:main_gsgm}, we obtain the next corollary.

\begin{corollary}\label{thm:main_jacsketch}
Consider the setup from Lemma~\ref{lem:lemmas39_310_jacsketch}. Suppose that $f$ is $\mu$-strongly convex and choose $\gamma \le \min\left\{\frac{1}{\mu},\frac{1}{2\cL_1 + M\frac{\cL_2}{n}}\right\}$ where $M > \frac{2\lambda_{\max}}{n\lambda_{\min}}$. Then the iterates of {\tt JacSketch} satisfy
    \begin{equation}\label{eq:convergence_jacsketch}
        \EE\left[V^k\right] \le \max\left\{(1-\gamma\mu)^k, \left(1 + \frac{2\lambda_{\max}}{nM} - \lambda_{\min}\right)^k\right\}V^0.
    \end{equation}

\end{corollary}

\subsection{Interpolation between methods ~\label{sec:interpol}}

Given that a set of stochastic gradients satisfy Assumption~\ref{as:general_stoch_gradient}, we show that an any convex combination of the mentioned stochastic gradients satisfy Assumption~\ref{as:general_stoch_gradient} as well.

\begin{lemma}\label{lem:convex_comb}
    Assume that sequences of stochastic gradients $\{g_1^k\}_{k\ge 0}, \ldots, \{g_m^k\}_{k\ge 0}$ at the common iterates $\{x^k\}_{k\ge 0}$ satisfy the Assumption~\ref{as:general_stoch_gradient} with parameters $A(j),B(j),\{\sigma_k^2(j)\}_{k\ge 0}, C(j),\rho(j),D_1(j),D_2(j)$, $j\in[m]$ respectively. Then for any vector $\tau = (\tau_1,\ldots,\tau_m)^\top$ such as $\sum\limits_{j=1}^m\tau_j = 1$ and $\tau_j \ge 0, j\in[m]$ stochastic gradient $g_\tau^k = \sum\limits_{j=1}^m\tau_j g_j^k$ satisfies the Assumption~\ref{as:general_stoch_gradient} with parameters:
    \begin{eqnarray}
        A_\tau = \sum\limits_{j=1}^m\tau_j A(j),\quad B_\tau = 1,\quad \sigma_{\tau,k}^2 = \sum\limits_{j=1}^m B(j)\tau_j \sigma_k^2(j),\quad \rho_\tau = \min\limits_{j\in[m]}\rho(j),\notag\\
        C_\tau = \sum\limits_{j=1}^m\tau_j C(j) B(j),\quad D_{\tau,1} = \sum\limits_{j=1}^m\tau_j D_1(j),\quad D_{\tau,2} = \sum\limits_{j=1}^m\tau_j D_2(j) B(j).\label{eq:conv_comb_params}
    \end{eqnarray}
    Furthermore, if stochastic gradients $g_1^k, \dots, g_m^k$ are independent for all $k$, Assumption~\ref{as:general_stoch_gradient} is satisfied with parameters
        \begin{eqnarray}
        A_\tau = L+\sum\limits_{j=1}^m\tau_j^2 A(j),\quad B_\tau = 1,\quad \sigma_{\tau,k}^2 = \sum\limits_{j=1}^m B(j)\tau_j^2 \sigma_k^2(j),\quad \rho_\tau = \min\limits_{j\in[m]}\rho(j),\notag\\
        C_\tau = \sum\limits_{j=1}^m\tau_j^2 C(j) B(j),\quad D_{\tau,1} = \sum\limits_{j=1}^m\tau_j^2 D_1(j),\quad D_{\tau,2} = \sum\limits_{j=1}^m\tau_j^2 D_2(j) B(j).\label{eq:conv_comb_params_indep}
    \end{eqnarray}
\end{lemma}

What is more, instead of taking convex combination one can choose stochastic gradient at random. Lemma~\ref{lem:flipping_a_coin} provides the result. 
\begin{lemma}\label{lem:flipping_a_coin}
    Assume that sequences of stochastic gradients $\{g_1^k\}_{k\ge 0}, \ldots, \{g_m^k\}_{k\ge 0}$ at the common iterates $\{x^k\}_{k\ge 0}$ satisfy the Assumption~\ref{as:general_stoch_gradient} with parameters $A(j),B(j),\{\sigma_k^2(j)\}_{k\ge 0}, C(j),\rho(j),D_1(j),D_2(j)$, $j\in[m]$ respectively. Then for any vector $\tau = (\tau_1,\ldots,\tau_m)^\top$ such as $\sum\limits_{j=1}^m\tau_j = 1$ and $\tau_j \ge 0, j\in[m]$ stochastic gradient $g_\tau^k$ which equals $g_j^k$ with probability $\tau_j$ satisfies the Assumption~\ref{as:general_stoch_gradient} with parameters:
    \begin{eqnarray}
        A_\tau = \sum\limits_{j=1}^m\tau_j A(j),\quad B_\tau =1,\quad \sigma_{\tau,k}^2 = \sum\limits_{j=1}^m\tau_j B(j) \sigma_k^2(j),\quad \rho_\tau = \min\limits_{j\in[m]}\rho(j),\notag\\
        C_\tau = \sum\limits_{j=1}^m\tau_j  B(j)C(j),\quad D_{\tau,1} = \sum\limits_{j=1}^m\tau_j D_1(j),\quad D_{\tau,2} = \sum\limits_{j=1}^mB(j)\tau_j D_2(j).\label{eq:flipping_a_coin}
    \end{eqnarray}
    Furthermore, if stochastic gradients $g_1^k, \dots, g_m^k$ are independent for all $k$, Assumption~\ref{as:general_stoch_gradient} is satisfied with parameters
        \begin{eqnarray}
        A_\tau = L+\sum\limits_{j=1}^m\tau_j^2 A(j),\quad B_\tau = 1,\quad \sigma_{\tau,k}^2 = \sum\limits_{j=1}^m B(j)\tau_j^2 \sigma_k^2(j),\quad \rho_\tau = \min\limits_{j\in[m]}\rho(j),\notag\\
        C_\tau = \sum\limits_{j=1}^m\tau_j^2 C(j) B(j),\quad D_{\tau,1} = \sum\limits_{j=1}^m\tau_j^2 D_1(j),\quad D_{\tau,2} = \sum\limits_{j=1}^m\tau_j^2 D_2(j) B(j).\label{eq:flipping_a_coin_indep}
    \end{eqnarray}
\end{lemma}

\begin{example}[{\tt $\tau$-L-SVRG}]
    Consider the following method~--- {\tt $\tau$-L-SVRG}~--- which interpolates between vanilla {\tt SGD} and {\tt L-SVRG}.
    \begin{algorithm}[h]
    \caption{{\tt $\tau$-L-SVRG}}
    \label{alg:tau-L-SVRG}
    \begin{algorithmic}
        \Require learning rate $\gamma>0$, probability $p\in (0,1]$, starting point $x^0\in\R^d$, convex combination parameter $\tau\in[0,1]$
        \State $w^0 = x^0$
        \For{ $k=0,1,2,\ldots$ }
        \State{Sample  $i \in \{1,\ldots, n\}$ uniformly at random}
        \State{$g^k_{L-SVRG} = \nabla f_i(x^k) - \nabla f_i(w^k) + \nabla f(w^k)$}
        \State{Sample  $j \in \{1,\ldots, n\}$ uniformly at random}
        \State{$g^k_{SGD} = \nabla f_j(x^k)$}
        \State{$g^k = \tau g^k_{SGD} + (1-\tau)g^k_{L-SVRG}$}        
        \State{$x^{k+1} = x^k - \gamma g^k$}
        \State{$w^{k+1} = \begin{cases}
            x^{k}& \text{with probability } p\\
            w^k& \text{with probability } 1-p
            \end{cases}$
        }
        \EndFor
    \end{algorithmic}
    \end{algorithm}
    When $\tau = 0$ the Algorithm~\ref{alg:tau-L-SVRG} becomes {\tt L-SVRG} and when $\tau = 1$ it is just {\tt SGD} with uniform sampling.    Notice that Lemmas~\ref{lem:l-svrg} and~\ref{lem:exp_smoothness_grad_up_bound_sgd-as} still hold as they does not depend on the update rule for $x^{k+1}$.
    
    Thus, sequences $\{g_{SGD}^k\}_{k\ge 0}$ and $\{g_{L-SVRG}^k\}_{k\ge 0}$ satisfy the conditions of Lemma~\ref{lem:convex_comb} and, as a consequence, stochastic gradient $g^k$ from {\tt $\tau$-L-SVRG} meets the Assumption~\ref{as:general_stoch_gradient} with the following parameters:
    \begin{eqnarray*}
        A_\tau =  L + 2\tau^2\cL + 2(1-\tau)^2L,\quad B_\tau = 1,\quad \sigma_{\tau,k}^2 = 2\frac{(1-\tau)^2}{n}\sum\limits_{i=1}^n\norm{\nabla f_i(w^k) - \nabla f_i(x^*)}^2,\notag\\
         \rho_\tau = p,\quad C_\tau = 2(1-\tau)^2Lp,\quad D_{\tau,1} = 2\tau^2\sigma^2,\quad D_{\tau,2} = 0.
    \end{eqnarray*}
\end{example}

\begin{remark}
Similar interpolation with the analogous analysis can be considered between {\tt SGD} and {\tt SAGA}, or {\tt SGD} and {\tt SVRG}. 
\end{remark}

\subsubsection*{Proof of Lemma~\ref{lem:convex_comb}}

    Indeed, \eqref{eq:general_stoch_grad_unbias} holds due to linearity of mathematical expectation. Next, summing inequalities \eqref{eq:general_stoch_grad_second_moment} for $g_1^k,\ldots,g_m^k$ and using convexity of $\norm{\cdot}^2$ we get
    \begin{eqnarray*}
        \EE\left[\norm{g_\tau^k -\nabla f(x^*)}^2\mid x^k\right] &\le& \sum\limits_{j=1}^m\tau_j\EE\left[\norm{g_j^k-\nabla f(x^*)}^2\mid x^k\right]
        \\
        & \overset{\eqref{eq:general_stoch_grad_second_moment}}{\le}& 2\sum\limits_{j=1}^m\tau_j A(j) D_f(x^k,x^*) + \sum\limits_{j=1}^mB(j)\tau_j\sigma_k^2(j) +     \sum\limits_{j=1}^m\tau_j D_1(j),
    \end{eqnarray*}    
    which implies \eqref{eq:general_stoch_grad_second_moment} for $g_\tau^k$ with $A_\tau = \sum\limits_{j=1}^m\tau_j A(j), B_\tau =1, \sigma_{\tau,k}^2 = \sum\limits_{j=1}^m\tau_j B(j)\sigma_k^2(j), D_{\tau,1} = \sum\limits_{j=1}^m\tau_j D_1(j)$.
    Finally, summing \eqref{eq:gsg_sigma} for $g_1^k,\ldots,g_m^k$ gives us
    \begin{equation*}
        \EE\left[\sigma_{\tau,k+1}^2\mid\sigma_{\tau,k}^2\right] \overset{\eqref{eq:gsg_sigma}}{\le} \left(1-\min\limits_{j\in[m]}\rho(j)\right)\sigma_{\tau,k}^2 + 2\sum\limits_{j=1}^m\tau_j B(j)C(j)D_f(x^k,x^*) + \sum\limits_{j=1}^m\tau_j B(j)D_2(j),
    \end{equation*}
    which is exactly \eqref{eq:gsg_sigma} for $\sigma_{\tau,k}^2$ with $\rho =\min\limits_{j\in[m]}\rho(j), C_\tau = \sum\limits_{j=1}^m\tau_j C(j), D_{\tau,2} = \sum\limits_{j=1}^m\tau_j D_2(j)$.

Next, for independent gradients we have
\begin{eqnarray}
\EE\left[\norm{g_\tau^k -\nabla f(x^*)}^2\mid x^k\right]   &=&
 \sum_{j=1}^m \tau_j^2 \EE \left[\norm{g_j^k -\nabla f(x^*)}^2\mid x^k\right] + 2\sum_{i< j} \tau_i \tau_j\EE\<g_j^k -\nabla f(x^*),g_i^k -\nabla f(x^*) >
 \nonumber
 \\
  &=&
 \sum_{j=1}^m \tau_j^2 \EE \left[\norm{g_j^k -\nabla f(x^*)}^2\mid x^k\right] + 2\sum_{i< j}\tau_i \tau_j \norm{\nabla f(x^k) -\nabla f(x^*)}^2
  \nonumber
 \\
  &\leq&
 \sum_{j=1}^m \tau_j^2 \EE \left[\norm{g_j^k -\nabla f(x^*)}^2\mid x^k\right] + \left(\sum_{ j=1}^m \tau_j\right)^2 \norm{\nabla f(x^k) -\nabla f(x^*)}^2
  \nonumber
  \\
  &=&
 \sum_{j=1}^m \tau_j^2 \EE \left[\norm{g_j^k -\nabla f(x^*)}^2\mid x^k\right] +  \norm{\nabla f(x^k) -\nabla f(x^*)}^2
  \nonumber
   \\
  &\leq&
 \sum_{j=1}^m \tau_j^2 \EE \left[\norm{g_j^k -\nabla f(x^*)}^2\mid x^k\right] +  2LD_f(x^k,x^*).
 \label{eq:indp_bounding}
\end{eqnarray}
and further the bounds follow. 

\subsubsection*{Proof of Lemma~\ref{lem:flipping_a_coin}}

    Indeed, \eqref{eq:general_stoch_grad_unbias} holds due to linearity and tower property of mathematical expectation. Next, using tower property of mathematical expectation and inequalities \eqref{eq:general_stoch_grad_second_moment} for $g_1^k,\ldots,g_m^k$ we get
    \begin{eqnarray*}
        \EE\left[\norm{g_\tau^k-\nabla f(x^*)}^2\mid x^k\right] &=& \EE\left[\EE_\tau\left[\norm{g_\tau^k-\nabla f(x^*)}^2\right]\mid x^k\right] = \sum\limits_{j=1}^m\tau_j\EE\left[\norm{g_j^k-\nabla f(x^*)}^2\mid x^k\right]\\ &\overset{\eqref{eq:general_stoch_grad_second_moment}}{\le}& 2\sum\limits_{j=1}^m\tau_j A(j) D_f(x^k,x^*) + \sum\limits_{j=1}^m B(j)\tau_j\sigma_k^2(j) +     \sum\limits_{j=1}^m\tau_j D_1(j),
    \end{eqnarray*}
    which implies \eqref{eq:general_stoch_grad_second_moment} for $g_\tau^k$ with $A_\tau = \sum\limits_{j=1}^m\tau_j A(j), B_\tau = 1, \sigma_{\tau,k}^2 = \sum\limits_{j=1}^m\tau_j B(j) \sigma_k^2(j), D_{\tau,1} = \sum\limits_{j=1}^m\tau_j D_1(j)$.
    Finally, summing \eqref{eq:gsg_sigma} for $g_1^k,\ldots,g_m^k$ gives us
    \begin{equation*}
        \EE\left[\sigma_{\tau,k+1}^2\mid\sigma_{\tau,k}^2\right] \overset{\eqref{eq:gsg_sigma}}{\le} \left(1-\min\limits_{j\in[m]}\rho(j)\right)\sigma_{\tau,k}^2 + 2\sum\limits_{j=1}^m\tau_j B(j) C(j)D_f(x^k,x^*) + \sum\limits_{j=1}^m\tau_jB(j)D_2(j),
    \end{equation*}
    which is exactly \eqref{eq:gsg_sigma} for $\sigma_{\tau,k}^2$ with $\rho =\min\limits_{j\in[m]}\rho(j), C_\tau = \sum\limits_{j=1}^m\tau_j B(j)C(j), D_{\tau,2} = \sum\limits_{j=1}^m\tau_j B(j)D_2(j)$.
To show~\eqref{eq:flipping_a_coin_indep}, it suffices to combine above bounds with the trick~\eqref{eq:indp_bounding}.

\begin{remark}
Recently, \cite{tran2019hybrid} demonstrated in that the convex combination of {\tt SGD} and {\tt SARAH}~\cite{nguyen2017sarah} performs very well on non-convex problems. 
\end{remark}

\clearpage
\section{Extra Experiments \label{sec:extra}}
\subsection{{\tt SGD-MB}: remaining experiments and exact problem setup.}

As already described in Section~\ref{sec:exp}, we demonstrate that {\tt SGD-MB} have indistinguishable iteration complexity to independent {\tt SGD}. The considered problem is logistic regression with Tikhonov regularization of order $\lambda$:
\begin{equation}\label{eq:logreg}
\frac1n \sum_{i=1}^n   \log \left(1+\exp\left(a_i^\top x\cdot  b_i\right) \right)+\frac{\lambda}{2} \norm{ x}^2,
\end{equation}
where $a_i\in \R^{n}$, $b_i \in \{-1,1\}$ is $i$-th data-label pair is a vector of labels and $\lambda\geq 0$ is the regularization parameter. The data and labels were obtained from LibSVM datasets {\tt a1a}, {\tt a9a}, w1a, {\tt w8a}, {\tt gisette}, {\tt madelon}, {\tt phishing} and {\tt mushrooms}. Further, the data were rescaled by a random variable $cu_i^2$ where $u_i$ is random integer from $1,2,\dots, 1000$ and $c$ is such that the mean norm of $a_i$ is $1$. 

Note that we have now an infinite array of possibilities on how to write~\eqref{eq:logreg} as~\eqref{eq:f_sum}. For simplicity, distribute $l2$ term evenly among the finite sum.  

The full results can be found in Figure~\ref{fig:SGDMB_full}.

\begin{figure}[!h]
\centering
\begin{minipage}{0.24\textwidth}
  \centering
\includegraphics[width =  \textwidth ]{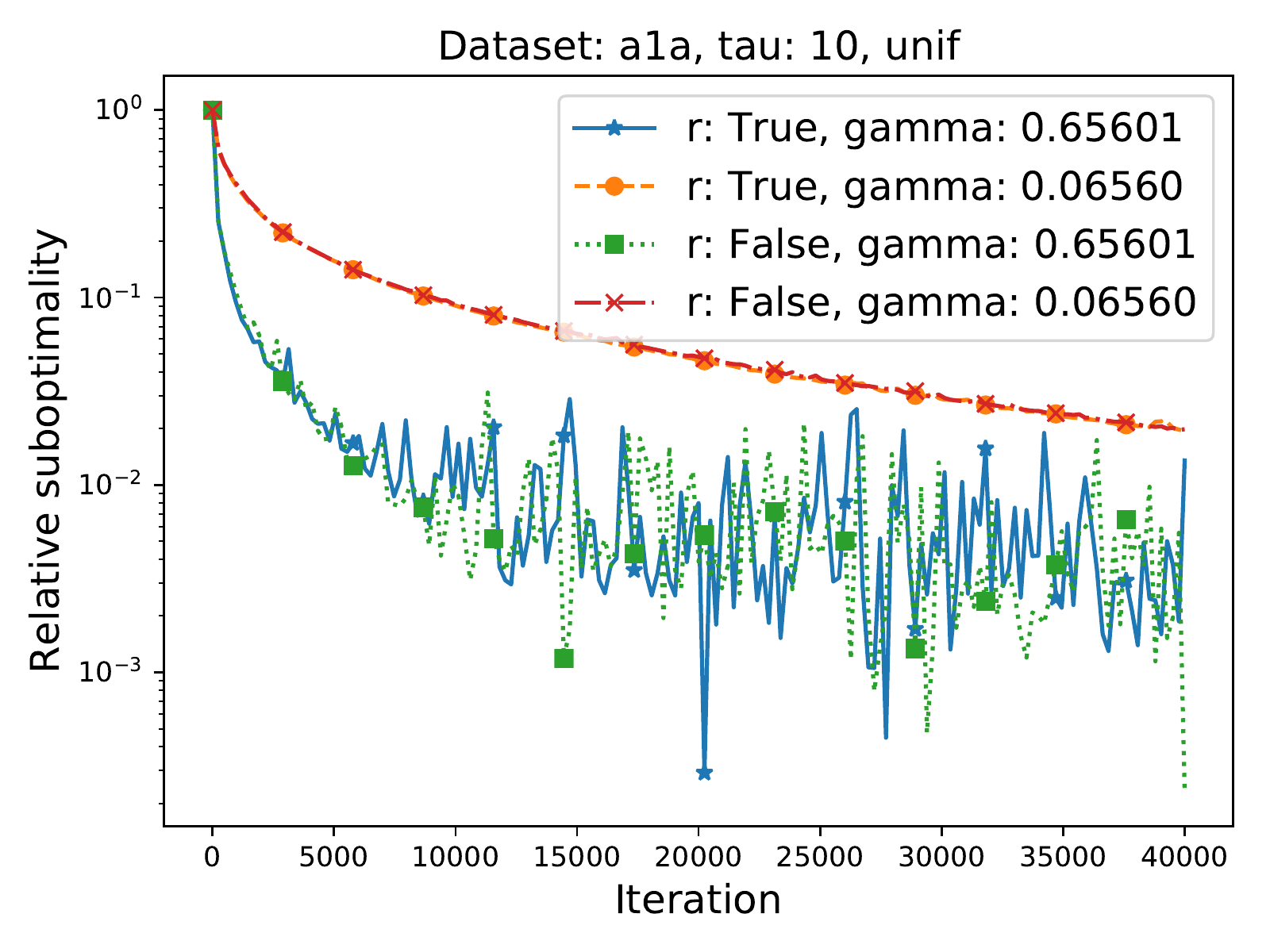}
\end{minipage}%
\begin{minipage}{0.24\textwidth}
  \centering
\includegraphics[width =  \textwidth ]{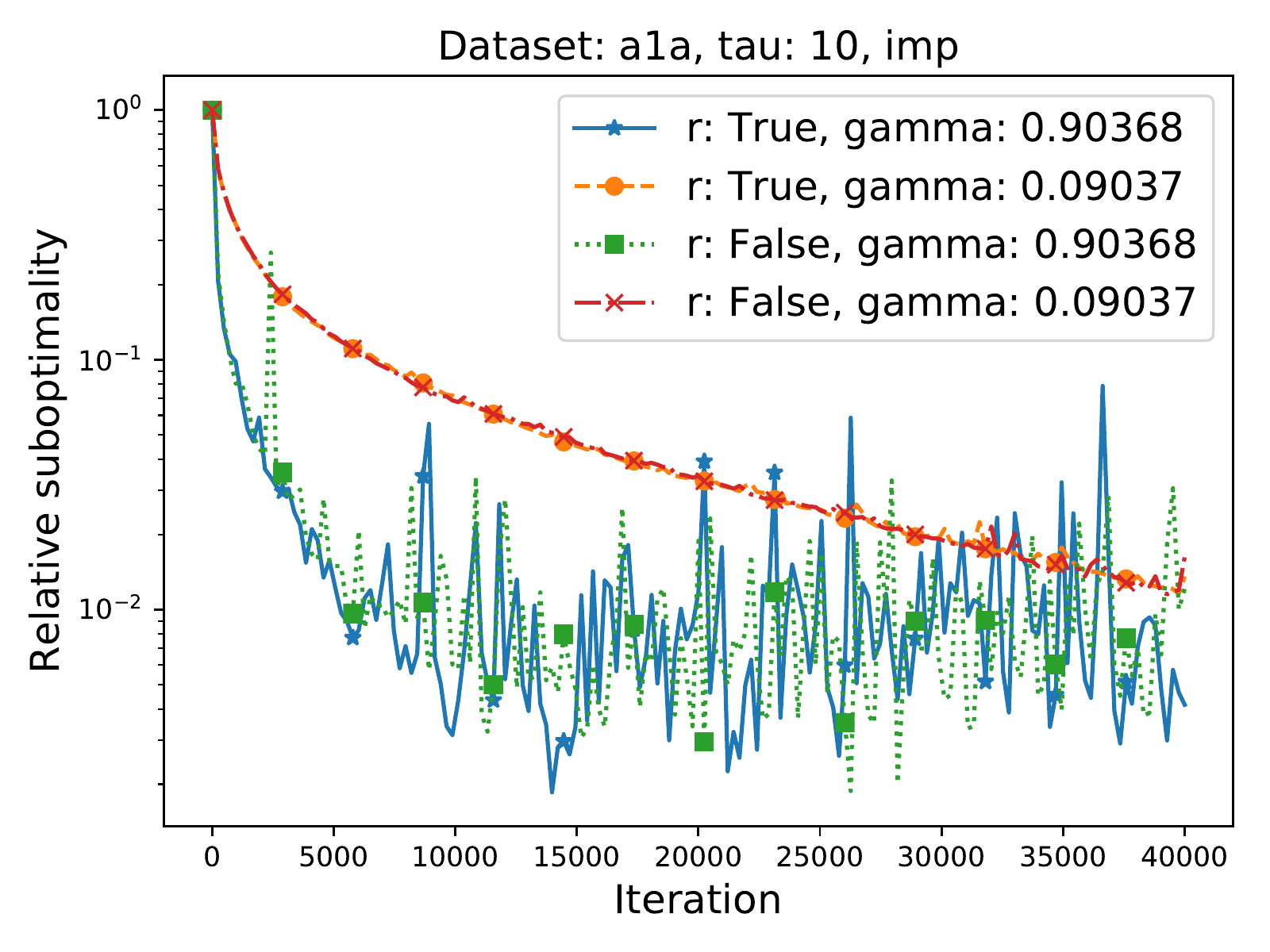}
\end{minipage}%
\begin{minipage}{0.24\textwidth}
  \centering
\includegraphics[width =  \textwidth ]{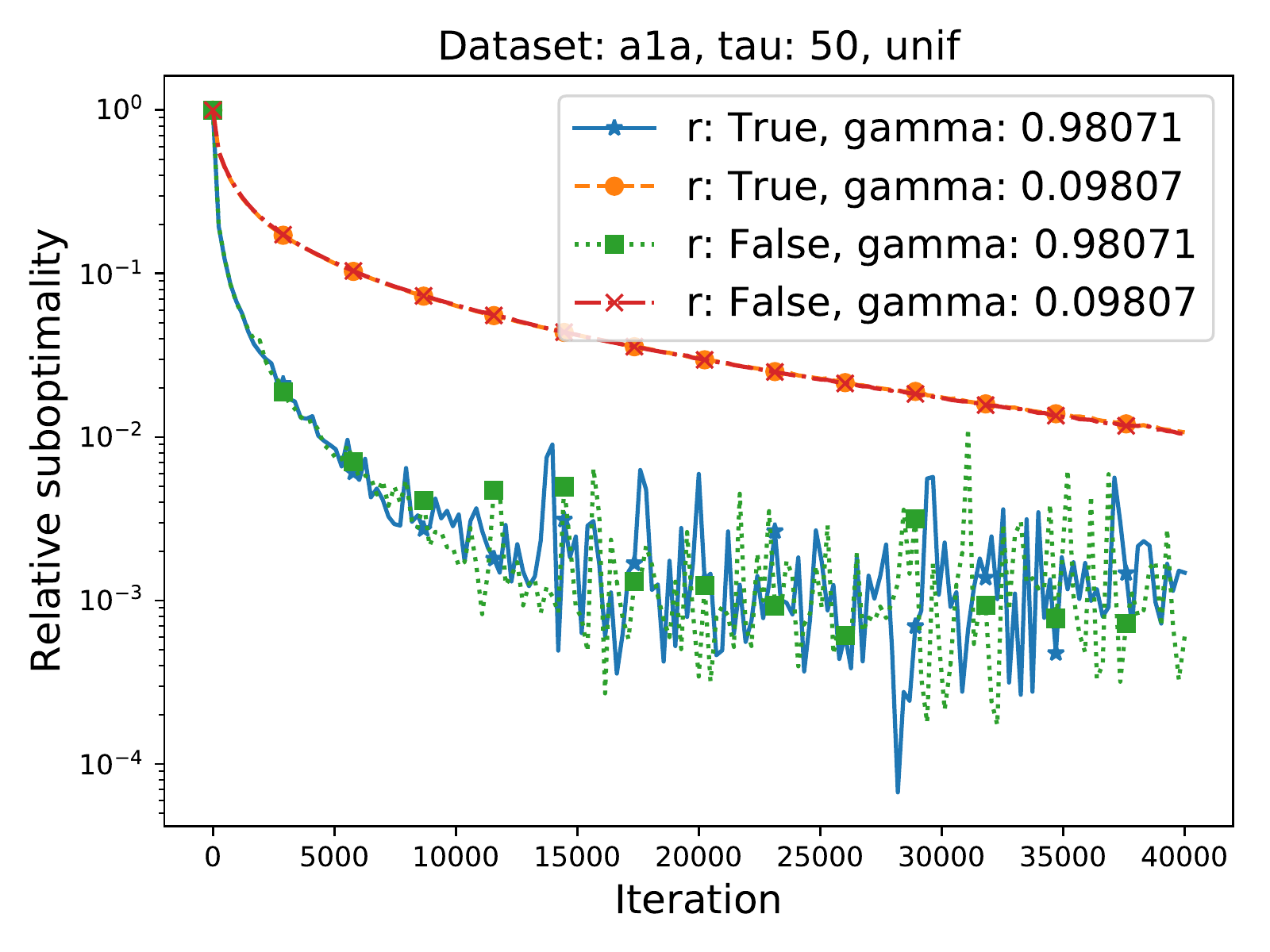}
\end{minipage}%
\begin{minipage}{0.24\textwidth}
  \centering
\includegraphics[width =  \textwidth ]{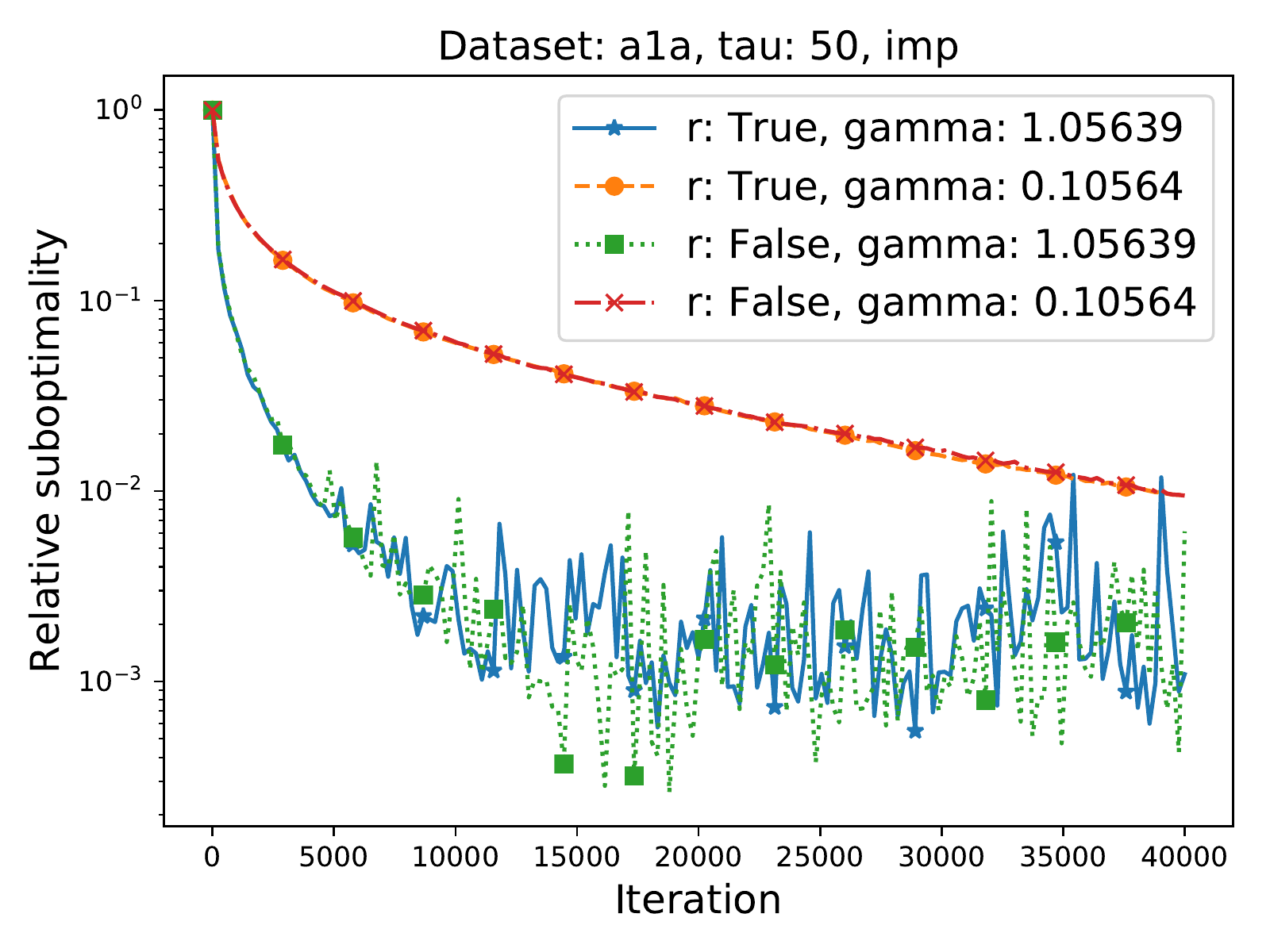}
\end{minipage}%
\\
\begin{minipage}{0.24\textwidth}
  \centering
\includegraphics[width =  \textwidth ]{SGD_w1a_tau_10_meth_3_unif.pdf}
\end{minipage}%
\begin{minipage}{0.24\textwidth}
  \centering
\includegraphics[width =  \textwidth ]{SGD_w1a_tau_10_meth_3_imp.pdf}
\end{minipage}%
\begin{minipage}{0.24\textwidth}
  \centering
\includegraphics[width =  \textwidth ]{SGD_w1a_tau_50_meth_3_unif.pdf}
\end{minipage}%
\begin{minipage}{0.24\textwidth}
  \centering
\includegraphics[width =  \textwidth ]{SGD_w1a_tau_50_meth_3_imp.pdf}
\end{minipage}%
\\
\begin{minipage}{0.24\textwidth}
  \centering
\includegraphics[width =  \textwidth ]{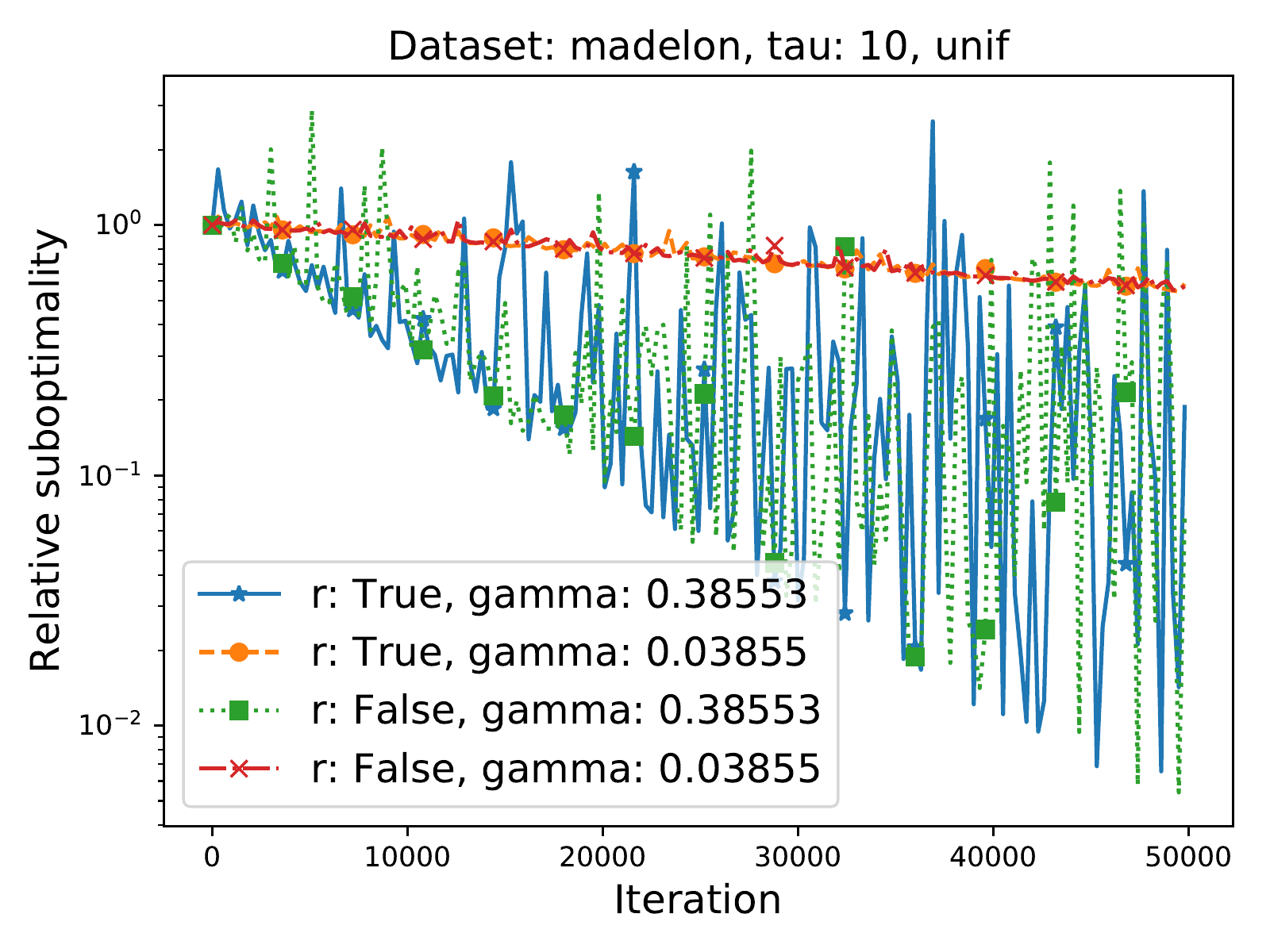}
\end{minipage}%
\begin{minipage}{0.24\textwidth}
  \centering
\includegraphics[width =  \textwidth ]{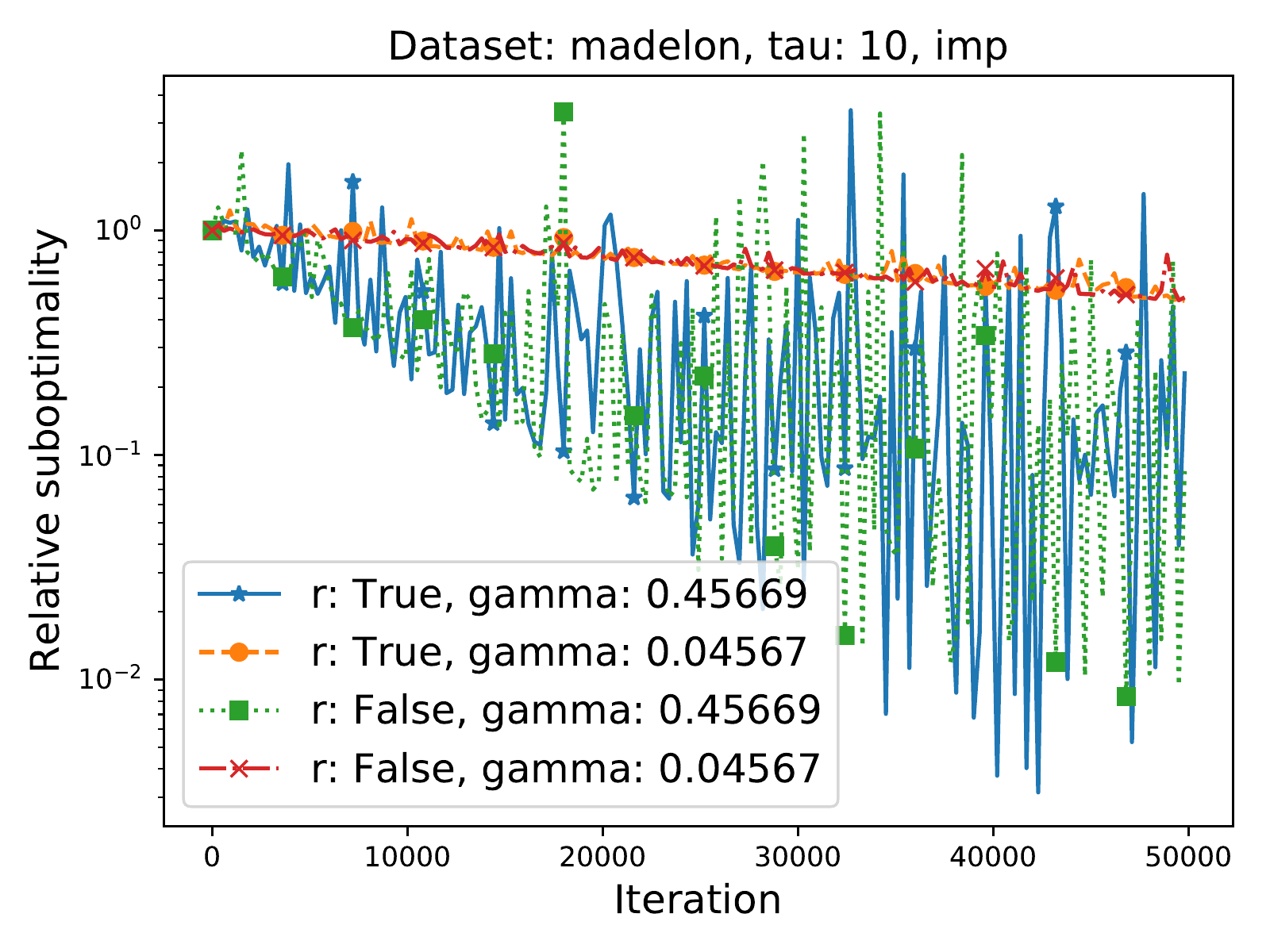}
\end{minipage}%
\begin{minipage}{0.24\textwidth}
  \centering
\includegraphics[width =  \textwidth ]{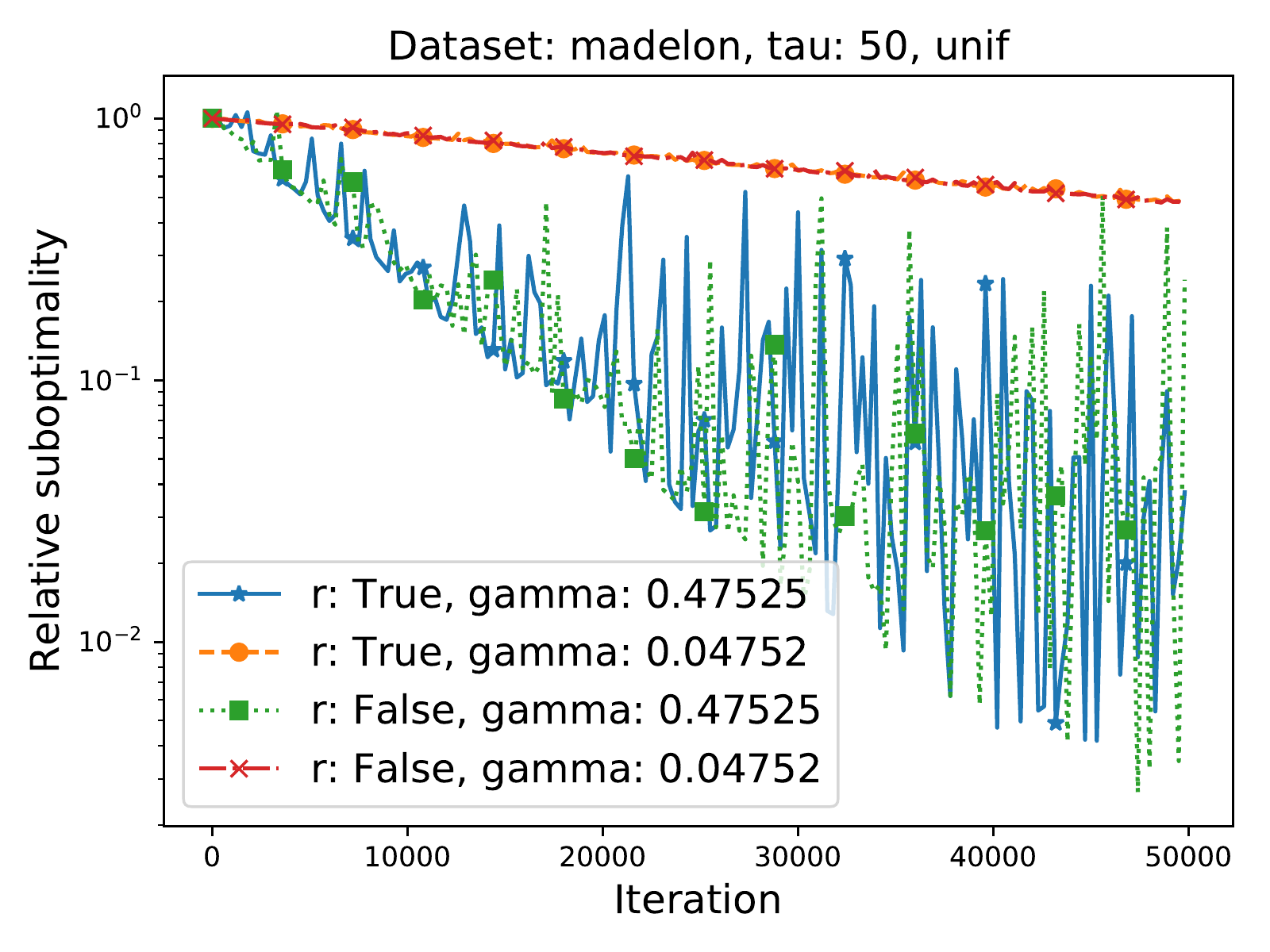}
\end{minipage}%
\begin{minipage}{0.24\textwidth}
  \centering
\includegraphics[width =  \textwidth ]{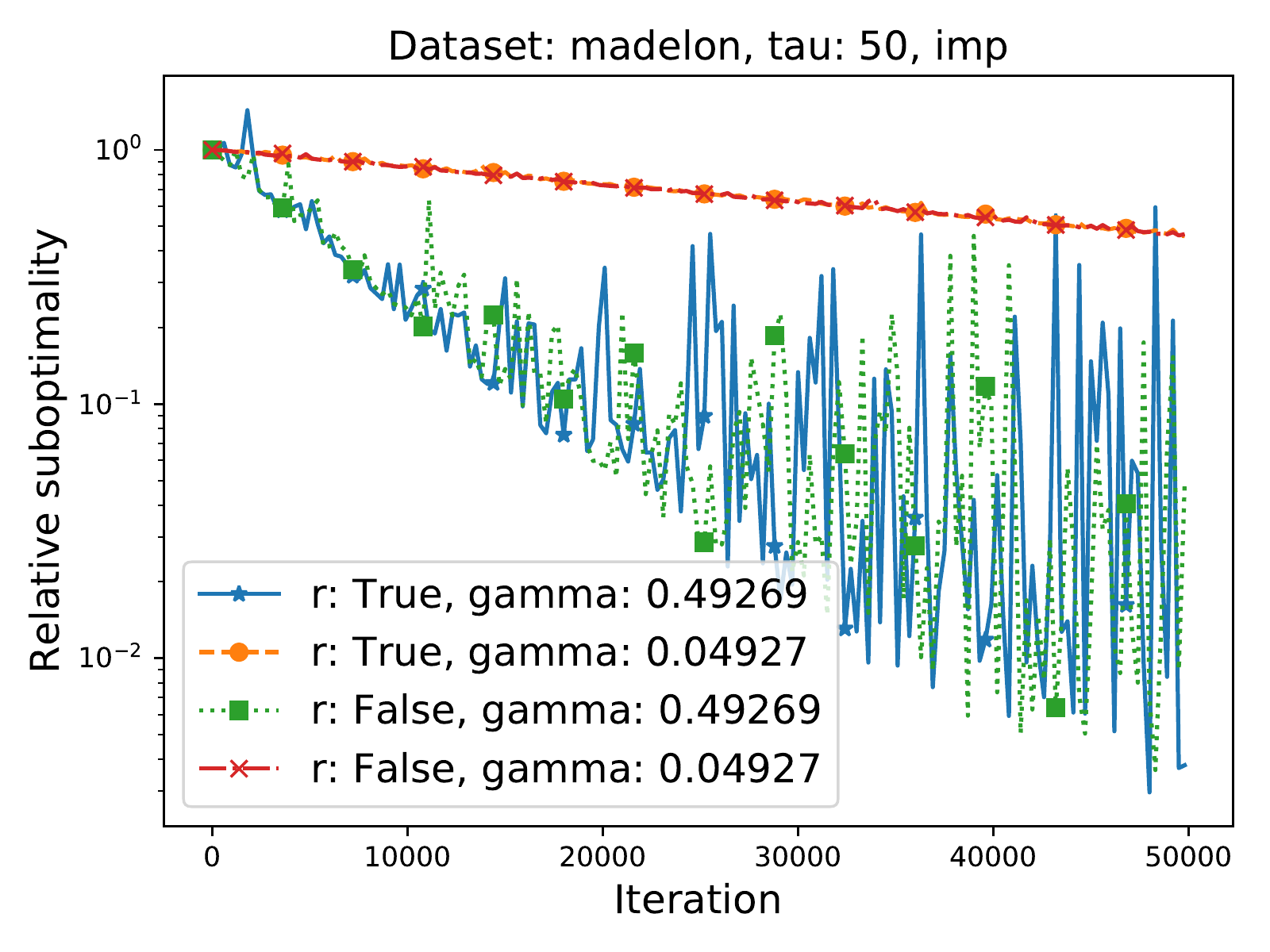}
\end{minipage}%
\\
\begin{minipage}{0.24\textwidth}
  \centering
\includegraphics[width =  \textwidth ]{SGD_gisette_scale_tau_10_meth_3_unif.pdf}
\end{minipage}%
\begin{minipage}{0.24\textwidth}
  \centering
\includegraphics[width =  \textwidth ]{SGD_gisette_scale_tau_10_meth_3_imp.pdf}
\end{minipage}%
\begin{minipage}{0.24\textwidth}
  \centering
\includegraphics[width =  \textwidth ]{SGD_gisette_scale_tau_50_meth_3_unif.pdf}
\end{minipage}%
\begin{minipage}{0.24\textwidth}
  \centering
\includegraphics[width =  \textwidth ]{SGD_gisette_scale_tau_50_meth_3_imp.pdf}
\end{minipage}%
\\
\begin{minipage}{0.24\textwidth}
  \centering
\includegraphics[width =  \textwidth ]{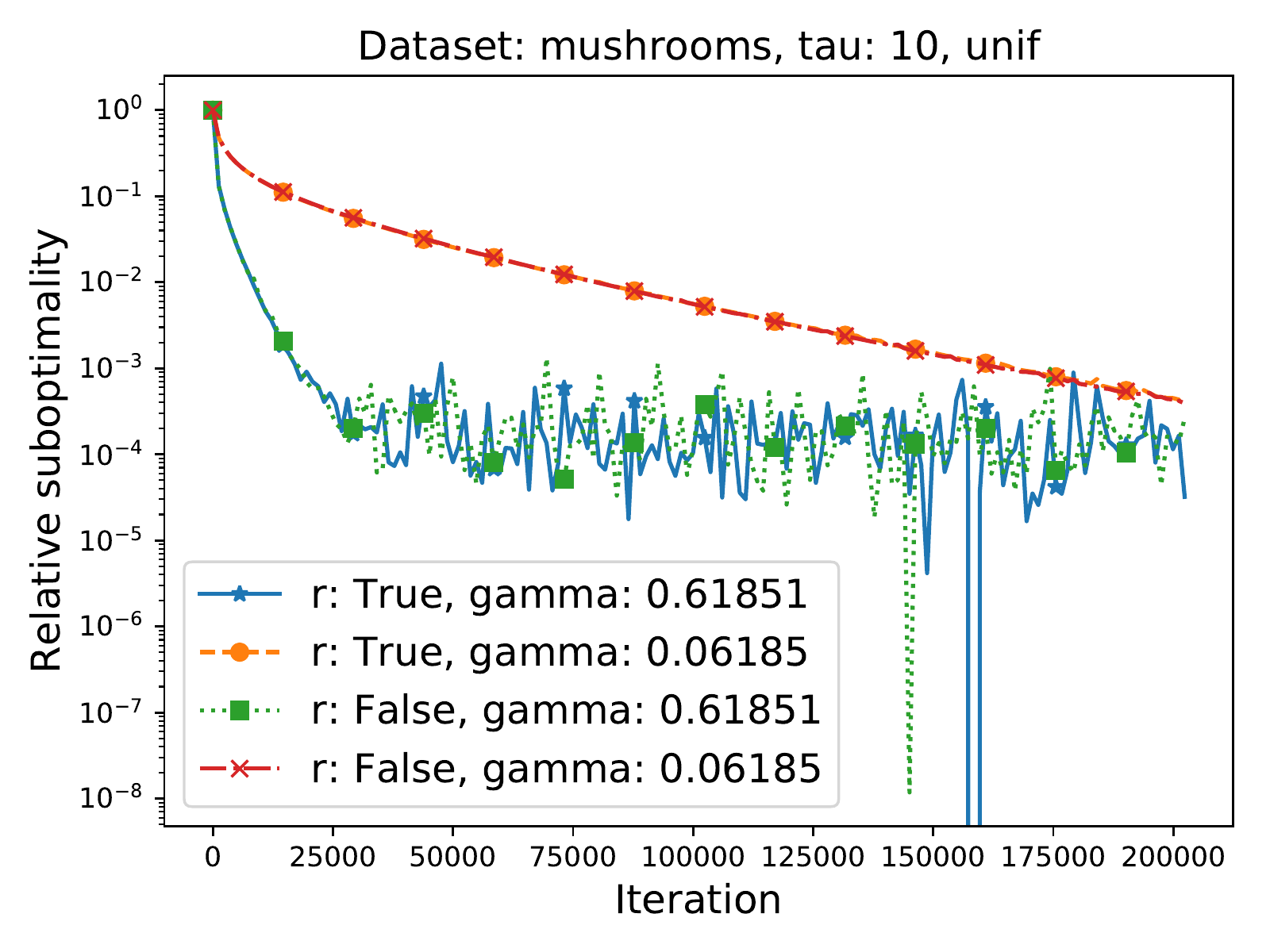}
\end{minipage}%
\begin{minipage}{0.24\textwidth}
  \centering
\includegraphics[width =  \textwidth ]{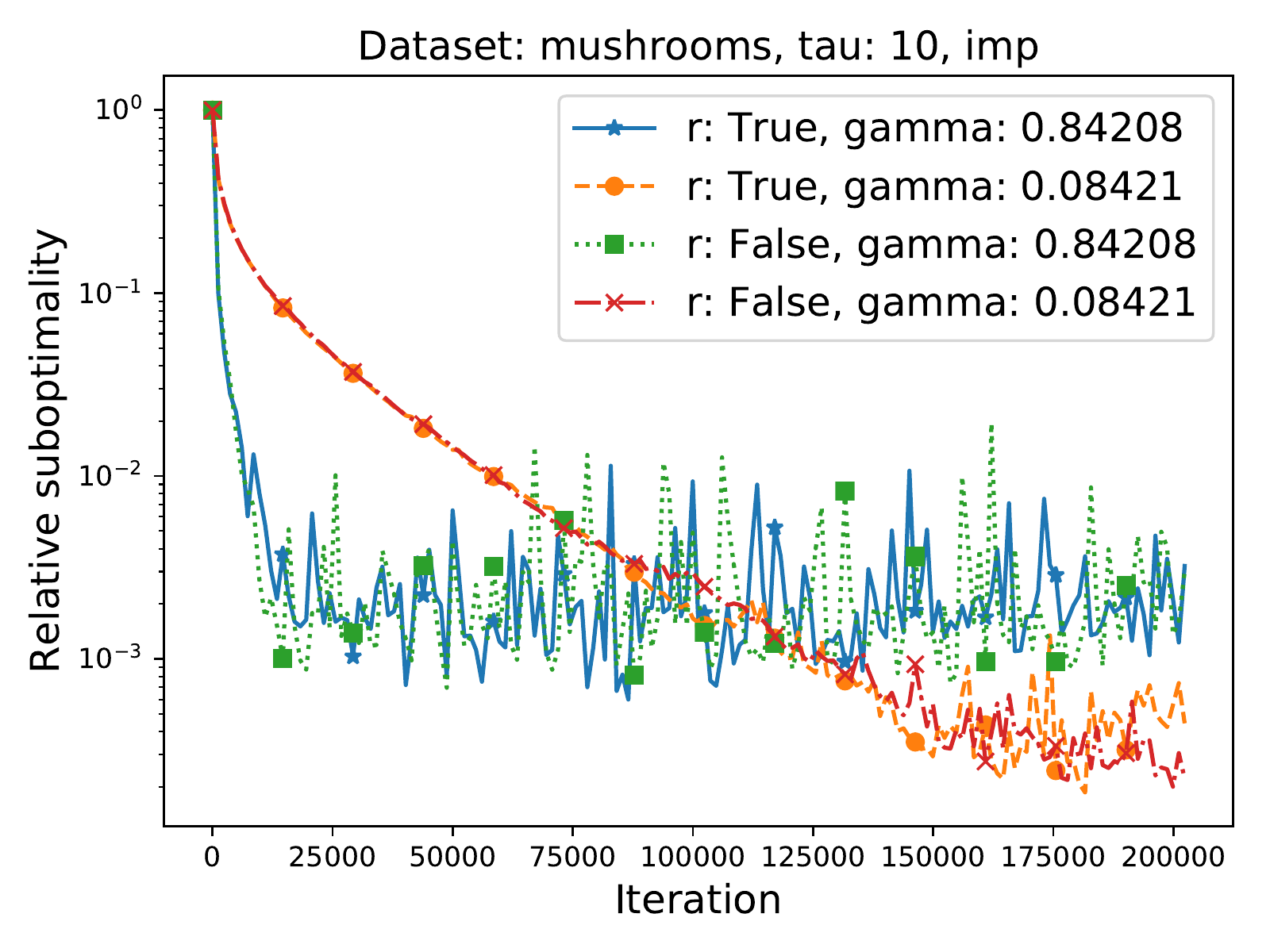}
\end{minipage}%
\begin{minipage}{0.24\textwidth}
  \centering
\includegraphics[width =  \textwidth ]{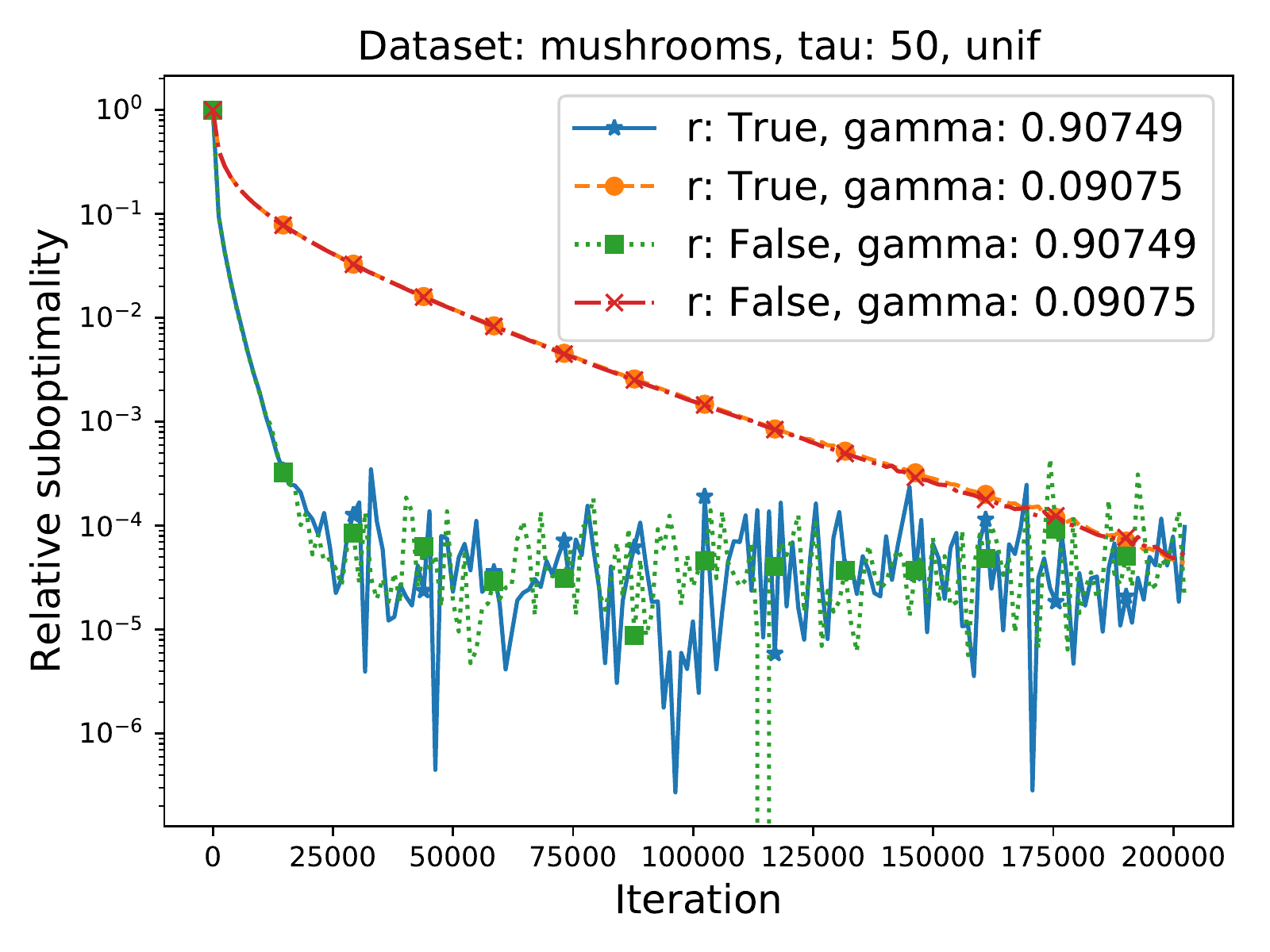}
\end{minipage}%
\begin{minipage}{0.24\textwidth}
  \centering
\includegraphics[width =  \textwidth ]{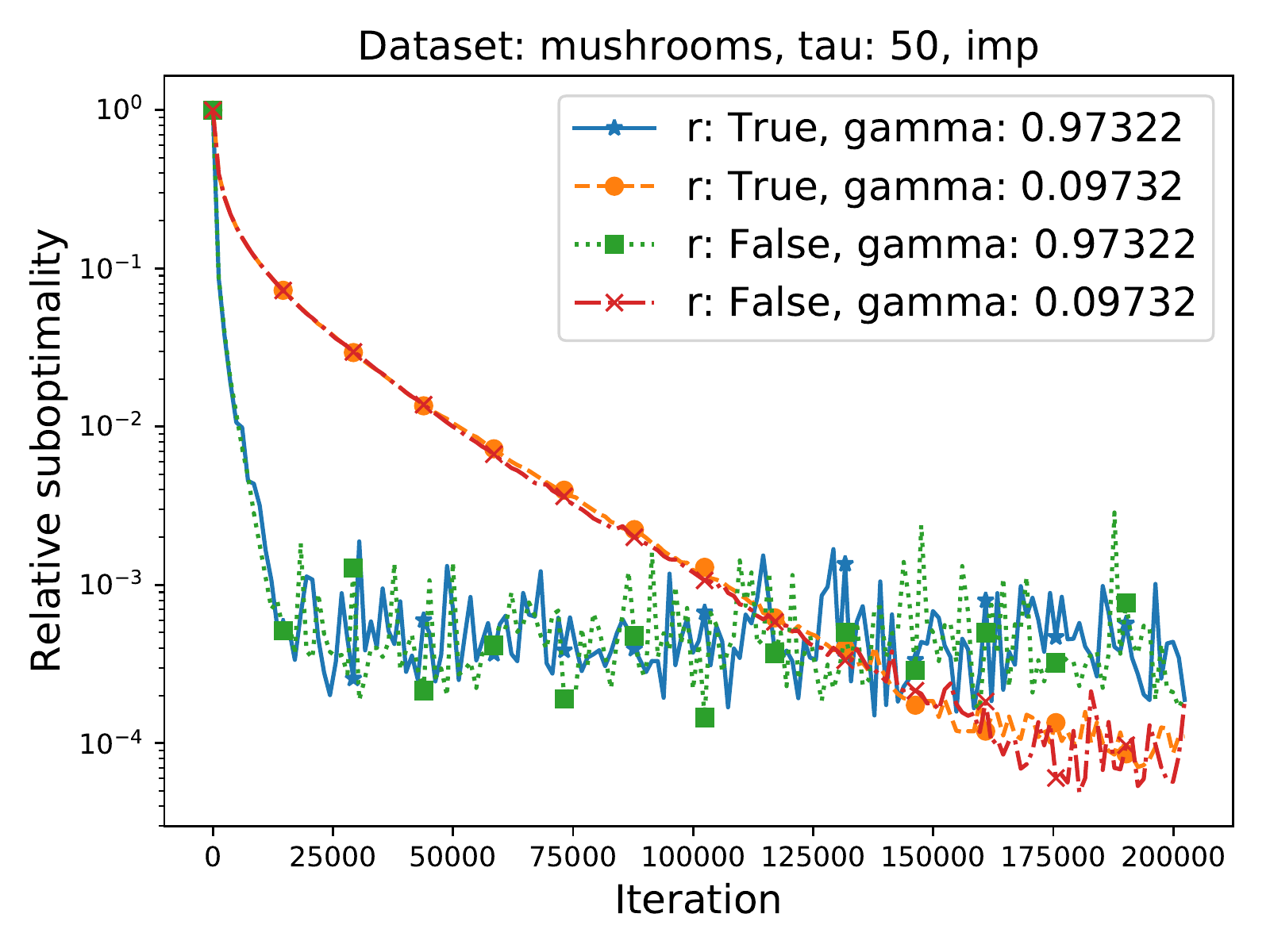}
\end{minipage}%
\caption{{\tt SGD-MB} and independent {\tt SGD} applied on LIBSVM~\cite{chang2011libsvm} datasets with regularization parameter $\lambda = 10^{-5}$. Axis $y$ stands for relative suboptimality, i.e. $\frac{f(x^k)-f(x^*)}{f(x^k)-f(x^0)}$. Title label ``unif'' corresponds to probabilities chosen by~\ref{item:unif} while label ``imp'' corresponds to probabilities chosen by~\ref{item:imp}. Lastly, legend label ``r'' corresponds to ``replacement'' with value ``True'' for {\tt SGD-MB} and value ``False'' for independent {\tt SGD}.}
\label{fig:SGDMB_full}
\end{figure}

Note that plots which are not included in the main body (due to space limitations) only support claims from Section~\ref{sec:exp}.

\subsection{Experiments on {\tt SGD-star} \label{sec:exp_star}}
In this section, we study {\tt SGD-star} and numerically verify claims from Section~\ref{sec:SGD-star}. In particular, Corollary~\ref{cor:SGD-star} shows that {\tt SGD-star} enjoys linear convergence rate which is constant times better to the rate of {\tt SAGA} (given that problem condition number is high enough). We compare 3 methods -- {\tt SGD-star}, {\tt SGD} and {\tt SAGA}. We consider simple and well-understood least squares problem $\min_x \frac12 \| \mA x-b\|^2$ where elements of $\mA,b$ were generated (independently) from standard normal distribution. Further, rows of $\mA$ were normalized so that $\|\mA_{i:}\|=1$. Thus, denoting $f_i(x) = \frac12 (\mA_{i:}^\top x-b_i )^2$, $f_i$ is 1-smooth. For simplicity, we consider {\tt SGD-star} with uniform serial sampling, i.e. $\cL=1$.

 Next, for both {\tt SGD-star} and {\tt SGD} we use stepsize $\gamma = \frac{1}{2}$ (theory supported stepsize for {\tt SGD-star}), while for {\tt SAGA} we set $\gamma = \frac{1}{5}$ (almost theory supported stepsize). Figure~\ref{fig:star} shows the results.

\begin{figure}[!h]
\centering
\begin{minipage}{0.3\textwidth}
  \centering
\includegraphics[width =  \textwidth ]{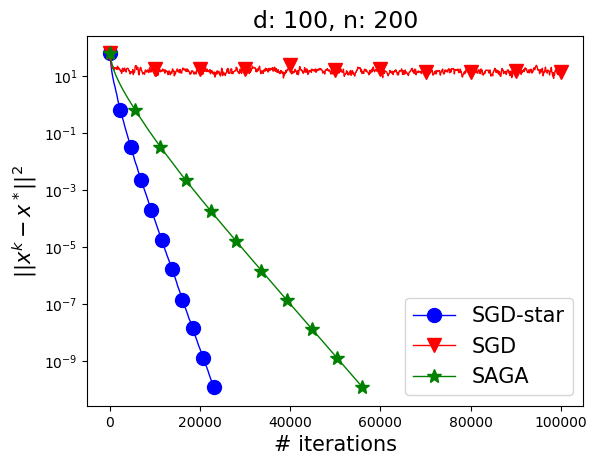}
\end{minipage}%
\begin{minipage}{0.3\textwidth}
  \centering
\includegraphics[width =  \textwidth ]{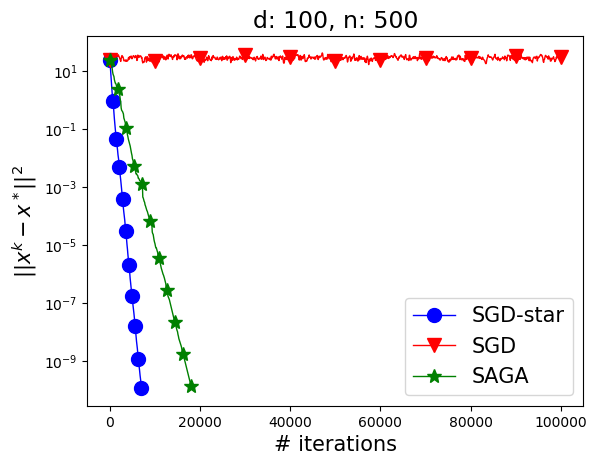}
\end{minipage}%
\begin{minipage}{0.3\textwidth}
  \centering
\includegraphics[width =  \textwidth ]{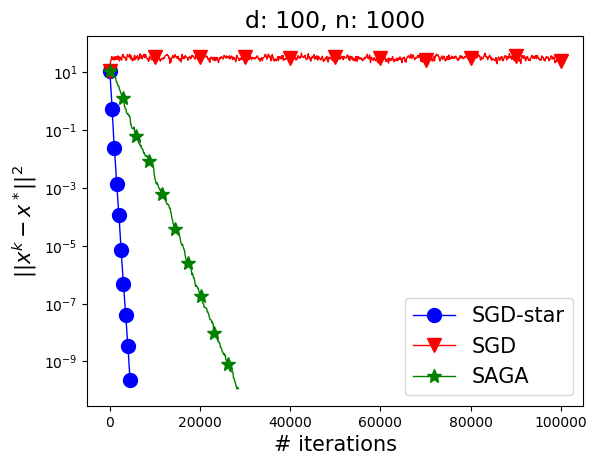}
\end{minipage}%
\\
\begin{minipage}{0.3\textwidth}
  \centering
\includegraphics[width =  \textwidth ]{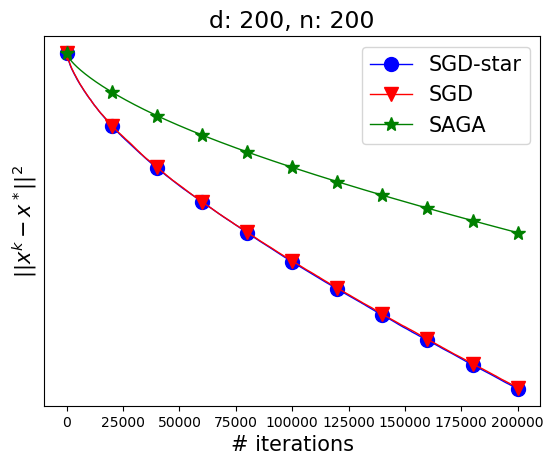}
\end{minipage}%
\begin{minipage}{0.3\textwidth}
  \centering
\includegraphics[width =  \textwidth ]{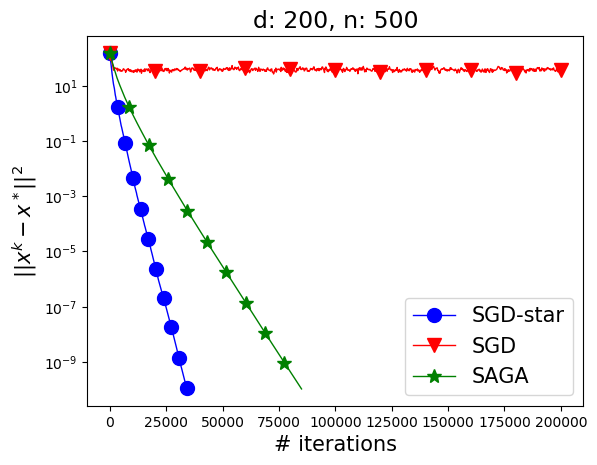}
\end{minipage}%
\begin{minipage}{0.3\textwidth}
  \centering
\includegraphics[width =  \textwidth ]{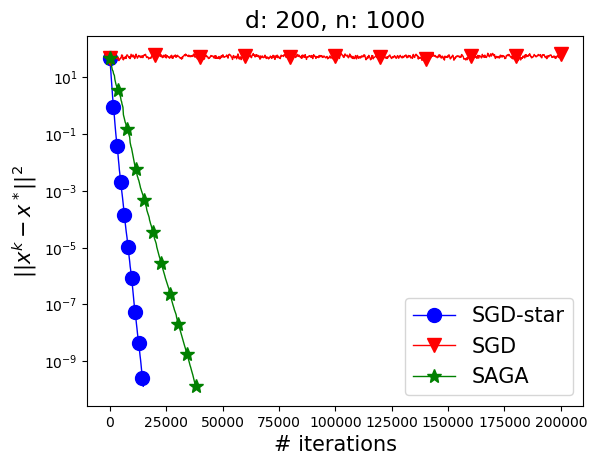}
\end{minipage}%
\\
\begin{minipage}{0.3\textwidth}
  \centering
\includegraphics[width =  \textwidth ]{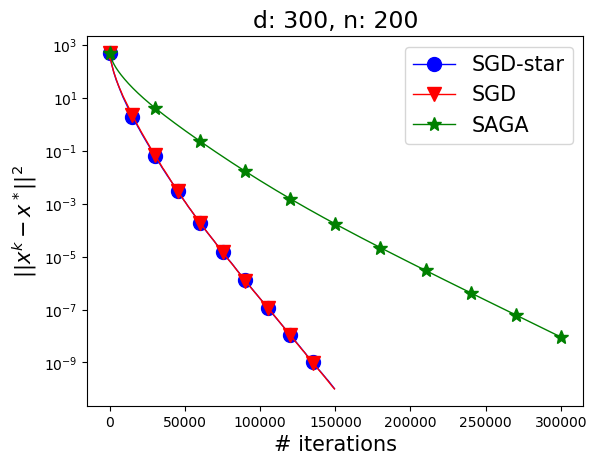}
\end{minipage}%
\begin{minipage}{0.3\textwidth}
  \centering
\includegraphics[width =  \textwidth ]{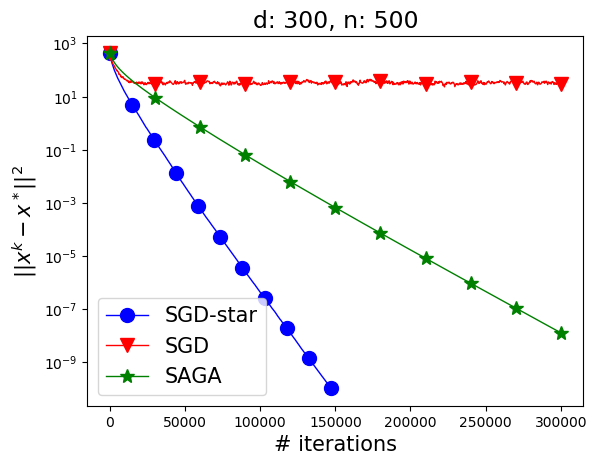}
\end{minipage}%
\begin{minipage}{0.3\textwidth}
  \centering
\includegraphics[width =  \textwidth ]{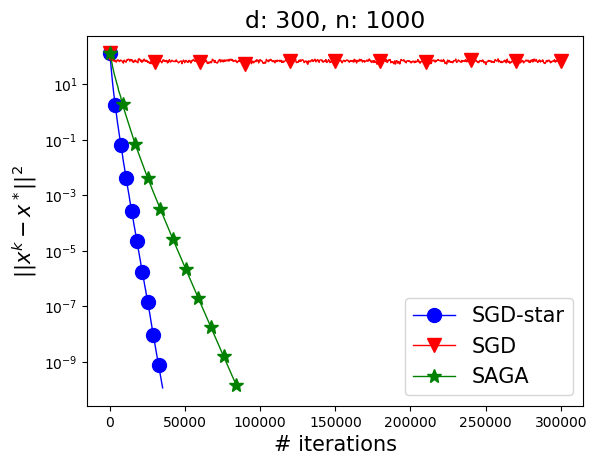}
\end{minipage}%
\caption{Comparison of {\tt SGD-star}, {\tt SGD} and {\tt SAGA} on least squares problem.}
\label{fig:star}
\end{figure}

Note that, as theory predicts, {\tt SGD-star} is always faster to {\tt SAGA}, although only constant times. Further, in the cases where $d\geq n$, performance of {\tt SGD} seems identical to the performance of {\tt SGD-shift}. This is due to a simple reason: if $d\geq n$, we must have $\nabla f_i(x^*) = 0$ for all $i$, and thus {\tt SGD} and {\tt SGD-shift} are in fact identical algorithms. 

\subsection{Experiments on {\tt N-SEGA} \label{sec:exp_nsega}}
In this experiment we study the effect of noise on {\tt N-SEGA}. We consider unit ball constrained least squares problem: $\min_{\|x\|\leq 1} f(x)$ where $f(x)=\|\mA x-b\|^2$. and we suppose that there is an oracle providing us with 
noised partial derivative $g_i(x,\zeta) = \nabla_i f(x) +\zeta$, where $\zeta \sim N(0,\sigma^2)$. For each problem instance (i.e. pair $\mA, b$), we compare performance of  {\tt N-SEGA} under various noise magnitudes $\sigma^2$.

The specific problem instances are presented in Table~\ref{tbl:leastsquares}. Figure~\ref{fig:nsega} shows the results. 

\begin{table}[!h]
\begin{center}
\begin{tabular}{|c|c|c|}
\hline
Type & $\mA $ & $b$ \\
 \hline
 \hline
 1   & $\mA_{ij}\sim N(0,1)$ (independently)  & vector of ones  \\
 \hline
  2   & Same as 1, but scaled so that $\lambda_{\max}(A^\top A)=1$ & vector of ones \\
\hline
 3   & $\mA_{ij} = \varrho_{ij}\varpi_{j}$ $\forall i,j:\varrho_{ij},\varpi_{j} \sim N(0,1)$ (independently)  & vector of ones  \\
 \hline
  4   & Same as 3, but scaled so that $\lambda_{\max}(A^\top A)=1$ & vector of ones \\
\hline
\end{tabular}
\end{center}
\caption{Four types of least squares. }
\label{tbl:leastsquares}
\end{table}

We shall mention that this experiment serves to support and give a better intuition about the results from Section~\ref{N-SEGA} and is by no means practical. The results show, as predicted by theory, linear convergence to a specific neighborhood of the objective. The effect of the noise varies, however, as a general rule, the larger strong convexity $\mu$ is (i.e. problems 1,3 where scaling was not applied), the smaller the effect of noise is.

\begin{figure}[!h]
\centering
\begin{minipage}{0.24\textwidth}
  \centering
\includegraphics[width =  \textwidth ]{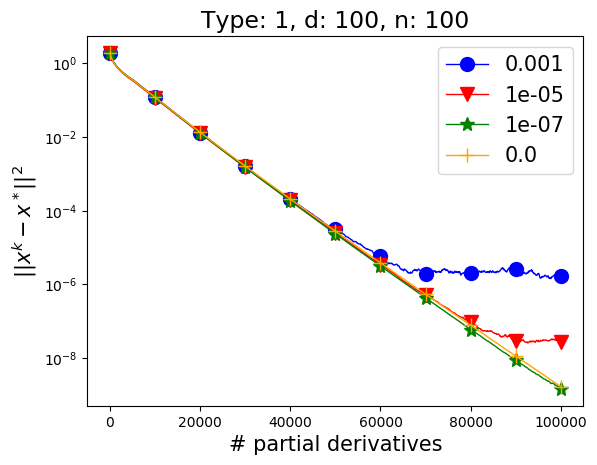}
\end{minipage}%
\begin{minipage}{0.24\textwidth}
  \centering
\includegraphics[width =  \textwidth ]{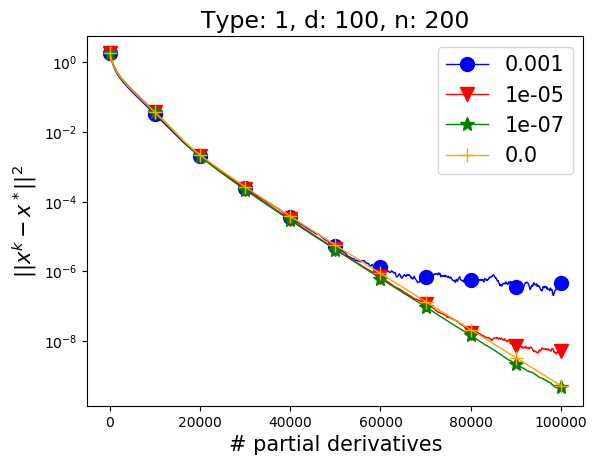}
\end{minipage}%
\begin{minipage}{0.24\textwidth}
  \centering
\includegraphics[width =  \textwidth ]{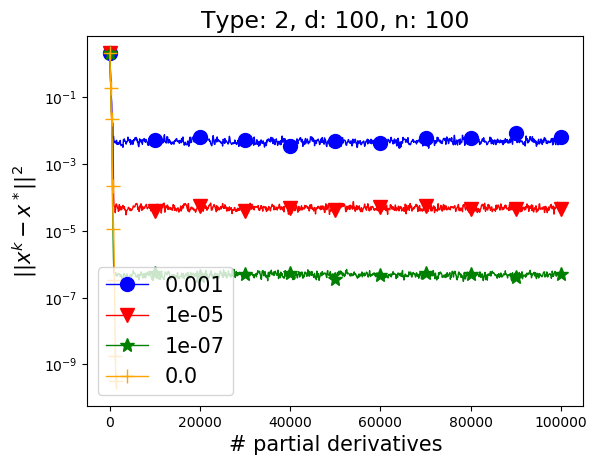}
\end{minipage}%
\begin{minipage}{0.24\textwidth}
  \centering
\includegraphics[width =  \textwidth ]{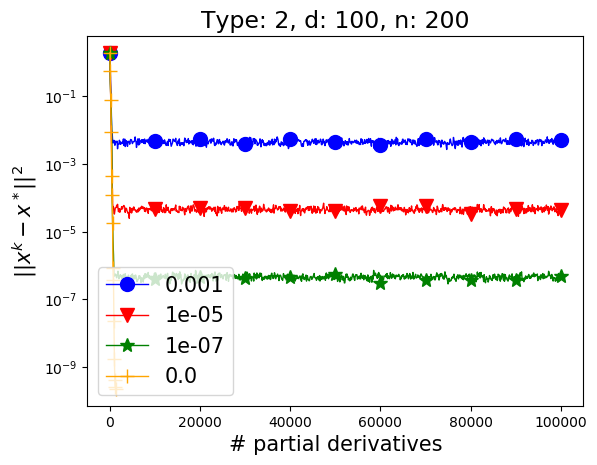}
\end{minipage}%
\\
\begin{minipage}{0.24\textwidth}
  \centering
\includegraphics[width =  \textwidth ]{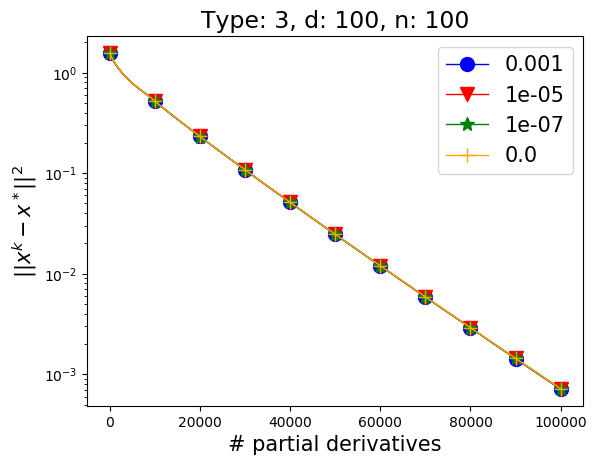}
\end{minipage}%
\begin{minipage}{0.24\textwidth}
  \centering
\includegraphics[width =  \textwidth ]{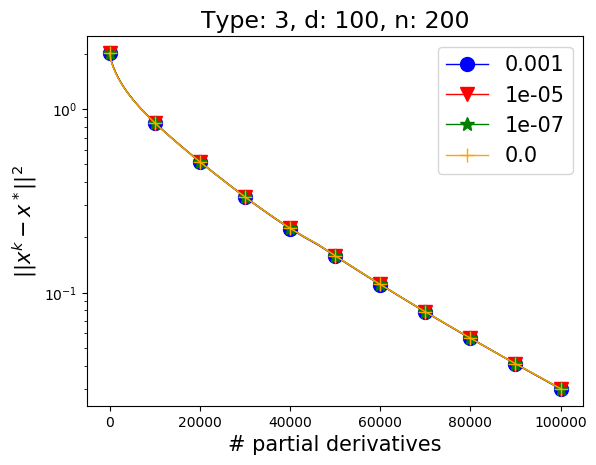}
\end{minipage}%
\begin{minipage}{0.24\textwidth}
  \centering
\includegraphics[width =  \textwidth ]{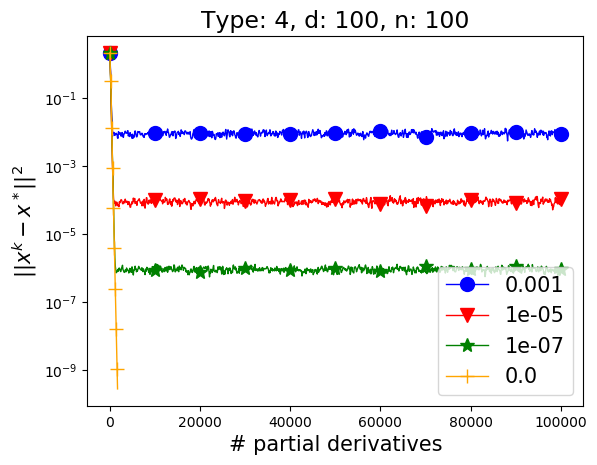}
\end{minipage}%
\begin{minipage}{0.24\textwidth}
  \centering
\includegraphics[width =  \textwidth ]{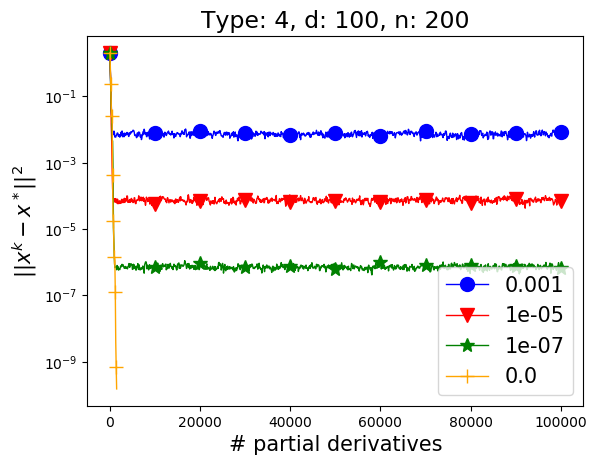}
\end{minipage}%
\caption{ {\tt N-SEGA} applied on constrained least squares problem with noised partial derivative oracle. Legend labels stand for the magnitude $\sigma^2$ of the oracle noise. }
\label{fig:nsega}
\end{figure}

\clearpage

\section{Proofs for Section~\ref{sec:main_res}}

\subsection{Basic Facts and Inequalities}\label{sec:basic_inequalities}

For all $a,b\in\R^d$ and $\xi > 0$ the following inequalities holds:
\begin{equation}\label{eq:fenchel}
    \langle a,b\rangle \le \frac{\norm{a}^2}{2\xi} + \frac{\xi\norm{b}^2}{2},
\end{equation}
\begin{equation}\label{eq:a_b_norm_squared}
    \norm{a+b}^2 \le 2\norm{a}^2 + 2\norm{b}^2,
\end{equation}
and
\begin{equation}\label{eq:1/2a_minus_b}
    \frac{1}{2}\norm{a}^2 - \norm{b}^2 \le \norm{a+b}^2.
\end{equation}

For a random vector $\xi \in \R^d$ and any $x\in \R^d$ the variance can be decomposed as
\begin{equation}\label{eq:variance_decomposition}
    \EE\left[\norm{\xi - \EE\xi}^2\right] = \EE\left[\norm{\xi-x}^2\right] - \EE\left[\norm{\EE\xi - x}^2\right] \;.
\end{equation}

\subsection{A Key Lemma}

The following lemma will be used in the proof of our main theorem.

\begin{lemma}[Key single iteration recurrence]  \label{lem:iter_dec}
Let Assumptions~\ref{as:general_stoch_gradient}~and~\ref{as:mu_strongly_quasi_convex}  be satisfied. Then the following inequality holds for all $k\geq 0$:
\begin{eqnarray*}
&& \EE\left[\norm{x^{k+1}-x^*}^2\right] + M\gamma^2\EE\left[\sigma_{k+1}^2\right]  + 2\gamma\left(1-\gamma(A+CM)\right)\EE\left[D_f(x^k,x^*)\right] \\ 
    && \qquad    \le
        (1-\gamma\mu)\EE \left[\norm{x^k - x^*}^2 \right] + \left(1 - \rho\right)M\gamma^2\EE\left[\sigma_k^2\right] + B \gamma^2\EE\left[\sigma_k^2\right] + (D_1+MD_2)\gamma^2.
\end{eqnarray*}
\end{lemma}
\begin{proof}

    We start with estimating the first term of the Lyapunov function. Let $r^k = x^k - x^*$. Then
    \begin{eqnarray*}
        \norm{r^{k+1}}^2 &=& \norm{ \prox_{\gamma R} (x^k- \gamma  g^k) - \prox_{\gamma R} ( x^* - \gamma\nabla f(x^*)) }^2 \\
         &\leq & 
          \norm{x^k- x^* - \gamma (  g^k - \nabla f(x^*)) }^2
          \\
          &= & 
           \norm{r^k}^2 - 2\gamma\langle r^k,g^k  - \nabla f(x^*)\rangle + \gamma^2\norm{ g^k  - \nabla f(x^*)}^2.
    \end{eqnarray*}
    Taking expectation conditioned on $x^k$ we get
    \begin{eqnarray*}
        \EE\left[\norm{r^{k+1}}^2\mid x^k\right] &=& \norm{r^k}^2 - 2\gamma\langle r^k,\nabla f(x^k)-\nabla f(x^*)\rangle + \gamma^2\EE\left[\norm{ g^k - \nabla f(x^*)}^2\mid x^k\right]
        \\
        &\overset{\eqref{eq:mu_strongly_quasi_convex}}{\le}& 
        (1-\gamma\mu)\norm{r^k}^2 - 2\gamma D_f(x^k,x^*) + \gamma^2\EE\left[\norm{g^k - \nabla f(x^*)}^2\mid x^k\right]
        \\
        &\overset{\eqref{eq:general_stoch_grad_unbias}+\eqref{eq:general_stoch_grad_second_moment}}{\le}& 
        (1-\gamma\mu)\norm{r^k}^2 + 2\gamma\left(A\gamma - 1\right)D_f(x^k,x^*) + B\gamma^2\sigma_k^2 + \gamma^2 D_1.
    \end{eqnarray*}
    Using this we estimate the full expectation of $V^{k+1}$ in the following way:
    \begin{eqnarray}
    && \EE\norm{x^{k+1}-x^*}^2 + M\gamma^2\EE\sigma_{k+1}^2\notag
        \\
        &\overset{\eqref{eq:gsg_sigma}}{\le}& (
        1-\gamma\mu)\EE\norm{x^k-x^*}^2 + 2\gamma\left(A\gamma - 1\right)D_f(x^k,x^*) + B\gamma^2\EE\sigma_k^2\notag
        \\
        &&\quad 
        + (1-\rho)M\gamma^2\EE\sigma_k^2  + 2CM\gamma^2\EE\left[D_f(x^k,x^*)\right] + (D_1+MD_2)\gamma^2\notag\\
        &=& (1-\gamma\mu)\EE\norm{x^k - x^*}^2 + \left(1 + \frac{B}{M} - \rho\right)M\gamma^2\EE\sigma_k^2\notag
        \\
        &&\quad
         + 2\gamma\left(\gamma(A+CM)-1\right)\EE\left[D_f(x^k,x^*)\right] + (D_1+MD_2)\gamma^2\notag \,.
        \label{eq:gsgm_recurrence}
    \end{eqnarray}
It remains to rearrange the terms.
\end{proof}

\subsection{Proof of Theorem~\ref{thm:main_gsgm}}

Note first that due to~\eqref{eq:gamma_condition_gsgm} we have $2\gamma\left(1-\gamma(A+CM)\right)\EE D_f(x^k,x^*)>0$, thus we can omit the term.

    Unrolling the recurrence from Lemma~\ref{lem:iter_dec} and using the Lyapunov function notation gives us
    \begin{eqnarray*}
        \EE V^{k} &\le& \max\left\{(1-\gamma\mu)^k,\left(1+\frac{B}{M}-\rho\right)^k\right\}V^0\\
        &&\quad + (D_1+MD_2)\gamma^2\sum\limits_{l=0}^{k-1}\max\left\{(1-\gamma\mu)^l,\left(1+\frac{B}{M}-\rho\right)^l\right\}\\
        &\le& \max\left\{(1-\gamma\mu)^k,\left(1+\frac{B}{M}-\rho\right)^k\right\}V^0\\
        &&\quad + (D_1+MD_2)\gamma^2\sum\limits_{l=0}^{\infty}\max\left\{(1-\gamma\mu)^l,\left(1+\frac{B}{M}-\rho\right)^l\right\}\\
        &\le& \max\left\{(1-\gamma\mu)^k,\left(1+\frac{B}{M}-\rho\right)^k\right\} V^0 + \frac{(D_1+MD_2)\gamma^2}{\min\left\{\gamma\mu, \rho - \frac{B}{M}\right\}}.
    \end{eqnarray*}

\end{document}